\documentclass{amsart}

\usepackage{amssymb}
\usepackage{amsthm}
\usepackage{amsmath}
\usepackage{enumitem}
\usepackage[mathscr]{euscript}

\usepackage{cancel}
\usepackage{soul}

\usepackage[normalem]{ulem}

\usepackage[foot]{amsaddr}

\usepackage[usenames]{xcolor}

\usepackage{hyperref} 

\usepackage{tikz-cd} 
 
\theoremstyle{plain}
\newtheorem{proposition}{Proposition}[section] 
\newtheorem{theorem}[proposition]{Theorem}
\newtheorem{lemma}[proposition]{Lemma}  
\newtheorem{corollary}[proposition]{Corollary}
\theoremstyle{definition}
\newtheorem{example}[proposition]{Example} 
\newtheorem{definition}[proposition]{Definition}
\newtheorem{observation}[proposition]{Observation}

\theoremstyle{remark}
\newtheorem{remark}[proposition]{Remark}

\DeclareMathOperator{\Aut}{\mathsf{Aut}}
\DeclareMathOperator{\Isom}{\mathsf{Isom}}

\DeclareMathOperator{\End}{End}

\DeclareMathOperator{\diam}{diam}

\DeclareMathOperator{\supp}{supp}

\DeclareMathOperator{\interior}{int}

\DeclareMathOperator{\id}{id}
\DeclareMathOperator{\rank}{rank}

\DeclareMathOperator{\dist}{d}

\DeclareMathOperator{\Gr}{Gr}

\DeclareMathOperator{\Ac}{\mathcal{A}}
\DeclareMathOperator{\Cc}{\mathcal{C}}
\DeclareMathOperator{\Dc}{\mathcal{D}}
\DeclareMathOperator{\Fc}{\mathcal{F}}
\DeclareMathOperator{\Gc}{\mathcal{G}}
\DeclareMathOperator{\Hc}{\mathcal{H}}

\DeclareMathOperator{\Lc}{\mathcal{L}}

\DeclareMathOperator{\Nc}{\mathcal{N}}
\DeclareMathOperator{\Oc}{\mathcal{O}}
\DeclareMathOperator{\Qc}{\mathcal{Q}}

\DeclareMathOperator{\Tc}{\mathcal{T}}
\DeclareMathOperator{\Uc}{\mathcal{U}}

\DeclareMathOperator{\Cb}{\mathbb{C}}

\DeclareMathOperator{\Hb}{\mathbb{H}}

\DeclareMathOperator{\Nb}{\mathbb{N}}
\DeclareMathOperator{\Pb}{\mathbb{P}}
\DeclareMathOperator{\Rb}{\mathbb{R}}

\DeclareMathOperator{\Zb}{\mathbb{Z}}

\DeclareMathOperator{\Gsf}{\mathsf{G}}
\DeclareMathOperator{\Psf}{\mathsf{P}}
\DeclareMathOperator{\PGL}{\mathsf{PGL}}
\DeclareMathOperator{\PSL}{\mathsf{PSL}}
\DeclareMathOperator{\PO}{\mathsf{PO}}
\DeclareMathOperator{\SL}{\mathsf{SL}}

\newcommand{\abs}[1]{\left|#1\right|}
\newcommand{\norm}[1]{\left\|#1\right\|}

\newcommand{\Ga}{\Gamma}
\newcommand{\ga}{\gamma}

\newcommand{\La}{\Lambda}

\newcommand{\Msf}{\mathsf{M}}
\newcommand{\Asf}{\mathsf{A}}

\newcommand{\Nsf}{\mathsf{N}}
\newcommand{\Ksf}{\mathsf{K}}

\newcommand{\fa}{\mathfrak{a}}
\newcommand{\mfa}{\mathfrak{a}}
\newcommand{\mfp}{\mathfrak{p}}
\newcommand{\mfg}{\mathfrak{g}}
\newcommand{\mfk}{\mathfrak{k}}

\newcommand{\Hsf}{\mathsf{H}}

\newcommand{\R}{\mathbb{R}}

\newcommand{\ba}{\backslash}

\newcommand{\opp}{\mathrm{i}}

\newcommand{\Mod}{\mathsf{Mod}}

\newcommand{\PML}{\mathcal{PML}}

\newcommand{\UE}{\mathcal{UE}}

\newcommand{\Prob}{\mathrm{Prob}}

\begin{document}

\title[Boundary maps and PS-measures for non-Borel Anosov groups]{Measurable boundary maps and Patterson--Sullivan measures for non-Borel Anosov groups on the Furstenberg boundary}

\author[Kim]{Dongryul M. Kim}
\email{dongryul.kim97@gmail.com}
\address{Simons Laufer Mathematical Sciences Institute, USA}

\author[Zimmer]{Andrew Zimmer}
\email{amzimmer2@wisc.edu}
\address{Department of Mathematics, University of Wisconsin-Madison, USA}
\date{\today}

 \subjclass[2020]{}

\begin{abstract} In this paper we develop a theory for Patterson--Sullivan measures for non-Borel Anosov groups on the Furstenberg boundary. Previously, such a theory has been successfully developed for measures supported on the partial flag manifold associated to the Anosov condition, which coincides with the Furstenberg boundary only under the strongest Anosov condition, Borel Anosov. We establish existence, uniqueness, and ergodicity results for the measures on the Furstenberg boundary under arbitrary Anosov conditions;  we show ergodicity  of Bowen--Margulis--Sullivan measures on the homogeneous space; and we establish strict convexity results for the critical exponent associated to functionals on the entire Cartan subspace. Using this strict convexity, we establish an entropy rigidity result for Anosov groups with Lipschitz limit set.

A key tool we develop is a new sufficient condition for the existence of a measurable boundary map associated to a Zariski dense representation. This result not only applies to Anosov groups, but also to transverse groups, mapping class groups, and discrete subgroups of the isometry groups of Gromov hyperbolic spaces.

\end{abstract}

\maketitle

\vspace{-2.5em}

\setcounter{tocdepth}{1}
\tableofcontents

\section{Introduction}

Throughout this paper $\Gsf$ will be a semisimple Lie group with finite center and no compact factors. We fix a Cartan decomposition $\mfg = \mfp + \mfk$ of the Lie algebra, a Cartan subspace  $\mfa \subset \mfp$, and a positive Weyl chamber $\mfa^+ \subset \mfa$. Let $\Delta \subset \mfa^*$ denote the system of simple restricted roots corresponding  to the choice of $\mfa^+$ and let $\kappa : \Gsf \rightarrow \mfa^+$ denote the Cartan projection. 

Given a non-empty subset $\theta \subset \Delta$, let $\Psf_\theta < \Gsf$ denote the associated parabolic subgroup. A discrete subgroup $\Gamma < \Gsf$ is \emph{$\Psf_\theta$-Anosov} if $\Gamma$ is word hyperbolic as an abstract group and its Gromov boundary nicely embeds into the partial flag manifold $\Fc_\theta : = \Gsf /\Psf_\theta$ (see Section \ref{section:Anosov transverse} for a precise definition). 

For each $\theta \subset \Delta$, there is a partial Cartan subspace $\mfa_\theta \subset \mfa$ and a cocycle $B_\theta^{IW} :~\Gsf \times \Fc_\theta \rightarrow \mfa_\theta$ called the \emph{(partial) Iwasawa cocycle}. This cocycle can be used to define Patterson--Sullivan measures as follows. 

\begin{definition} \label{defn:Quints definition}
Given a subgroup $\Gamma < \Gsf$, $\theta \subset \Delta$,  $\phi \in \mfa_\theta^*$, and $\delta \geq 0$, a Borel probability measure $\mu$ on $\Fc_\theta$ is a \emph{$(\Gamma, \phi, \delta)$-Patterson--Sullivan measure} if for every $\gamma \in \Gamma$ the measures $\mu$, $\gamma_*\mu$ are absolutely continuous and 
$$
\frac{d\gamma_* \mu}{d\mu}(x) = e^{-\delta \phi B_\theta^{IW}(\gamma^{-1}, x)} \quad \mu\text{-a.e.}
$$
\end{definition} 

When $\Gsf$ is of rank one, the above definition (with an appropriate choice of functional) coincides with the classical Patterson--Sullivan measures introduced by Patterson \cite{Patterson1976} and Sullivan \cite{Sullivan_density}. In higher rank, the above definition is  due to Quint \cite{Quint_PS}.

The theory of Patterson--Sullivan measures on $\Fc_\theta$ for $\Psf_\theta$-Anosov groups has been extensively developed and has become a useful tool for studying Anosov groups.  In this case, the definition of an Anosov group implies that a $\Psf_\theta$-Anosov group acts very nicely on $\Fc_{\theta} =  \Gsf / \Psf_{\theta}$, but in general the action of a $\Psf_\theta$-Anosov group  can be very complicated on larger flag manifolds, especially the Furstenberg boundary~$\Fc_\Delta$.  

In this paper we use and extend ideas from our earlier work~\cite{KimZimmer1} to develop a theory of Patterson--Sullivan measures on the Furstenberg boundary $\Fc_\Delta$ for Zariski dense $\Psf_\theta$-Anosov groups, even though $\theta \neq \Delta$.  Having such a Patterson--Sullivan theory on the Furstenberg boundary enables us to obtain several applications as discussed below, which was previously known only for $\Psf_{\Delta}$-Anosov groups (or, Borel Anosov groups). Recall that the $\Psf_{\Delta}$-Anosov condition is very restrictive; in many cases, the group must either be virtually a free group or a surface group~\cite{CT_restrictions_2020,Tsouvalas2020,Dey2025,DGZ_restrictions_2024}.

Our results hold for the more general class of transverse groups (which also contain the relatively Anosov groups), for any partial flag manifold containing $\Fc_\theta$,
and for weaker irreducibility conditions than Zariski dense. However, for simplicity \textbf{in the introduction we only state our results for Zariski dense Anosov groups} and Patterson--Sullivan measures on the Furstenberg boundary. 

Before stating our results for discrete subgroups of Lie groups, we describe one of the tools we develop. 

\subsection{Measurable boundary maps} 
The existence of boundary maps for representations of lattices played a significant role in the proof of Mostow's Rigidity \cite{Mostow_QCrigidity,Mostow_book,Prasad_rigidity} and Margulis' Superrigidity \cite{MargulisBook}. A key tool in this work is constructing measurable boundary maps associated to Zariski dense representations.

The domains in our boundary map theorem are ``Patterson--Sullivan systems,'' a notion introduced in our earlier work~\cite{KimZimmer1}. A precise definition is given in Section ~\ref{sec:PS systems} below, but informally these consist of a compact metric space $M$, a Borel probability measure $\mu$ on $M$, and an action $\Gamma \curvearrowright M$ which preserves the measure class of $\mu$ which satisfy certain assumptions that allowed us to extend the classical theory of Patterson--Sullivan measures. We note that a different framework for abstract Patterson--Sullivan-like measures was given in~\cite{BCZZ_coarse}.

Amongst ``Patterson--Sullivan systems,'' we further identified a special class called ``well-behaved Patterson--Sullivan systems'' and for these systems defined conical limit sets (again see Section ~\ref{sec:PS systems} for a precise definition). Under the assumption that the conical limit set has full measure, in this paper we prove the following existence theorem for boundary maps. 

\begin{theorem}[see Theorem~\ref{thm:boundary map} below]\label{thm:boundary map in intro} Suppose $(M, \Gamma, \sigma, \mu)$ is a well-behaved Patterson--Sullivan system with respect to a hierarchy $\mathscr{H} = \{ \mathscr{H} (R) \subset \Ga : R \ge 0\}$, the $\Hc$-conical limit set has full $\mu$-measure, and the $\Gamma$-action on $(M,\mu)$ is amenable. 
  
If  $\rho : \Gamma \rightarrow \Gsf$ is a Zariski dense representation, then there exists a unique $\rho$-equivariant $\mu$-a.e.\ defined measurable map $f : M \rightarrow \Fc_\Delta$. 
\end{theorem} 

\begin{remark} The above is actually a special case of Theorem~\ref{thm:boundary map}, which assumes weaker conditions than Zariski dense, provides the existence of $\mu$-a.e.\  conical limits, and shows that $f$ maps into the ``contracting conical limit set.''
\end{remark} 

We highlight three examples of well-behaved Patterson--Sullivan systems: 
\begin{enumerate} 
\item If $\Gamma < \Gsf$ is $\Psf_\theta$-Anosov and $\mu$ is a Patterson--Sullivan measure (in the sense of Definition~\ref{defn:Quints definition}) supported on the limit set of $\Gamma$ in $\Fc_\theta$, then $\Fc_\theta$, $\Gamma$,  $\mu$ are part of a well-behaved Patterson--Sullivan system where the conical limit set has full $\mu$-measure. In fact it suffices to assume $\Gamma$ is a $\Psf_\theta$-transverse group and the Poincar\'e series associated to $\mu$ diverges at the critical exponent, see Section \ref{sec:shadows and cc limit set}  for more details. When $\Ga$ is $\Psf_{\theta}$-Anosov, amenability follows from the work of Adams \cite{Adams_hypgps}. For $\Psf_{\theta}$-transverse groups, we show amenability in Section \ref{sec:amenable actions of transverse groups}. See Theorem \ref{thm:lifting properties} for more details.

\item Let $X$ be a proper Gromov hyperbolic metric space, let $\partial X$ denote the Gromov boundary of $X$, let $\Gamma < \Isom(X)$ be a discrete subgroup,  let $\delta_X(\Gamma)$ denote the critical exponent of $\Gamma$, and let $\mu$ be a Patterson--Sullivan measure of dimension $\delta_X(\Gamma)$. Then $\partial X$, $\Gamma$, $\mu$ are part of a well-behaved Patterson--Sullivan system. Further, if 
$$
\sum_{\gamma \in \Gamma} e^{-\delta_X(\Gamma) \dist(o,\gamma o) } = +\infty
$$
for some (any) $o \in X$, then the conical limit set has full $\mu$-measure. When $X$ has exponential bounded geometry,  Kaimanovich \cite{Kaimanovich2004} showed that the $\Gamma$-action on $(\partial X,\mu)$ is amenable. See 
Theorem \ref{thm:GH exp bdd} for more details. 

\item Let $S$ be a connected orientable surface of finite type with negative Euler characteristic, let $\Mod(S)$ denote its mapping class group, let $\PML$ denote the space of projective measured laminations on $S$, and let ${\rm Leb}$  denote the natural Lebesgue measure class on $\PML$. Then $\PML$, $\Mod(S)$, ${\rm Leb}$ are part of a well-behaved Patterson--Sullivan system where the conical limit set has full ${\rm Leb}$-measure. Further, the work of Hamenst\"adt \cite{Hamenstadt_amenable} implies that the  $\Mod(S)$-action on $(\PML, {\rm Leb})$ is amenable. See 
Theorem \ref{thm:PML boundary map body} for more details. 
\end{enumerate}

We note that Bader--Furman~\cite{BaderFurman2014,BaderFurman2025} have developed a different abstract setting that leads to the existence of boundary maps. Many examples seem to satisfy both setting, but there does not seem to be any obvious implication between the two abstract settings. Further some examples, like  Poisson boundaries of countable groups, satisfy the Bader--Furman assumptions but not ours. Other examples, like Kleinian groups of divergent type with infinite Bowen--Margulis--Sullivan (BMS) measure, satisfy our assumptions but do not seem to satisfy the Bader--Furman assumptions.

In this paper our main motivation is the study of $\Psf_\theta$-transverse groups. These have natural flow spaces which can have infinite  BMS-measures and thus it seems that they do not satisfy the Bader--Furman assumptions  (in fact any Kleinian group is a transverse group). 

One idea in the proof of Theorem~\ref{thm:boundary map in intro} is to adapt the ``standard argument'' for almost sure convergence of random walks on symmetric spaces (e.g.\ as in \cite[Chapter VI, Section 2]{MargulisBook} or \cite[Section 4.2]{BQ_book}) where we replace the use of the Martingale convergence theorem with a technical continuity result along ``conical sequences'' (see Section~\ref{sec:continuity properties of boundary maps}).

\subsection{Existence and uniqueness of PS-measures} For $\alpha \in \Delta$, let $\omega_\alpha$ denote the associated fundamental weight. The dual space of the  partial Cartan subspace $\mfa_\theta \subset \mfa$ can be identified with ${\rm span} \, \{ \omega_\alpha\}_{\alpha \in \theta}$. Then given $\phi \in \mfa_\theta^*:= {\rm span} \, \{ \omega_\alpha\}_{\alpha \in \theta}$, the \emph{$\phi$-critical exponent} of a discrete subgroup $\Gamma < \Gsf$ is the exponential growth rate
\begin{equation*}
\delta^\phi(\Gamma) : = \limsup_{T \rightarrow +\infty} \frac{1}{T} \log \# \left\{ \gamma \in \Gamma : \phi(\kappa(\gamma)) \leq T\right\} \in [0,+\infty],
\end{equation*}
equivalently 
\begin{equation}\label{eqn:series defn of critical exp}
\delta^{\phi}(\Ga) = \inf \left\{ s > 0 : \sum_{\ga \in \Ga} e^{-s \phi(\kappa(\ga))} < + \infty \right\} \in [0, + \infty].
\end{equation} 

We establish existence and uniqueness of Patterson--Sullivan (PS) measures on the Furstenberg boundary. Further, we show that the measures are supported on the ``contracting conical limit set'' which is a smaller subset than the usual conical limit set defined in terms of shadows in the symmetric space. See Section \ref{sec:shadows and cc limit set} for details.

\begin{theorem}[PS-measures on the Furstenberg boundary]
   \label{thm:Anosov existence uniqueness} Suppose $\Gamma < \Gsf$ is Zariski dense and $\Psf_\theta$-Anosov. If $\phi \in \mfa_\theta^*$ and $ \delta^\phi(\Gamma) < +\infty$, then there exists a unique $(\Gamma, \phi,  \delta^\phi(\Gamma))$-Patterson--Sullivan measure on the Furstenberg boundary $\Fc_\Delta$. Moreover, this unique Patterson--Sullivan measure is supported on the contracting conical limit set of $\Ga$ in $\Fc_{\Delta}$.
   
\end{theorem}

Previously, existence and uniqueness for Patterson--Sullivan measures for $\Psf_\theta$-Anosov groups were only established on the partial flag manifold $\Fc_\theta$. The limit set in $\Fc_\theta$ compactifies the group (in fact identifies with the Gromov boundary of the group) which makes the construction on $\Fc_\theta$ much more straightforward. 

We construct the Patterson--Sullivan measure on $\Fc_\Delta$ by using Theorem~\ref{thm:boundary map in intro} to construct a measurable section $\Fc_\theta \rightarrow \Fc_\Delta$ and then pushing forward the Patterson--Sullivan measure on $\Fc_\theta$.

\begin{remark} \label{remark:Quint work}
Quint \cite{Quint_PS} proved that for a general Zariski dense discrete subgroup $\Ga < \Gsf$, $\phi \in \fa^*$, and $\delta \ge 0$, a $(\Ga, \phi, \delta)$-Patterson--Sullivan measure on $\Fc_{\Delta}$ exists if $\delta \cdot \phi$ is tangent to the so-called growth indicator of $\Ga$ at the direction in the interior $\interior \fa^+$. We suspect that this condition can fail to hold for our most general existence theorem (see Theorem~\ref{thm:lifting properties}) and in particular when the flow space associated to the transverse group has infinite BMS-measure. Indeed, in general, it is a priori possible that a tangent direction is on the wall $\fa^+ \cap \fa_{\theta}$. See Figure \ref{fig:tangent direction} below.

\begin{figure}[h]
\begin{tikzpicture}[scale=0.9]
    \draw (-2, 0) -- (2, 0) -- (0, 3.4641016151) -- (-2, 0);
    \draw[very thick, red] (2, 0) -- (0, 3.4641016151);
    \draw[red] (2, 0) node[right] {$\fa^+ \cap \fa_{\theta}$};    
    \draw[white] (-2, 0) node[left] {$\fa^+ \cap \fa_{\theta}$};

    \draw[very thick, color=blue, fill=blue, opacity=0.3] (1.7, 0.5196152423) .. controls (-0.5, 0.5) .. (0.3, 2.9444863729) --  (1.7, 0.5196152423);

    \filldraw (1, 1.7320508076) circle(2pt);
    \draw (1, 1.7320508076) node[right] {tangent direction};
   \draw[white] (-1, 1.7320508076) node[left] {tangent direction};

    \draw[color=blue] (-0.5, 0.5) node {$\Lc(\Ga)$};

\end{tikzpicture}
\caption{Slice of $\fa^+$ along its unit sphere, when $\rank \Gsf = 3$ and $\# \theta = 2$. $\Lc(\Ga) \subset \fa^+$ denotes the asymptotic cone of $\kappa(\Ga)$, and there is a tangent direction for $\delta^{\phi}(\Ga) \cdot \phi$ on $\Lc(\Ga) \cap \fa_{\theta}$.} \label{fig:tangent direction}
\end{figure}
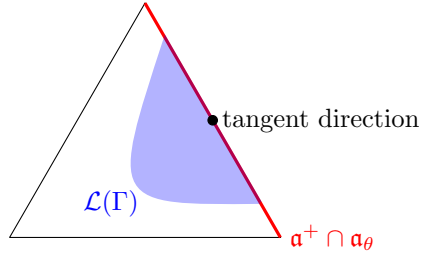
\end{remark}

\begin{remark}
   In particular, if $\Ga < \Gsf$ is Zariski dense and $\Psf_{\Delta}$-Anosov, then the existence of PS-measures on $\Fc_{\Delta}$ follows from Quint's work as in Remark \ref{remark:Quint work}. In this case, the uniqueness was proved by Edwards--Lee--Oh \cite{ELO_unique} when $\rank \Gsf \le 3$, and by Lee--Oh \cite{LO_Dichotomy} without rank assumption. On the other hand, for $\Psf_{\theta}$-Anosov case with $\theta \neq \Delta$, the existence and uniqueness of PS-measures on $\Fc_{\Delta}$ were not known, to the best of our knowledge.
\end{remark}

We also show that these Patterson--Sullivan measures are singular to every other Patterson--Sullivan measure. 

\begin{theorem}  
   \label{thm:singular introduction} 
   
   Suppose $\Gamma < \Gsf$ is Zariski dense and $\Psf_\theta$-Anosov. Assume 
\begin{itemize}
\item $\phi_1 \in \mfa_\theta^*$, $\delta^{\phi_1}(\Gamma) < +\infty$, and $\mu_1$ is the $(\Ga, \phi_1, \delta^{\phi_1}(\Ga))$-Patterson--Sullivan measure  on the Furstenberg boundary $\Fc_\Delta$. 
\item $\phi_2 \in \mfa^*$, $\delta \ge 0$, and $\mu_2$ is a $(\Ga, \phi_2, \delta)$--Patterson--Sullivan measure on the Furstenberg boundary $\Fc_\Delta$. 
\end{itemize}
Then, 
$$
\text{$\mu_1$ and $\mu_2$ are non-singular} \quad \Longleftrightarrow \quad \mu_1 = \mu_2 \quad \Longleftrightarrow \quad \delta^{\phi_1}(\Ga) \cdot \phi_1 = \delta \cdot \phi_2.
$$
\end{theorem} 

Notice that the functional $\phi_2$ is not assumed to be in $\mfa_\theta^*$. Previously, singularity results of this form were only established for Patterson--Sullivan measures associated to functionals in $\mfa_\theta^*$ on the partial flag manifold $\Fc_\theta$. The Anosov case was established in \cite{LO_Invariant,Sambarino_report}  and the transverse group case was established in \cite{BCZZ_coarse, Kim2024}.

\subsection{Double ergodicity and Bowen--Margulis--Sullivan measures}

Let $\Ksf < \Gsf$ denote the maximal compact subgroup with Lie algebra $\mfk$, let $\Asf< \Gsf$ denote the subgroup with Lie algebra $\mfa$, and let $\Msf < \Ksf$ denote the centralizer of $\Asf$ in $\Ksf$. 

When $\Gamma$ is Zariski dense and $\Psf_\Delta$-Anosov, the dynamics of the $\Asf$-action on the homogeneous space $\Gamma \ba \Gsf / \Msf$ has been extensively studied; see for instance \cite{CS_local,ELO_local, BLLO, Sambarino_report, LO_Dichotomy, KOW_PD, CZZ_relative,KOW_SF}.
 However, very little is known when $\Gamma$ is only $\Psf_\theta$-Anosov and $\theta \neq \Delta$. One reason for this is that to construct Bowen--Margulis--Sullivan measures, one needs to start with appropriate Patterson--Sullivan measures on the Furstenberg boundary which previously were only known to exist when $\Gamma$ is $\Psf_\Delta$-Anosov, or under extra hypotheses as in Remark~\ref{remark:Quint work}.

Using Theorem~\ref{thm:Anosov existence uniqueness} we can now construct Bowen--Margulis--Sullivan measures on $\Gamma \ba \Gsf / \Msf$ when $\Gamma < \Gsf$ is Zariski dense and $\Psf_\theta$-Anosov. We briefly describe the construction here, for more details see Section \ref{section:BMS}. 

Suppose $\Gamma < \Gsf$ is Zariski dense $\Psf_\theta$-Anosov and $\phi \in \mfa_\theta^*$ satisfies $\delta^\phi(\Gamma) < +\infty$. Let $\opp : \mfa \rightarrow \mfa$ denote the opposite involution. The adjoint $\opp^* : \mfa^* \rightarrow \mfa^*$ preserves the set of simple roots and we write $\opp^* \theta := \{ \opp^* \alpha : \alpha \in \theta\}$. Then $\opp^*\phi  \in \mfa_{\opp^* \theta}^*$ and $\Gamma$ is $\Psf_{\opp^* \theta}$-Anosov. 
Since 
$$
\opp^*\phi(\kappa(g)) =\phi(\kappa(g^{-1})) \quad \text{for all} \quad g \in \Gsf,  
$$
we have $\delta : = \delta^\phi(\Gamma) = \delta^{\opp^*\phi}(\Gamma)$. Then, by Theorem~\ref{thm:Anosov existence uniqueness} there exist unique $(\Ga, \phi, \delta)$ and $(\Ga, \opp^*\phi, \delta)$-Patterson--Sullivan measures $\mu_{\phi}$ and $ \mu_{\opp^*\phi}$ on $\Fc_\Delta$, respectively.

Next let $\Fc_\Delta^{(2)}$ denote the space of ordered transverse pairs and let $\Gc_\Delta : \Fc_\Delta^{(2)} \rightarrow \mfa$ denote the Gromov product (see Equation \eqref{eqn:Gromov product} for a definition). 
One can show that the measure 
$$
d\nu_\phi(x,y) : = e^{\delta \phi \Gc_\Delta(x,y)}d\mu_{\phi}(x) \otimes d\mu_{\opp^*\phi}(y) \quad \text{on} \quad \Fc_{\Delta}^{(2)}
$$
is $\Gamma$-invariant.  There is a natural equivariant identification of the  homogeneous space $\Gsf / \Msf$ with $\Fc_\Delta^{(2)} \times \mfa$ where the left multiplication of $\Gsf$ on $\Gsf / \Msf$ descends to the $\Gsf$-action on $\Fc_\Delta^{(2)} \times \mfa$  by 
$$
g \cdot (x,y, u) = (gx, gy, u+B_{\Delta}^{IW}(g, x)).
$$
Since $\Msf$ commutes with $\Asf$, we also have a right $\Asf$-action on $\Gsf / \Msf$, which corresponds to the $\fa$-action on $\Fc_{\Delta}^{(2)} \times \fa$ by translation on the $\fa$-component.
Further, if ${\rm Leb}_{\fa}$ denotes the Lebesgue measure on $\mfa$, then the measure
$
\tilde{\mathsf{m}}_\phi: = \nu_{\phi} \otimes {\rm Leb}_{\fa}
$
on $\Gsf / \Msf$ is $\Ga$-invariant, and hence it descends to a Radon measure
$$\mathsf{m}_\phi \quad \text{on} \quad \Ga \ba \Gsf / \Msf$$
called the \emph{Bowen--Margulis--Sullivan measure}, which is invariant under the $\Asf$-action.

\begin{theorem}[Ergodicity on homogeneous spaces]
    \label{thm:BMSergodic}
   With the notations above, 
\begin{enumerate}
\item The measure $\nu_{\phi}$ is non-zero and the $\Gamma$-action on $(\Fc_\Delta^{(2)}, \nu_\phi)$ is ergodic. 
\item The Bowen--Margulis--Sullivan measure $\mathsf{m}_{\phi}$ is non-zero and the $\Asf$-action on $(\Gamma \backslash \Gsf / \Msf, \mathsf{m}_\phi)$ is ergodic. 
\end{enumerate}
\end{theorem}

More generally, we obtain an ergodic dichotomy for transverse groups in Theorem \ref{thm:ergodic dichotomy}, generalizing the classical Hopf--Tsuji--Sullivan dichotomy \cite{Tsuji_potential,Hopf_ergodic,Sullivan_density,AS_rational,Roblin_ergodic}. Similar ergodicity results were previously known only for certain abstract flow spaces associated to $\theta \subset \Delta$ \cite{Sambarino_report, CZZ2024, KOW_PD}, which are not homogeneous spaces unless $\theta = \Delta$. When $\theta = \Delta$, the ergodic dichotomy for $\Psf_{\Delta}$-Anosov groups was first proved in \cite{LO_Dichotomy}.

To the best of our knowledge, Theorem \ref{thm:BMSergodic}  is the first ergodicity result on the homogeneous space for $\Psf_\theta$-Anosov groups when $\theta \neq \Delta$.

\begin{remark}
   One might also ask for the ergodicity as in Theorem \ref{thm:BMSergodic} when $\phi \notin \fa_{\theta}^*$. However, this is not true in general. Indeed, consider $\Gsf = \PSL(2, \Rb) \times \PSL(2, \Cb)$ and 
   $$
   \Ga := \left\{ (\rho_1(\ga), \rho_2(\ga)) \in \Gsf : \ga \in \pi_1(S) \right\}
   $$
   where $S$ is a closed surface of genus at least two, $\rho_1 : \pi_1(S) \to \PSL(2, \Rb)$ is a cocompact representation, and $\rho_2 : \pi_1(S) \to \PSL(2, \Cb)$ is a discrete faithful representation such that $\rho_2(\pi_1(S)) \ba \Hb^3$ has one geometrically finite end and one geometrically infinite end.

   In this case, we can assume that $\Delta = \{ \alpha_1, \alpha_2\}$  where $\alpha_1$ is the simple root for $\PSL(2, \Rb)$ and $\alpha_2$ is the one for $\PSL(2, \Cb)$, and $\Fc_{\Delta} = \partial \Hb^2 \times \partial \Hb^3$. In addition, $\Ga$ is Zariski dense and $\Psf_{\{\alpha_1\}}$-Anosov.

   On the other hand, if one produces a measure $\nu_{\alpha_2}$ on $\Fc_{\Delta}^{(2)}$ associated to $\alpha_2$, then $\nu_{\alpha_2}$ is not ergodic under the $\Ga$-action. This is because, if $\mu$ is a $(\Ga, \alpha_2, \delta)$-Patterson--Sullivan measure on $\Fc_{\Delta}$, then passing to the projection $ \pi : \Fc_{\Delta} \to \partial \Hb^3$, the measure $\pi_* \mu$ is a $(\rho_2(\pi_1(S)), \alpha_2, \delta)$-Patterson--Sullivan measure on $\partial \Hb^3$. Hence, the $\Ga$-ergodicity of $\nu_{\alpha_2}$ implies that the $\rho_2(\pi_1(S))$-action on $\partial \Hb^3 \times \partial \Hb^3 \smallsetminus \text{diag}$ with respect to the measure $\pi_* \mu \otimes \pi_* \mu$ is ergodic, which is impossible due to the work of Canary \cite{Canary1993ends}. A similar discussion holds for Bowen--Margulis--Sullivan measures.
\end{remark}

\begin{example}
   Consider the case that $\Gsf = \Gsf_1 \times \Gsf_2$ for some semisimple Lie groups $\Gsf_1$ and $\Gsf_2$. Then for any Anosov  subgroup $\Ga_1 < \Gsf_1$ and a representation $\rho :~\Ga_1 \to~\Gsf_2$, the subgroup
   $$
   \Ga := \{ (\ga, \rho(\ga)) : \ga \in \Ga_1 \} < \Gsf
   $$
   is $\Psf_{\theta}$-Anosov for some $\theta \subset \Delta$ consisting of simple roots for $\Gsf_1$. While $\rho$ can even be an indiscrete representation, Theorem \ref{thm:BMSergodic} applies to $\Ga \ba \Gsf / \Msf$ and implies the $\Asf$-ergodicity  when $\Ga$ is Zariski dense.

   A typical example is when $\Gsf_1 = \Gsf_2 = \PSL(2, \Rb)$, $\Ga_1 < \PSL(2, \Rb)$ discrete, and $\rho : \Ga_1 \to \PSL(2, \Rb)$ indiscrete. Such examples arise as Kobayashi geodesic curves in Hilbert modular varieties.
\end{example}

\subsection{Strict convexity of entropy} Using the theory developed here, we establish new strict convexity results for variations of critical exponent. 

Given a discrete subgroup $\Gamma < \Gsf$, one can show that the subset 
\begin{equation}\label{eqn:the set Q}
\Qc_\theta(\Gamma):=\{ \phi \in \mfa_\theta^* : \delta^\phi(\Gamma) \leq 1\}
\end{equation} 
is convex. Further, when $\Gamma$ is $\Psf_\theta$-Anosov the set
$
\Qc_\theta(\Gamma)
$
is strictly convex \cite[Corollary 5.9.1]{Sambarino_report} (see \cite[Corollary 13.2]{CZZ2024} for the $\Psf_{\theta}$-transverse case). In this paper, we will establish the following strict convexity result for $\Qc_\Delta(\Gamma)$. 

\begin{theorem}[Strict convexity in non-Anosov directions I] 
   \label{thm:extreme points} If $\Gamma < \Gsf$ is  Zariski dense and $\Psf_\theta$-Anosov, then 
$$
\{ \phi \in \mfa_\theta^* : \delta^\phi(\Gamma) = 1\}
$$
are extreme points of $\Qc_\Delta(\Gamma)$. 
\end{theorem} 

Theorem~\ref{thm:extreme points} is a consequence of the following estimate for critical exponent. 
We also establish a more general version for transverse groups (see Section \ref{section:strict convexity}).

\begin{theorem}[Strict convexity in non-Anosov directions II] 
   \label{thm:strict convexity intro}
   Suppose $\Gamma < \Gsf$ is a Zariski dense $\Psf_\theta$-Anosov group, $\phi \in \fa_\theta^*$, and $\delta^\phi(\Gamma)<+\infty$. If $\phi_1, \phi_2 \in \fa^*$ are linearly independent  
   with $\delta^{\phi_1}(\Ga), \delta^{\phi_2}(\Ga) < + \infty$ 
   and $\phi = c_1 \phi_1 + c_2 \phi_2$ for some $c_1,c_2 > 0$, then 
$$
\delta^\phi(\Gamma) < \frac{1}{ \frac{c_1}{\delta^{\phi_1}(\Gamma)} + \frac{c_2}{\delta^{\phi_2}(\Gamma)} }.
$$
\end{theorem} 

The proof of strict convexity is based on the following proof strategy: if the critical exponent does not drop under convex combination, then the associated Patterson--Sullivan measures should be non-singular, which should imply that the associated lengths are equal. Implementing this strategy has several parts: 
\begin{itemize}
\item Constructing Patterson--Sullivan measures. Prior to our work, it was known that there is a Patterson--Sullivan measure for $\phi$ on $\Fc_\theta$ and there was no guarantee that Patterson--Sullivan measures exist for $\phi_1,\phi_2$ (see Remark~\ref{remark:Quint work}). In this paper we show that there is a Patterson--Sullivan measure for $\phi$ on $\Fc_\Delta$ (see Theorem~\ref{thm:Anosov existence uniqueness}). In a separate paper \cite{KimZimmer_vector}, we show that there are Patterson--Sullivan measures for $\phi_1,\phi_2$ on a different boundary $\partial_\Delta X$ which contains $\Fc_\Delta$. The boundary $\partial_\Delta X$ is constructed as a vector-valued horofunction boundary of the symmetric space using Cartan projections, see Section~\ref{sec:compactifications and PS measures} for details. 

\item Showing non-singularity of Patterson--Sullivan measures when the critical exponent does not drop. This will be a consequence of Proposition~\ref{prop:uniquePS}. 
\item Showing non-singularity of Patterson--Sullivan measures implies equality of lengths. This was the main result of our earlier work~\cite{KimZimmer1} in the context of Patterson--Sullivan systems. 
\end{itemize}

\subsection{Entropy rigidity} As an application of our strict convexity results, we establish an entropy rigidity result for Anosov groups with Lipschitz limit sets.

In the following discussion, let $\mfa = \{ {\rm diag}(t_1,\dots, t_d) : t_1 + \cdots + t_d = 0\}$ denote the standard Cartan subspace for $\SL(d,\Rb)$ and let $\Delta =\{ \alpha_1,\dots, \alpha_{d-1}\}$ denote the standard simple roots defined by 
$$
\alpha_j ({\rm diag}(t_1,\dots, t_d))=t_j-t_{j+1}.
$$
For notation simplicity let $\Psf_k : = \Psf_{\{\alpha_k\}}$, $\Fc_k := \Fc_{\{\alpha_k\}}$, and $\Lambda_k( \Gamma) : = \Lambda_{\alpha_k}(\Gamma)$. Notice that we have a natural identification $\Fc_1 = \Pb(\Rb^d)$. 

The \emph{Hilbert functional} $\phi_{\rm H} \in \mfa^*$ is 
$$
\phi_{\rm H}( {\rm diag}(t_1,\dots, t_d)) = \frac{1}{2}(t_1 - t_d)
$$
and the \emph{Hilbert critical exponent} of a discrete subgroup $\Gamma < \SL(d,\Rb)$ is the critical exponent $\delta^{\phi_{\rm H}}(\Gamma)$ associated to $\phi_{\rm H}$. The motivation for this terminology comes from the fact that if $\Gamma$ preserves a properly convex domain $\Omega \subset \Pb(\Rb^d)$, then $\delta^{\phi_{\rm H}}(\Gamma)$ coincides with the critical exponent of $\Gamma$ with respect to the Hilbert metric on $\Omega$ (this follows from \cite[Proposition 10.1]{DGK_cc}).

For a $\Psf_1$-Anosov group $\Ga < \SL(d, \Rb)$ with Lipschitz limit set $\La_1(\Ga) \subset \Pb(\Rb^d)$, Pozzetti--Sambarino--Wienhard established the following upper bound for the Hilbert entropy.

\begin{theorem}[{\cite[consequence of Theorem A]{PSW2023}}] Suppose $\Gamma < \SL(d,\Rb)$ is a $\Psf_1$-Anosov group acting strongly irreducibly on $\Rb^d$ and  on $\wedge^{p+1} \Rb^d$ for some $p \le d-2$. If $\Lambda_1(\Gamma)$ is a Lipschitz $p$-manifold, then 
$$
\delta^{\phi_{\rm H}}(\Gamma) \leq p.
$$
\end{theorem}

Using Theorem \ref{thm:strict convexity intro} we prove rigidity in the equality case. 

\begin{theorem}[Entropy rigidity] 
   \label{thm:Lipschitz limit set rigidity intro}
   Suppose $\Gamma < \SL(d,\Rb)$ is a $\Psf_1$-Anosov group acting strongly irreducibly on $\Rb^d$ and on $\wedge^{p+1} \Rb^d$ whose limit set $\Lambda_1(\Gamma)$ is a Lipschitz $p$-manifold for some $p \le d-2$. Then
$$
\delta^{\phi_{\rm H}}(\Gamma) = p
$$
if and only if $\Gamma$ is conjugate to a uniform lattice in $\mathsf{SO}(d-1, 1)$ and $p=d-2$. 
\end{theorem} 

Previously Pozzetti--Sambarino--Wienhard proved a variant of this rigidity result when $\Gamma$ is also $\Psf_{p+1}$-Anosov and  $\Lambda_1(\Gamma)$ is a $\Cc^1$-smooth $p$-manifold ~\cite[Proposition 7.7]{PSW2023}. In this case, the $\Psf_{p+1}$-Anosov assumption implies that the set $\Qc_{\{\alpha_1,\alpha_{p+1}\}}(\Gamma)$ in Equation~\eqref{eqn:the set Q} is strictly convex and this strict convexity is essential in their proof. Using our ``strict convexity in non-Anosov directions'' theorem, we are able to extend their argument to the non-$\Psf_{p+1}$-Anosov case. 

We also note that Theorem~\ref{thm:Lipschitz limit set rigidity intro} generalizes a result of Crampon~\cite{Crampon2009} who proved that if $\Gamma < \SL(d,\Rb)$ is a discrete group which acts cocompactly on a strictly convex domain $\Omega \subset \Pb(\Rb^d)$, then 
$$
\delta_{\rm Hil}(\Gamma) \leq d-2
$$
with equality if and only if $\Omega$ is an ellipsoid (and hence $\Gamma$ is conjugate to a uniform lattice in $\mathsf{SO}(d-1, 1)$). Under Crampon's hypothesis, $\Gamma$ is $\Psf_1$-Anosov and acts irreducibly on $\Rb^d$ and hence also $\wedge^{d-1} \Rb^d$, see~\cite[Section 6.2]{GW2012}. Further, $\Lambda_1(\Gamma)$ coincides with $\partial\Omega$ and is hence a Lipschitz $(d-2)$-manifold. So Theorem~\ref{thm:Lipschitz limit set rigidity intro} does indeed generalize Crampon's result.

\subsection{Outline of Paper} 

The first part of the paper is expository. In Section~\ref{sec:PS systems}, we recall  the definition and some properties of abstract Patterson--Sullivan systems, which were introduced in our earlier work \cite{KimZimmer1}. In Section~\ref{sec: notation for ss groups},  we fix the notation involving semisimple Lie groups that we will use throughout the paper. 
 In Section~\ref{sec:compactifications and PS measures}, we recall the definition and some properties of vector-valued horofunction boundaries of symmetric spaces. 

In the second part of the paper, we prove the existence of boundary maps for Zariski dense representations of groups in a Patterson--Sullivan system. In Section~\ref{sec:continuity properties of boundary maps}, we prove a technical continuity result for measurable maps into separable metric spaces whose domain is well-behaved Patterson--Sullivan systems. 
In Section~\ref{sec: existence of boundary maps}, we establish the existence of a boundary map associated to a Zariski dense representation of a group which is part of a well-behaved Patterson--Sullivan system.
Our boundary map existence result requires that the group action in the Patterson--Sullivan system is amenable and in Section~\ref{sec:amenable actions of transverse groups} we verify this for transverse groups.

In the final part of the paper, we apply the previous two parts to study Patterson--Sullivan measures for transverse groups, mapping class groups, and discrete subgroups of isometry groups of Gromov hyperbolic spaces. In Section \ref{section:lifting map}, we establish existence and uniqueness results for Patterson--Sullivan measures on $\Fc_\Theta$ associated to $\Psf_\theta$-transverse groups when $\Theta \supset \theta$. In Section \ref{section:BMS}, we construct Bowen--Margulis--Sullivans measures on the homogeneous space $\Gamma \ba \Gsf / \Msf$ when $\Gamma$ is a Zariski dense $\Psf_\theta$-transverse group and prove a version of the  Hopf--Tsuji--Sullivan dichotomy for such measures. 
In Section \ref{section:strict convexity}, we prove the strict convexity of critical exponents in non-Anosov directions. In Section \ref{section:Lipschitz}, we prove the entropy rigidity for Anosov subgroups with Lipschitz limit sets. In Section \ref{section:MCG}, we consider representations of mapping class groups and obtain measurable boundary maps from $\PML$ equipped with the Lebesgue measure class. In Section \ref{section:GH exp bdd}, we consider representations of discrete isometry groups of Gromov hyperbolic spaces with exponentially bounded geometry and also obtain measurable boundary maps from their Gromov boundaries.
Finally, in Section \ref{section:deductions}, we explain how to establish the statements in the introduction from the results in the body of the paper.

\subsection*{Acknowledgements} Kim thanks the University of Wisconsin--Madison for hospitality during a visit in October 2025.
Zimmer was partially supported by a Sloan research fellowship and grants DMS-2105580 and
DMS-2452068 from the National Science Foundation. 
This material is based upon work supported by the National Science Foundation under Grant No. DMS-2424139, while the authors were in residence at the Simons Laufer Mathematical Sciences Institute in Berkeley, California, during the Spring 2026 semester.

\part{Background} 


\section{Patterson--Sullivan systems}\label{sec:PS systems} 


In this section, we recall the definition and some properties of abstract Patterson--Sullivan systems, which were introduced in our earlier work \cite{KimZimmer1}. The main idea in this previous work was to identify the key features of a group action on a probability space that allows one to extend the theory of Patterson--Sullivan measures. We note that a different framework for abstract Patterson--Sullivan-like measures was given in~\cite{BCZZ_coarse}.

Given a  compact metric space $M$, a subgroup $\Gamma < \mathsf{Homeo}(M)$, and $\kappa \geq 0$, a function $\sigma : \Ga \times M \to \mathbb{R}$ is called a \emph{$\kappa$-coarse-cocycle} if 
$$
   \left| \sigma(\ga_1 \ga_2, x) - \left( \sigma(\ga_1, \ga_2 x) + \sigma(\ga_2, x) \right) \right| \le \kappa
$$
  for any $\ga_1, \ga_2 \in \Ga$ and $x \in M$.
Given such a coarse-cocycle and $\delta \ge 0$, a Borel probability measure $\mu$ on $M$ is called \emph{coarse $(\Ga, \sigma,\delta)$-Patterson--Sullivan measure} if there exists $C\geq1$ such that for any $\ga \in \Ga$ the measures $\mu, \gamma_*\mu$ are absolutely continuous and 
   \begin{equation} \label{eqn.coarse PS meaasure intro}
   C^{-1}e^{ - \delta \sigma(\ga^{-1}, x)} \le \frac{d \ga_* \mu}{d \mu}(x) \le Ce^{- \delta \sigma(\ga^{-1}, x)} \quad \text{for } \mu\text{-a.e. } x \in M.
   \end{equation} 
When $C = 1$ and hence equality holds in Equation \eqref{eqn.coarse PS meaasure intro}, we call $\mu$ a \emph{$(\sigma,\delta)$-Patterson--Sullivan measure}.

Now we recall the definition of Patterson--Sullivan systems.

\begin{definition}\label{defn:PS systems}
A \emph{Patterson--Sullivan-system (PS-system) of dimension $\delta \ge 0$} consists of
\begin{itemize}
\item a coarse-cocycle $\sigma : \Gamma \times M \rightarrow \Rb$, 
\item a coarse $(\sigma,\delta)$-Patterson--Sullivan measure (PS-measure) $\mu$,
\item for each $\gamma \in \Gamma$, a number $\norm{\gamma}_\sigma \in \R$ called the \emph{$\sigma$-magnitude of $\gamma$}, and
\item for each $\gamma \in \Gamma$ and $R > 0$, a non-empty open set $\Oc_R(\gamma) \subset M$ called the \emph{$R$-shadow of $\gamma$}
\end{itemize}
such that:
\begin{enumerate}[label=(PS\arabic*)]
\item\label{item:coycles are bounded} For any $\ga \in \Ga$, there exists $c=c(\gamma) > 0$ such that $\abs{\sigma(\ga, x)}\leq  c(\gamma)$ for  $\mu$-a.e.\ $x\in M$. 
\item\label{item:almost constant on shadows} For every $R> 0$ there is a constant $C=C(R) > 0$ such that
$$
 \norm{\gamma}_\sigma - C  \leq \sigma(\gamma,x) \leq \norm{\gamma}_\sigma + C 
$$
for all $\gamma \in \Gamma$ and $\mu$-a.e.\ $x \in \gamma^{-1} \Oc_R(\gamma)$.
\item\label{item:empty Z intersection} If $\{\gamma_n\} \subset \Gamma$, $R_n \rightarrow +\infty$, $Z \subset M$ is compact, and $[M \smallsetminus \gamma_n^{-1}\Oc_{R_n}(\gamma_n)] \rightarrow Z$ with respect to the  Hausdorff distance, then for any $x \in Z$, there exists $g \in \Ga$ such that 
$$gx \notin Z.$$
\end{enumerate} 

We call the PS-system \emph{well-behaved} with respect to a collection
$$
\mathscr{H} := \{ \mathscr{H}(R) \subset \Ga : R \ge 0 \}
$$
 of non-increasing subsets of $\Ga$ if the following additional properties hold:
\begin{enumerate}[label=(PS\arabic*)]
\setcounter{enumi}{3}
\item\label{item:properness} $\Gamma$ is countable and for any $T > 0$, the set $\{ \gamma \in \mathscr{H}(0) : \norm{\gamma}_\sigma \leq T\}$ is finite. 

\item\label{item:baire} If $\{\gamma_n\} \subset \Gamma$, $R_n \rightarrow +\infty$, $Z \subset M$ is compact, and $[M \smallsetminus \gamma_n^{-1}\Oc_{R_n}(\gamma_n)] \rightarrow Z$ with respect to the  Hausdorff distance, then for any $h_1, \ldots, h_m \in \Ga$ and $x \in Z$, there exists $g \in \Ga$ such that
$$
g x \notin \bigcup_{i = 1}^m h_i Z.
$$

\item\label{item:shadow inclusion} If $R_1 \leq R_2$ and $\gamma \in \mathscr{H}(0)$, then $\Oc_{R_1}(\gamma) \subset \Oc_{R_2}(\gamma)$. 

\item\label{item:intersecting shadows} For any $R > 0$ there exist $C>0$ and $R'> 0$ such that: if $\alpha, \beta \in \mathscr{H}(R)$, $\norm{\alpha}_\sigma \leq \norm{\beta}_\sigma$, and $\Oc_R(\alpha) \cap \Oc_R(\beta) \neq \emptyset$, then 
$$
\Oc_R(\beta) \subset \Oc_{R'}(\alpha)
$$ 
and  
$$
\abs{\norm{\beta}_\sigma - (\norm{\alpha}_\sigma + \norm{\alpha^{-1}\beta}_\sigma)} \leq C. 
$$
\item \label{item:diam goes to zero ae} For every $R > 0$, there exists a set $M' \subset M$ of full $\mu$-measure such that 
$$
\lim_{n \rightarrow + \infty} {\rm diam} \Oc_R(\gamma_n) = 0
$$
whenever $\{\gamma_n\} \subset \mathscr{H}(R)$ is an escaping sequence and 
$$
x \in M' \cap  \bigcap_{n \ge 1} \Oc_R(\gamma_n).
$$
\end{enumerate} 
We call the collection $\mathscr{H}$ the \emph{hierarchy} of the Patterson--Sullivan system. 
\end{definition}

For well-behaved PS-systems, there is a natural notion of conical limit set.

\begin{definition} \label{defn:conical limit set}
   Let $(M, \Ga, \sigma, \mu)$ be a PS-system. 
\begin{itemize}
\item Given a subset $H \subset \Gamma$ and $R > 0$, let $\Lambda_R(H) \subset M$ be the set of points $x \in M$ where there exists an escaping sequence $\{\gamma_n\} \subset H$  such that 
$$
x \in \bigcap_{n \geq 1} \Oc_{R}(\gamma_n).
$$
\item If $(M, \Ga, \sigma, \mu)$ is well-behaved with respect to a hierarchy $\mathscr{H} = \{ \mathscr{H}(R) \subset \Ga : R \ge 0\}$, then the \emph{$\mathscr{H}$-conical limit set} is
\begin{equation} \label{eqn:conical}
\La^{\rm con}(\mathscr{H}) := \Ga \cdot \bigcup_{R > 0} \bigcap_{n \ge 1} \Lambda_R(\mathscr{H}(n)).
\end{equation}
\end{itemize} 
\end{definition}

 
 \subsection{Boundary rigidity} \label{section:boundary rigidity}


 One of the primary aims for developing the theory of PS-systems in \cite{KimZimmer1} was to establish a general setting where Tukia's measurable boundary rigidity theorem~\cite{Tukia1989} holds.

\begin{theorem}[Boundary rigidity, {\cite[Theorem 1.28]{KimZimmer1}}] \label{thm:KZ Rigidity}
 Suppose 
\begin{itemize} 
\item $(M_1, \Gamma_1, \sigma_1, \mu_1)$ is a well-behaved PS-system of dimension $\delta_1$ with respect to a hierarchy $\mathscr{H}_1 = \{ \mathscr{H}_1(R) \subset \Ga_1  : R \ge 0 \}$ and
$$\mu_1(\La^{\rm con}(\mathscr{H}_1)) = 1.$$
\item $(M_2, \Gamma_2, \sigma_2, \mu_2)$ is a  PS-system of dimension $\delta_2$.
\item There exists an onto homomorphism $\rho : \Gamma_1 \rightarrow \Gamma_2$ and a $\mu_1$-a.e.\ defined measurable $\rho$-equivariant injective map $f : M_1 \rightarrow M_2$.
\end{itemize} 
If the measures $f_*\mu_1$ and $\mu_2$ are not singular, then 
$$
\sup_{\gamma \in \Gamma_1} \abs{\delta_1 \norm{\gamma}_{\sigma_1} - \delta_2 \norm{\rho(\gamma)}_{\sigma_2}} < +\infty.
$$
\end{theorem}

\subsection{Properties of Patterson--Sullivan systems}

We now recall some useful properties of Patterson--Sullivan systems proved in \cite{KimZimmer1}, which were key ingredients in the proof of Theorem \ref{thm:KZ Rigidity}. They will also be used in this paper.

We begin with the ergodic property of a well-behaved Patterson--Sullivan system.

 \begin{theorem}[{\cite[Corollary 5.2]{KimZimmer1}}] \label{thm:ergodicity KZ}
   
If $(M, \Ga, \sigma, \mu)$ is a well-behaved PS-system with respect to a hierarchy $ \mathscr{H}$ and $\mu(\La^{\rm con}(\mathscr{H}))=1$, then the $\Ga$-action on $(M, \mu)$ is ergodic.
 \end{theorem}

The ergodicity as in Theorem \ref{thm:ergodic dichotomy} is based on the study of shadows and in particular a version of the Shadow Lemma. 

\begin{proposition}[Shadow Lemma, {\cite[Proposition 3.1]{KimZimmer1}}] \label{prop.shadowlemma} Let $(M, \Ga, \sigma, \mu)$ be a PS-system of dimension $\delta \ge 0$. For any $R > 0$ sufficiently large there exists $C=C(R) > 1$ such that 
$$
\frac{1}{C} e^{-\delta \norm{\gamma}_\sigma} \leq \mu( \Oc_R(\gamma)) \leq C e^{-\delta \norm{\gamma}_\sigma} 
$$
for all $\ga \in \Ga$.
\end{proposition}

When the underlying PS-system is well-behaved, shadows also satisfy the following version of the Vitali covering lemma. 

\begin{lemma}[Vitali Covering, {\cite[Lemma 3.2]{KimZimmer1}}] \label{lem:V covering} Let $(M, \Ga, \sigma, \mu)$ be a well-behaved PS-system with respect to a hierarchy $\mathscr{H} = \{ \mathscr{H}(R) \subset \Ga : R \ge 0 \}$. Fix $R > 0$ and let $R' > 0$ be the constant satisfying Property \ref{item:intersecting shadows} for $R$. Then  for any $I \subset \mathscr{H}(R)$, there exists $J \subset I$ such that 
$$
\bigcup_{\gamma \in I} \Oc_R(\gamma) \subset \bigcup_{\gamma \in J} \Oc_{R'}(\gamma)
$$
and the shadows $\{ \Oc_R(\gamma) : \gamma \in J\}$ are pairwise disjoint. 
\end{lemma}


\section{Notations for semisimple Lie groups}\label{sec: notation for ss groups}


In this section we fix the notation involving semisimple Lie groups that we will use throughout the paper. Of particular importance for our arguments are the linear representations fixed in Section~\ref{sec:linear representations}. 

Recall from the introduction that $\Gsf$ is a semisimple Lie group with finite center and no compact factors, $\mfg = \mfp + \mfk$ is a fixed Cartan decomposition of the Lie algebra, $\mfa \subset \mfp$ is a fixed Cartan subspace, and $\mfa^+ \subset \mfa$ is a fixed positive Weyl chamber. We use $\Sigma \subset \mfa^*$ to denote the set of restricted roots and use $\Delta \subset \mfa^*$ to denote the system of simple restricted roots corresponding  to the choice of $\mfa^+$. Then 
$$
\mfg = \mfg_0 \oplus \bigoplus_{\alpha \in \Sigma} \mfg_\alpha
$$
where 
$$
\mfg_\alpha = \{ X \in \mfg : [H,X] = \alpha(H)X \text{ for all } H \in \mfa\}.
$$
Let $\Sigma^+$ (resp. $\Sigma^-$) denote the restricted roots which are non-negative (respectively non-positive) linear combinations of elements of $\Delta$. 

\subsection{Cartan and Jordan projections} Let $\Ksf < \Gsf$ denote the maximal compact subgroup with Lie algebra $\mfk$. Recall that every $g \in \Gsf$ has a \emph{Cartan decomposition}, i.e. $g = k e^{\kappa(g)} \ell$ for some $k, \ell \in \Ksf$ and a unique $\kappa(g) \in \mfa^+$ (the elements $k, \ell$ are not unique). The map  $\kappa : \Gsf \rightarrow \mfa^+$ is called the \emph{Cartan projection}. The \emph{Jordan projection} $\lambda : \Gsf \to \fa^+$ is defined as 
$$
\lambda(g) := \lim_{n \to  + \infty} \frac{ \kappa(g^n)}{n}.
$$

We fix a representative $w_0 \in \Ksf$ of the longest Weyl element which is of order $2$. Let $\opp := - {\rm Ad}_{w_0} : \fa \to \fa$ denote the \emph{opposition involution}. This map has the property that 
\begin{equation}\label{eqn:opp inv and cartan proj}
\kappa(g^{-1}) = \opp \kappa(g) \quad \text{for all} \quad g \in \Gsf. 
\end{equation} 
The adjoint $\opp^* : \fa^* \to \fa^*$ of the opposite involution preserves the set of simple roots and for a subset $\theta \subset \Delta$, we define 
$$
\opp^* \theta := \{ \opp^*\alpha : \alpha \in \theta \}.
$$

\subsection{Parabolic subgroups and flag manifolds} Given a non-empty $\theta \subset \Delta$, the associated parabolic subgroup $\Psf_\theta$ is the stabilizer under the adjoint action of the Lie algebra 
$$
\mathfrak{u}_\theta^+ : = \bigoplus_{\alpha \in \Sigma_\theta^+} \mfg_\alpha
$$
where $\Sigma_\theta^+ : = \Sigma^+ \smallsetminus {\rm span}(\Delta \smallsetminus \theta)$. We also set
$ \Asf := \exp \fa$ and $\Asf^+ := \exp \fa^+$, and denote by $\Nsf < \Psf_{\Delta}$ the unipotent radical of $\Psf_{\Delta}$.

The \emph{Furstenberg boundary} and general \emph{$\theta$-boundary} are the quotient spaces 
$$
\Fc_{\Delta} := \Gsf / \Psf_{\Delta} \quad \text{and} \quad \Fc_{\theta} := \Gsf / \Psf_{\theta}.
$$
We also call $\Fc_{\theta}$ a \emph{partial flag manifold}.
Two elements $x \in \Fc_\theta$ and $y \in \Fc_{\opp^* \theta}$ are \emph{transverse} if there exists $g \in \Gsf$ such that 
$$
x = g \Psf_\theta \quad \text{and} \quad y = g w_0 \Psf_{\opp^* \theta},
$$
equivalently $(x,y)$ is contained in the unique open $\Gsf$-orbit in $\Fc_\theta \times \Fc_{\opp^* \theta}$.

\subsection{The special linear group}\label{sec:sl notation} In this section, we fix some notation when $\Gsf= \SL(V)$ and $V$ is a finite dimensional real vector space endowed with an inner product. 

For a linear tranformation $T : V \rightarrow V$, we let 
$$
\sigma_1(T)  \geq \cdots \geq \sigma_d(T)
$$
denote the singular values of $T$ with respect to the inner product and let $\norm{T} := \sigma_1(T)$ denote the operator norm. We assume that the fixed maximal compact $\Ksf$ of $\SL(V)$ coincides with the orthogonal group of the fixed inner product. Then we can fix a Cartan subspace $\mfa$, a positive Weyl chamber $\mfa^+$, and a labelling $\Delta = \{ \alpha_1,\dots, \alpha_{\dim V-1}\}$ of the simple roots such that 
$$
\alpha_j(\kappa(g)) = \log \frac{\sigma_j(g)}{\sigma_{j+1}(g)}
$$
for each $1 \leq j \leq \dim V-1$. Then the associated fundamental weights satisfy 
$$
\omega_{\alpha_j}( \kappa(g)) = \log \big( \sigma_1(g) \cdots \sigma_j(g)\big).
$$
For notation simplicity we let $\Psf_k : = \Psf_{\{\alpha_k\}}$, $\Fc_k := \Fc_{\{\alpha_k\}}$, and $\Lambda_k( \Gamma) : = \Lambda_{\alpha_k}(\Gamma)$. 

When $V = \Rb^d$, we always use the standard inner product, the standard Cartan subspace 
$$
\mfa = \{ {\rm diag}(t_1,\dots, t_d) : t_1 + \cdots + t_d = 0\},
$$
and the standard positive Weyl chamber
$$
\mfa^+ = \{ {\rm diag}(t_1,\dots, t_d) \in \mfa : t_1 \geq \cdots \geq t_d \}.
$$
Then 
$$
\alpha_j ({\rm diag}(t_1,\dots, t_d))=t_j-t_{j+1} \quad \text{and} \quad \omega_{\alpha_j} ({\rm diag}(t_1,\dots, t_d))=t_1+\cdots + t_j. 
$$

\subsection{Projection to the flag manifold}\label{sec:proj to flag manifolds}

For $g \in \Gsf$ with $\min_{\alpha \in \theta} \alpha(\kappa(g)) >~0$, we define 
$$
U_{\theta}(g) := k \Psf_\theta \in \Fc_{\theta}
$$
where $g$ has Cartan decomposition $g = k a \ell \in \Ksf \Asf^+ \Ksf$ (the condition on the roots implies that $U_{\theta}(g)$ is well-defined). By Equation~\eqref{eqn:opp inv and cartan proj}, 
$$
\min_{\alpha \in \theta} \alpha(\kappa(g))  = \min_{\alpha \in \opp^*\theta} \alpha(\kappa(g^{-1}))
$$
and so $U_{\opp^*\theta}(g^{-1})$ is also well-defined when  $\min_{\alpha \in \theta} \alpha(\kappa(g)) >~0$.

These projection maps have the following dynamical behavior (for a proof see for instance \cite[Section 4]{KLP_Anosov} or \cite[Proposition 2.3]{CZZ2024}). 

\begin{proposition}\label{prop:NS dynamics} If $\{g_n\} \subset \Gsf$, $x^+ \in \Fc_\theta$, and $x^- \in \Fc_{\opp^* \theta}$, then the following are equivalent: 
\begin{enumerate}
\item $g_n x \rightarrow x^+$ for all $x \in \Fc_\theta$ transverse to $x^-$ and the convergence is uniform on compact subsets.
\item $\min_{\alpha \in \theta} \alpha(\kappa(g_n))\rightarrow +\infty$, $U_\theta(g_n) \rightarrow x^+$, and $U_{\opp^* \theta}(g_n^{-1}) \rightarrow x^-$. 
\end{enumerate} 

\end{proposition}

\subsection{The partial Iwasawa cocycle} 

The \emph{Iwasawa cocycle} $B_{\Delta}^{IW} : \Gsf \times \Fc_{\Delta} \to \fa$ is defined as follows: for $g \in \Gsf$ and $x \in \Fc_{\Delta}$, $B_{\Delta}^{IW}(g, x) \in \fa$ is the unique element such that
$$
gk \in \Ksf ( \exp B_{\Delta}^{IW}(g, x)) \Nsf
$$
for $k \in \Ksf$ such that $k \Psf_{\Delta} = x$ in $\Fc_{\Delta}$. 

For general $\theta \subset \Delta$, let 
$$
\mfa_\theta : = \{ H \in \mfa : \alpha(H) = 0 \text{ for all } \alpha \notin \theta\}.
$$
For $\alpha \in \Delta$, let $\omega_\alpha$ denote the (restricted) fundamental weight associated to $\alpha$.
Then $\{ \omega_\alpha|_{\mfa_\theta}\}_{\alpha \in \theta}$ is a basis for $\mfa_\theta^*$ and so there exists a unique projection 
\begin{equation*}
\pi_{\theta} : \fa \to \fa_{\theta}
\end{equation*}
satisfying 
\begin{equation}\label{eqn:defn of pi_theta to mfa_theta} 
\omega_\alpha \pi_\theta(H) = \omega_\alpha(H)
\end{equation} 
for all $H \in \mfa$ and $\alpha \in \theta$.

The \emph{partial Iwasawa cocycle} $B_{\theta}^{IW} : \Gsf \times \Fc_{\theta} \to \fa_{\theta}$ is defined as 
\begin{equation}\label{eqn:defn of BIW} 
B_{\theta}^{IW}(g, x) := \pi_{\theta} B_{\Delta}^{IW}(g, \tilde x)
\end{equation} 
for any $\tilde x \in \Fc_{\Delta}$ that projects to $x \in \Fc_{\theta}$ under the canonical projection $\Fc_{\Delta} \to \Fc_{\theta}$. The above definition is independent of the choice of $\tilde x$ and defines a cocycle \cite[Lemma 6.1]{Quint_PS}. 

Recall from the introduction that the partial Iwasawa cocycle can be used to define a notion of Patterson--Sullivan measures on the partial flag manifold $\Fc_\theta$ (see Definition~\ref{defn:Quints definition}).

\subsection{Linear Representations}\label{sec:linear representations}

 Throughout the paper,  for each $\alpha \in \Delta$ we fix an irreducible representation $\Phi_{\alpha} : \Gsf \to \SL(V_{\alpha})$ and a $\Phi_\alpha(\Ksf)$-invariant inner product on $V_\alpha$ with the following properties (using the notation in Section~\ref{sec:sl notation}): 
\begin{enumerate}[label=(R\arabic*)]
\item\label{item:SVs of Tits repn} 
There exists $N_{\alpha} \in \Nb$ 
such that if $g \in \Gsf$, then 
$$
\log \norm{\Phi_\alpha(g)}= N_{\alpha} \omega_\alpha(\kappa(g)) \quad \text{and}\quad \log \frac{\sigma_1(\Phi_{\alpha}(g))}{\sigma_2(\Phi_{\alpha}(g))} = \alpha(\kappa(g)).
$$
\item\label{item:splitting for Tits repn}  There exists a $\Phi_\alpha(\Asf)$-invariant orthogonal splitting $V_\alpha = V_\alpha^+ \oplus V_\alpha^-$ such that $\dim V_\alpha^+ = 1$. Moreover, if  $H \in \mfa$ and  $v \in V_\alpha^+$, then 
$$
\Phi_\alpha(e^H)v = e^{N_{\alpha} \omega_\alpha(H)} v.
$$

\item\label{item:boundary maps of Tits repn} There exist $\Phi_\alpha$-equivariant boundary maps $\zeta_\alpha : \Fc_\alpha \rightarrow \Pb(V_\alpha)$ and $\zeta_\alpha^* : \Fc_{\opp^* \alpha} \rightarrow \Gr_{\dim V_\alpha-1}(V_\alpha)$ such that:
\begin{enumerate}
\item $\zeta_\alpha(\Psf_{\alpha}) = V_\alpha^+$ and $\zeta_\alpha^*(w_0 \Psf_{\opp^* \alpha } ) = V_\alpha^-$.
\item $x \in \Fc_\alpha$ and $y \in \Fc_{\opp^* \alpha}$ are transverse if and only if $\zeta_\alpha(x)$ and $\zeta_\alpha^*(y)$ are transverse. 
\end{enumerate} 

\end{enumerate} 

\begin{remark} 

   Such representations exist due to Tits \cite[Theorem 7.2]{Tits_representations}. Indeed, Tits proved the first claim in Property \ref{item:SVs of Tits repn}. For a proof of the second assertion in Property~\ref{item:SVs of Tits repn} and Property~\ref{item:splitting for Tits repn}, see for instance \cite[Lemma 2.13]{Smilga_proper} and \cite[Sections 6.8, 6.9]{BQ_book}. For a proof of Property~\ref{item:boundary maps of Tits repn}, see for instance ~\cite[Section 3]{GGKW2017}.

\end{remark}

\begin{remark} 
   We abuse notation and when $\alpha \in \theta$, also often use $\zeta_\alpha$ to also denote the map $\Fc_\theta \rightarrow \Pb(V_\alpha)$ obtained by precomposing $\zeta_\alpha: \Fc_\alpha \rightarrow \Pb(V_\alpha)$ with the natural projection $\Fc_\theta \rightarrow \Fc_\alpha$. Likewise, we also use $\zeta_{\opp ^*\alpha}^*$ to denote the analogous map defined on $\Fc_{\opp^* \theta}$. 
\end{remark}

The following lemma relates the projections to the flag manifolds introduced in Section~\ref{sec:proj to flag manifolds} to these representations. 

\begin{lemma}[{\cite[Lemma 3.5]{KimZimmer_vector}}] \label{lem:rank one limits in Valpha} 
   Fix $\alpha \in \Delta$ and assume $\{g_n\} \subset \Gsf$ is such that  $\alpha(\kappa(g_n)) \rightarrow +\infty$, $U_\alpha(g_n) \rightarrow x$,  and $U_{\opp^* \alpha }(g_n^{-1}) \rightarrow y$. Then any limit point of 
$$
\frac{1}{\norm{\Phi_\alpha(g_n)}} \Phi_\alpha(g_n) \quad \text{in } \End(V_{\alpha})
$$
has image $\zeta_\alpha(x)$ and kernel $\zeta_\alpha^*(y)$. 
\end{lemma}

\subsection{Irreducible actions} 

A subgroup $\Hsf < \SL(V)$ is called \emph{irreducible} if there is no non-trivial and proper subspace of $V$ invariant under $\Hsf$, and is called \emph{strongly irreducible} if any finite index subgroup of $\Hsf$ is irreducible. We transfer these notions to $\Gsf$ using the $\Phi_{\alpha}$'s.

\begin{definition}
    A subgroup $\Ga < \Gsf$ is \emph{$(\Phi_{\alpha})_{\alpha \in \theta}$-irreducible} if $\Phi_{\alpha}(\Ga) < \SL(V_{\alpha})$ is irreducible for all $\alpha \in \theta$, and \emph{strongly  $(\Phi_{\alpha})_{\alpha \in \theta}$-irreducible}  if any finite index subgroup of $\Ga$ is  $(\Phi_{\alpha})_{\alpha \in \theta}$-irreducible. 
\end{definition} 

\begin{remark}  \label{remark:Zdense irreducible}
   Notice that a Zariski dense subgroup is strongly  $(\Phi_{\alpha})_{\alpha \in \Delta}$-irreducible. \end{remark} 

We will use the following observation several times. 
 
\begin{lemma}[{\cite[Lemma 3.8]{KimZimmer_vector}}] \label{lem:rho irr implies flag irr}
Suppose $\Ga < \Gsf$ is strongly  $(\Phi_{\alpha})_{\alpha \in \theta}$-irreducible. If 
\begin{itemize}
\item $\alpha_1,\dots, \alpha_m$ are (possibly non-distinct) elements of $\theta$, 
\item $v_i \in V_{\alpha_i} \smallsetminus\{0\}$ for $i=1,\dots,m$, and 
\item $W_i \subset V_{\alpha_i}$ is a proper linear subspace for $i=1,\dots, m$, 
\end{itemize} 
then there exists $\gamma \in \Gamma$ with 
$$
\Phi_{\alpha_i}(\gamma) v_i \notin W_i
$$
for all $i=1,\dots, m$. 
\end{lemma}


\subsection{(Relatively) Anosov and transverse groups} \label{section:Anosov transverse}


In the rest of this section we recall the definitions of (relatively) Anosov and transverse groups, and some results about Patterson--Sullivan measures for such groups.

Given a discrete subgroup $\Gamma < \Gsf$, a point $x \in \Fc_{\theta}$ is a \emph{limit point} of $\Ga$ if there exists an escaping sequence $\{ \ga_n \} \subset \Ga$ with $\min_{\alpha \in \theta} \alpha(\kappa(\gamma_n)) \to + \infty$ and  $U_{\theta}(\ga_n) \rightarrow x$. The \emph{limit set of $\Gamma$}, denoted by
$$
\La_{\theta}(\Ga) \subset \Fc_{\theta},
$$
is the set of all limit points of $\Ga$. Proposition~\ref{prop:NS dynamics} implies that the points in $\Lambda_\theta(\Gamma)$ are exactly the points $x^+ \in \Fc_\theta$ where there exists a sequence $\{\gamma_n\} \subset \Gamma$ and a non-empty open set $\Uc \subset \Fc_\theta$ such that $\gamma_n x \rightarrow x^+$ for all $x \in \Uc$, uniformly on compact subsets.

A discrete subgroup  $\Ga < \Gsf$ is \emph{$\Psf_{\theta}$-transverse} if $\min_{\alpha \in \theta} \alpha(\kappa(\gamma_n)) \to + \infty$ for any sequence $\{ \ga_n \} \subset \Ga$ of distinct elements and any two distinct points in $\La_{\theta \cup \opp^* \theta}(\Ga)$ are transverse.

Sometimes the definition of transverse group includes the assumption that $\theta$ is symmetric (i.e., $\theta = \opp^* \theta)$. However, as the next observation  demonstrates, this results in no loss of generality. 

\begin{observation}\label{obs:symmetry theta} $\Gamma < \Gsf$ is $\Psf_{\theta}$-transverse if and only if $\Gamma$ is $\Psf_{\theta \cup \opp^*\theta}$-transverse. Moreover, in this case the the projection $\Fc_{\theta \cup \opp^* \theta} \rightarrow \Fc_\theta$ induces a homeomorphism $\Lambda_{\theta \cup \opp^* \theta}(\Gamma) \rightarrow \Lambda_\theta(\Gamma)$.
\end{observation}

A $\Psf_{\theta}$-transverse group is called \emph{non-elementary} if $\# \La_{\theta}(\Ga) \ge 3$, in which case the natural $\Ga$-action on $\La_{\theta}(\Ga)$ is a minimal convergence action  and $\# \La_{\theta}(\Ga) = + \infty$, see \cite[Theorem 4.16]{KLP_Anosov} or \cite[Proposition 3.3]{CZZ2026}.

A $\Psf_{\theta}$-transverse group is  \emph{$\Psf_{\theta}$-Anosov} if the action of $\Ga$ on $\La_{\theta}(\Ga)$ is a uniform convergence action, equivalently $\Gamma$ is word hyperbolic as an abstract group and there exists an equivariant homeomorphism from the Gromov boundary to the limit set $\La_{\theta}(\Ga)$ \cite{bowditch1998a-topological}. Likewise, a $\Psf_{\theta}$-transverse group is  \emph{$\Psf_{\theta}$-relatively Anosov} if the action of $\Ga$ on $\La_{\theta}(\Ga)$ is geometrically finite, equivalently if $\Gamma$ has the structure of a relatively hyperbolic group and there exists an equivariant homeomorphism from the associated Bowditch boundary to the limit set $\La_{\theta}(\Ga)$ \cite{Yaman2004topological}.

Canary, Zhang, and the second author established the following existence and uniqueness results for Patterson--Sullivan measure supported on the limit set.

\begin{theorem} \label{thm:CZZ PS transverse existence} \label{thm:CZZ PS transverse}
Suppose $\Ga < \Gsf$ is a non-elementary $\Psf_{\theta}$-transverse group,  $\phi \in \fa_{\theta}^*$, and $\delta^{\phi}(\Ga) < + \infty$. 
\begin{enumerate}
   \item \cite[Proposition 4.2]{CZZ2024} There exists a $(\Ga, \phi, \delta^{\phi}(\Ga))$-Patterson--Sullivan measure supported on $\La_{\theta}(\Ga)$.
   \item \cite[Proposition 8.1]{CZZ2024} If $\mu$ is a $(\Ga, \phi, \beta)$-Patterson--Sullivan measure supported on $\La_{\theta}(\Ga)$, then $\beta \geq \delta^\phi(\Gamma)$. 
   \item \cite[Corollaries 12.1 and 12.2, and Proposition 9.1]{CZZ2024} If 
   $$
   \sum_{\ga \in \Ga} e^{-\delta^{\phi}(\Ga) \phi(\kappa(\ga))} = + \infty,
   $$
   then there exists a unique $(\Ga, \phi, \delta^{\phi}(\Ga))$-Patterson--Sullivan measure $\mu$ supported on $\La_{\theta}(\Ga)$. Moreover, $\Gamma$ acts ergodically on $( \La_{\theta}(\Ga), \mu)$ and $\mu$ is supported on the conical limit set in the sense of the convergence action of $\Ga$ on $\La_{\theta}(\Ga)$. 
\end{enumerate}
\end{theorem}

\begin{remark} When $\Ga < \Gsf$ is Zariski dense, the uniqueness holds for measures supported on $\Fc_{\theta}$, as shown by the first author, Oh, and Wang \cite{KOW_PD}. \end{remark}


\section{Vector-valued horofunction compactifications and PS-measures}\label{sec:compactifications and PS measures}


In this section, we recall the vector-valued horofunction compactifications of the symmetric space $X = \Gsf/\Ksf$ associated to $\Gsf$ introduced in our other work~\cite{KimZimmer_vector}. It turns out that they contain $\theta$-boundaries, and we also consider Patterson--Sullivan measures there. Similar compactifications, but using Finsler metrics, appear in \cite{KL_Finsler,HSWW_horofunction,LP_horofunction_2023,Lemmens_horofunction_2023}.

Fix the basepoint $o:=\Ksf \in X$. The symmetric space distance is given by 
$$
\dist_X(go, ho) = \norm{\kappa(g^{-1}h)}
$$
where $\norm{\cdot}$ is some norm on $\mfa$. 
For $x = go$, define the vector-valued horofunction $b_x : X \rightarrow \fa$ by 
$$
b_x(ho) = \kappa(h^{-1}g) - \kappa(g). 
$$
Then the maps $\{ b_x : x \in X\}$ are uniformly Lipschitz~\cite[Lemma 4.1]{KimZimmer_vector}.

For non-empty $\theta \subset \Delta$,  let $\pi_\theta : \mfa \rightarrow \mfa_\theta$ be the projection satisfying Equation~\eqref{eqn:defn of pi_theta to mfa_theta}  and then let  $\partial_{\theta} X$ be the set of functions $\xi : X \rightarrow \mfa_\theta$ where there exists an escaping sequence $\{x_n\} \subset X$ with $\pi_\theta b_{x_n} \rightarrow \xi$ in the compact-open topology. The uniform Lipschitzness above, together with the separability of $X$, implies that $\partial_\theta X$ is compact in the compact-open topology. Further, $\Gsf$  acts on $\partial_{\theta} X$  by 
$$
g\cdot \xi = \xi \circ g^{-1} - \xi(g^{-1}o).
$$
The space $\partial_\theta X$ can be used to compactify $X$. 

\begin{proposition}[{\cite[Proposition 4.2]{KimZimmer_vector}}] \label{prop:compactification} The space $\overline{X}^{\theta}: = X \sqcup \partial_{\theta} X$ has a topology which makes it a compactification of $X$, that is $\overline{X}^{\theta}$ is a compact metrizable space and the inclusion $X \hookrightarrow \overline{X}^{\theta}$ is a topological embedding with open dense image. Moreover with respect to this topology:
\begin{enumerate}
\item  $\{x_n\} \subset X$ converges to $\xi \in \partial_{\theta} X$ if and only if $\dist_X(o,x_n) \rightarrow +\infty$ and $\pi_\theta b_{x_n} \rightarrow \xi$ in the compact-open topology.
\item The $\Gsf$-action on $\overline{X}^{\theta}$ is continuous.
\end{enumerate} 

\end{proposition}

Patterson--Sullivan measures on $\partial_{\theta} X$ can naturally be defined as follows. 

\begin{definition}\label{defn: PS on X theta}
Given  a subgroup $\Gamma < \Gsf$, $\theta \subset \Delta$, $\phi \in \mfa_\theta^*$, and $\delta \geq 0$, a Borel probability measure $\mu$ on $\partial_{\theta} X$ is a \emph{coarse $(\Gamma, \phi, \delta)$-Patterson--Sullivan measure} if  there exists $C \ge 1$ such that for every $\gamma \in \Gamma$ the measures $\mu$, $\gamma_*\mu$ are absolutely continuous and 
$$
C^{-1}  e^{-\delta \phi  \xi(\ga o)} \le \frac{d\gamma_* \mu}{d\mu}(\xi) \le  C e^{-\delta \phi \xi(\ga o)} \quad \mu\text{-a.e.}
$$
We call $\mu$ a \emph{$(\Ga, \phi, \delta)$-Patterson--Sullivan measure} if $C = 1$. 
\end{definition} 

Using Patterson's original construction for Fuchsian groups, we proved the following existence result. 

\begin{proposition}[{\cite[Proposition 4.5]{KimZimmer_vector}}]  \label{prop:PS meas exists}
   If $\Gamma < \Gsf$ is discrete, $\phi \in \mfa_\theta^*$, and $\delta^\phi(\Gamma) < +\infty$, then there exists a $(\Gamma, \phi, \delta^\phi(\Gamma))$-Patterson--Sullivan measure on $\partial_\theta X$. 
\end{proposition}

\subsection{Embeddings of partial flag manifolds}

The partial flag manifold $\Fc_{\theta}$ turns out to be naturally embedded into $\partial_{\theta} X$. Recall that $B_\theta^{IW} : \Gsf \times \Fc_\theta \rightarrow~\mfa_\theta$ denotes the partial Iwasawa cocyle. Quint~\cite[Lemma 6.6]{Quint_PS} proved that 
$$
\lim_{n \rightarrow \infty} \pi_\theta\kappa(g^{-1}h_n)-\pi_\theta\kappa(h_n) =B^{IW}_\theta(g^{-1},x) \quad \text{for all} \quad g \in \Gsf
$$
when $\min_{\alpha \in \theta} \alpha(\kappa(h_n)) \rightarrow +\infty$  and $U_\theta(h_n) \rightarrow x$. Using this fact, we showed that $\Fc_\theta$ embeds into $\partial_\theta X$. 
 
\begin{proposition}[{\cite[Proposition 4.7]{KimZimmer_vector}}] \label{prop:embedding of partial flags}
There exists a topological embedding  $\iota : \Fc_\theta \rightarrow \partial_\theta X$ that satisfies 
\begin{equation}\label{eqn:definition of iota}
\iota(x)(ho) = B_\theta^{IW}(h^{-1},x)
\end{equation}
for all $x \in \Fc_\theta$ and $h \in \Gsf$. Moreover:
\begin{enumerate}
\item If a sequence $\{g_n\} \subset \Gsf$ satisfies $\min_{\alpha \in \theta} \alpha(\kappa(g_n)) \to + \infty$ and $U_{\theta}(g_n) \rightarrow x$,  then 
$$
g_n o \rightarrow \iota(x) \quad \text{in} \quad  \overline{X}^\theta.
$$
\item If $\mu$ is a (coarse, resp.) $(\Gamma, \phi, \delta)$-Patterson--Sullivan measure on $\Fc_\theta$ in the sense of Definition~\ref{defn:Quints definition}, then $\iota_* \mu$ is a (coarse, resp.) $(\Gamma, \phi, \delta)$-Patterson--Sullivan measure on $\partial_\theta X$ in the sense of Definition~\ref{defn: PS on X theta}.
\end{enumerate}
\end{proposition}


\subsection{Shadows and contracting~conical~limit~sets}\label{sec:shadows and cc limit set} 
 

We now define shadows on $\partial_{\theta} X$ and use them to introduce the contracting conical limit set of a discrete subgroup. 

 In \cite[Lemma 4.6]{KimZimmer_vector}, we proved: if $\xi \in \partial_\theta X$ and $g \in \Gsf$, then  
$$
\omega_\alpha\xi(g^{-1}o) \leq \omega_\alpha \kappa(g)
$$
for all $\alpha \in \theta$. Given $g \in \Gsf$, we then define shadows in $\partial_\theta X$ by considering the set of functionals $\xi$ that are close to maximizing the expression $\omega_\alpha\xi(g^{-1}o)$ for all $\alpha \in \theta$. 

More precisely, for $g \in \Gsf$ and $R > 0$, the associated \emph{shadow} is defined by
$$
\Oc_R^\theta(g):= g \cdot \{    \xi \in \partial_{\theta} X : \omega_\alpha \xi(g^{-1}o) > \omega_\alpha \kappa(g) - R \text{ for all } \alpha \in \theta  \}.
$$
 In what follows we use $\pi_\theta$ to denote both the projection $\fa \rightarrow \fa_\theta$ satisfying Equation~\eqref{eqn:defn of pi_theta to mfa_theta} and the map $\partial_{\Delta} X \to \partial_{\theta} X$ obtained by the postcomposition with this projection. Since $\omega_\alpha \xi = \omega_\alpha \pi_\theta \xi$ for all $\alpha \in \theta$ and $\xi \in \partial_\Delta X$, we have 
$$
\pi_\theta \Oc_R^\Delta(g) \subset \Oc_R^\theta(g). 
$$
We also use Proposition \ref{prop:embedding of partial flags} to view $\Fc_\theta$ as a subset of $\partial_\theta X$. Then Equation~\eqref{eqn:defn of BIW}  and Proposition~\ref{prop:embedding of partial flags} imply that $\pi_\theta|_{\Fc_\Delta}$ coincides with the natural projection $\Fc_\Delta \rightarrow \Fc_\theta$ given by $g \Psf_\Delta \rightarrow g \Psf_\theta$.

The boundary $\partial_\theta X$ has a natural cocycle  $B_\theta : \Gsf \times \partial_{\theta} X \rightarrow \fa$ defined by 
$$
B_\theta(g,x) :=  \xi(g^{-1} o).
$$
Indeed one can see that 
$
B_\theta(g_1g_2,\xi) = B_\theta(g_1, g_2 \xi) + B_\theta(g_2, \xi)
$
for all $g_1,g_2 \in \Gsf$ and $\xi \in \partial_{\theta} X$.  

In ~\cite{KimZimmer_vector}, we verified that the definitions above give PS-systems on $\overline{X}^\theta$.

\begin{theorem}[{\cite[Theorem 6.1]{KimZimmer_vector}}] \label{thm:Zdense PS} Suppose $\theta \subset \Delta$ and $\phi \in \mfa_\theta^*$. 
   If $\Gamma < \Gsf$ is strongly $(\Phi_{\alpha})_{\alpha \in \theta}$-irreducible and $\mu$ is a coarse $(\Gamma, \phi, \delta)$-Patterson--Sullivan measure on $\partial_{\theta} X$, then $(\partial_{\theta} X, \Gamma, \phi \circ B_{\theta}, \mu)$ is a PS-system, with  magnitude $\norm{\ga}_{\phi} := \phi(\kappa(\ga))$ and the  $R$-shadows $\Oc_R^{\theta}(\ga)$ for each $\ga \in \Ga$. Moreover, \ref{item:baire} holds.

\end{theorem}   

For transverse groups, we showed that the system is well-behaved. 

\begin{theorem}[{\cite[Theorem 6.2]{KimZimmer_vector}}]  \label{thm:transverse well-behaved}
Suppose $\theta \subset \Delta$,
$\phi \in \mfa_\theta^*$, and $\Gamma < \Gsf$ is a strongly $(\Phi_{\alpha})_{\alpha \in \theta}$-irreducible  $\Psf_{\theta}$-transverse group. Let $\mu$ be a coarse $(\Gamma, \phi, \delta)$-Patterson--Sullivan measure on $\partial_{\theta} X$. Then the PS-system $(\partial_{\theta}X , \Gamma, \phi \circ B_{\theta}, \mu)$ in Theorem \ref{thm:Zdense PS} is well-behaved with respect to the trivial hierarchy $\mathscr{H}(R) \equiv \Ga$.

\end{theorem}

\begin{remark} See \cite[Section 5.2]{KimZimmer_vector} for a comparison of shadows in $\partial_{\theta} X$ with the shadows in $\Fc_{\theta}$ defined in terms of positive Weyl chambers in the symmetric space. 
 \end{remark}

\subsection{Properties of shadows} We record some properties of shadows. For part~(3), we fix any metric generating the topology on $\overline{X}^\theta$.

\begin{lemma}[{\cite[Lemmas 5.4, 5.5, 5.6]{KimZimmer_vector}}] \label{lem:shadow translate shadow} \label{lem:endpoint in shadow} \label{lem:diam decay}
\
\begin{enumerate}
   \item For any  $g \in \Gsf$ and $R > 0$, there exists $R' = R'(g, R) > 0$ such that: if $h \in \Gsf$, then 
   $$
   g\Oc_R^{\theta}(h) \subset \Oc_{R'}^{\theta}(gh).
   $$
   \item For $g \in \Gsf$ with $\min_{\alpha \in \theta} \alpha(\kappa(g)) > 0$, 
   $$
   U_{\theta}(g) \in \Oc_R^{\theta}(g)
   $$
   for all $R > 0$. 

   \item For a sequence $\{g_n\} \subset \Gsf$, if $\min_{\alpha \in \theta} \alpha(\kappa(g_n)) \rightarrow +\infty$, then 
$$
 {\rm diam} \Oc_R^\theta(g_n) \rightarrow 0.
$$

\end{enumerate}
\end{lemma}

\subsection{Contracting conical limit sets}

Given a subgroup $\Ga < \Gsf$, we define its \emph{conical limit set} in $\partial_{\theta} X$ by
\begin{equation} \label{eqn:conical limit set classic}
   \La_{\theta}^{\rm con}(\Ga) := \left\{ \xi \in \partial_{\theta} X : \exists \, R > 0, \text{ escaping } \{ \ga_n \} \subset \Ga \text{ s.t. } \xi \in \bigcap_{n \ge 1} \Oc_R^{\theta}(\ga_n) \right\},
\end{equation}
following the classical definition of conical limit sets in rank one settings.

When $\Gsf$ is of higher rank, the intersection $\bigcap_{n \ge 1}\Oc_R^{\theta}(\ga_n)$ may not be a singleton, even after intersecting with the partial flag manifold $\Fc_{\theta}$ and the conical limit set $\La_{\theta}^{\rm con}(\Ga)$, even after intersecting with $\Fc_{\theta}$, may not be a subset of the limit set $\La_{\theta}(\Ga)$, see ~\cite[Example 5.8]{KimZimmer_vector}.  Hence, in view of Lemma \ref{lem:diam decay}(3), we define the following smaller subset of conical limit set, which only involves shrinking shadows.

\begin{definition} \label{def:concon}
   
Given a subgroup $\Ga < \Gsf$, we call $\xi \in \partial_{\theta} X$ a \emph{contracting conical limit point} of $\Ga$ if there exist $R > 0$ and a sequence $\{ \ga_n \} \subset \Ga$ such that
$$
\lim_{n \to + \infty} \min_{\alpha \in \theta} \alpha(\kappa(\ga_n)) = + \infty \quad \text{and} \quad \xi \in \bigcap_{n \ge 1} \Oc_{R}^{\theta}(\ga_n).
$$
We denote by $\La_{\theta}^{\rm concon}(\Ga)$ the \emph{contracting conical limit set} of $\Ga$, which is defined as the set of all contracting conical limit points of $\Ga$. 
\end{definition}

For general $\Ga < \Gsf$, the contracting conical limit set $\La_{\theta}^{\rm concon}(\Ga)$ is a $\Gamma$-invariant subset of the limit set $\Lambda_\theta(\Gamma) \subset \Fc_\theta$ introduced in Section~\ref{section:Anosov transverse} (this follows from Lemma~\ref{lem:shadow translate shadow}).

 For transverse groups we have the following.

\begin{proposition}[{\cite[Proposition 5.3, Theorem 6.3]{KimZimmer_vector}}] \label{prop:concon limit set for transverse groups} If $\Ga < \Gsf$ is a non-elementary $\Psf_{\theta}$-transverse group, then:
\begin{enumerate} 
\item $\La_{\theta}^{\rm concon}(\Ga)=\La_\theta^{\rm con}(\Gamma)$ and $\La_\theta^{\rm con}(\Gamma)$ is a subset of the limit set $\Lambda_\theta(\Gamma) \subset \Fc_\theta$ introduced in Section~\ref{section:Anosov transverse}. 
\item  $\La_\theta^{\rm con}(\Gamma)$ coincides  with the conical limit set in the convergence group sense (recall that $\Gamma$ acts on $\Lambda_\theta(\Gamma)$ as a convergence group). 
\item If $\phi \in \fa_{\theta}^*$, $\delta^{\phi}(\Ga) < + \infty$, and  $\sum_{\ga \in \Ga} e^{-\delta^{\phi}(\Ga) \phi(\kappa(\ga))} = + \infty$, then there exists a unique $(\Ga, \phi, \delta^{\phi}(\Ga))$-Patterson--Sullivan measure $\mu$ on $\partial_{\theta} X$. Moreover, 
   $$
   \mu(\La_{\theta}^{\rm con}(\Ga)) = 1.
   $$
   \end{enumerate} 

\end{proposition}

\part{Construction of measurable boundary maps}


\section{Continuity properties of boundary maps}\label{sec:continuity properties of boundary maps}

 
In this section we prove the following technical continuity result for measurable maps whose domains are well-behaved PS-systems. Notice that the theorem does not assume that the map has any equivariance properties.

\begin{theorem}\label{thm:weak continuity property of boundary maps} Suppose $(M, \Gamma, \sigma, \mu)$ is a well-behaved PS-system with respect to a hierarchy $\mathscr{H} = \{ \mathscr{H} (R) \subset \Ga : R \ge 0\}$. Assume $F : M \rightarrow (Y, \dist_Y)$ is a Borel measurable map into a separable metric space. 

For $\mu$-a.e.\ $x_0 \in M$, if $x_0 \in \bigcap_{n \ge 1} \ga \Oc_R(\ga_n)$ for some $\ga \in \Ga$, escaping sequence $\{ \ga_n \}$ with $\ga_n \in \mathscr{H}(n)$, and $R > 0$, then there exist a subsequence $\{\gamma_{n_j}\}$, $h_1,\dots, h_m \in \Gamma$, and a $\mu$-null set $E \subset M$ such that for each $x  \in M \smallsetminus E$ there exists some $1 \leq i \leq m$ where
$$
\lim_{j \rightarrow  + \infty} F(\gamma \gamma_{n_j} h_i x) = F(x_0).
$$
Moreover, if 
$$
\mu(Z) = 0
$$
whenever $S_n \rightarrow +\infty$  and $\left[  M \smallsetminus \gamma_n^{-1} \Oc_{S_n}(\ga_n) \right] \to Z$ with respect to the Hausdorff distance, then we can assume that $m=1$ and $h_1 = \id$. 
\end{theorem} 

We use the following result from our earlier work.

\begin{lemma}[{\cite[Corollary~5.5]{KimZimmer1}}] \label{lem:KZ continuity bdr map} Suppose $(M, \Gamma, \sigma, \mu)$ is a well-behaved PS-system with respect to a hierarchy $\mathscr{H} = \{ \mathscr{H} (R) \subset \Ga : R \ge 0\}$. Assume that $F : M \rightarrow (Y, \dist_Y)$ is a Borel measurable map into a separable metric space.  

If $R > 0$ is sufficiently large, then for $\mu$-a.e.\ $x \in M$ we have 
$$
0 = \lim_{n \rightarrow  + \infty} \frac{1}{\mu(\gamma\Oc_R(\gamma_n))}\mu\left( \left\{ y \in \gamma\Oc_R(\gamma_n) : \dist_Y(F(x), F(y)) > \epsilon\right\} \right)
$$
for all $\epsilon > 0$ whenever $x \in \bigcap_{n \ge 1} \ga \Oc_R(\ga_n)$ for some $\ga \in \Ga$ and escaping sequence $\{ \ga_n \} \subset \mathscr{H}(R)$.

\end{lemma} 

\subsection{Proof of Theorem~\ref{thm:weak continuity property of boundary maps}} Fix $R_j \nearrow +\infty$. After possibly increasing each $\{R_j\}$, by Lemma \ref{lem:KZ continuity bdr map} for each $j \geq 1$  there exists a full $\mu$-measure set $M_j \subset M$ such that: if $x_0 \in M_j \cap \bigcap_{n \ge 1} \ga \Oc_{R_j}(\ga_n)$ for some $\ga \in \Ga$ and escaping sequence $\{ \ga_n \} \subset \mathscr{H}(R_j)$, then 
$$
0 = \lim_{n \rightarrow  + \infty} \frac{1}{\mu(\gamma\Oc_{R_j}(\gamma_n))}\mu\left( \left\{ x \in \gamma\Oc_{R_j}(\gamma_n) : \dist_{Y}(F(x), F(x_0)) > \epsilon\right\} \right)
$$
for all $\epsilon > 0$.

Let
$$
M':= \bigcap_{\gamma \in \Gamma, j \geq 1} \gamma M_j.
$$
Then $M'$ is a $\Gamma$-invariant set with full $\mu$-measure.

Fix $x_0 \in M'$, $\gamma \in \Gamma$, and $\{\gamma_n\} \subset \Gamma$ escaping such that $\gamma_n \in  \mathscr{H}(n)$ and 
$$
x_0 \in \bigcap_{n \geq 1}\gamma \Oc_R(\gamma_n)
$$
for some $R > 0$. Passing to a subsequence and relabelling, we can suppose that $\gamma_n \in  \mathscr{H}(R_n)$ for all $n \geq 1$.

We can also assume $R_j \geq R$ for all $j$ and hence by Property \ref{item:shadow inclusion}
$$
x_0 \in \bigcap_{n \geq 1}\gamma \Oc_{R_j}(\gamma_n).
$$
Since $j \mapsto \mathscr{H}(R_j)$ is non-increasing, we have $\{\gamma_n\}_{n \geq j} \subset \mathscr{H}(R_j)$. So 
$$
0 = \lim_{n \rightarrow  + \infty} \frac{1}{\mu(\gamma\Oc_{R_j}(\gamma_n))}\mu\left( \left\{ x \in \gamma\Oc_{R_j}(\gamma_n) : \dist_{Y}(F(x), F(x_0)) > \epsilon\right\} \right)
$$
for all $j \in \Nb$ and $\epsilon > 0$. By Property~\ref{item:coycles are bounded}, 
\begin{align*}
 \lim_{n \rightarrow  + \infty} & \frac{1}{\mu(\Oc_{R_j}(\gamma_n))}\mu\left( \left\{ x \in \Oc_{R_j}(\gamma_n) : \dist_{Y}(F(\gamma x), F(x_0)) > \epsilon\right\} \right) \\
 & \asymp  \lim_{n \rightarrow  + \infty}  \frac{1}{\gamma_*^{-1}\mu(\Oc_{R_j}(\gamma_n))}\gamma_*^{-1}\mu\left( \left\{ x \in \Oc_{R_j}(\gamma_n) : \dist_{Y}(F(\ga x), F(x_0)) > \epsilon\right\} \right) \\
 & =  \lim_{n \rightarrow  + \infty} \frac{1}{\mu(\gamma\Oc_{R_j}(\gamma_n))}\mu\left( \left\{ x \in \gamma\Oc_{R_j}(\gamma_n) : \dist_{Y}(F(x), F(x_0)) > \epsilon\right\} \right)=0
\end{align*}
for all  $j \in \Nb$ and $\epsilon > 0$.

By Property~\ref{item:almost constant on shadows}, there exists $C_j  > 1$ such that 
   $$
   C_j e^{-\delta\norm{\gamma_n}_{\sigma}} \le \frac{d{\ga_n^{-1}}_*\mu}{d\mu} \le C_j e^{-\delta\norm{\gamma_n}_{\sigma}} \quad \mu\text{-a.e.\ on $\ga_n^{-1} \Oc_{R_j}(\ga_n)$}
   $$
   where $\delta \ge 0$ is the dimension of the PS-system.
 Hence 
   \begin{align*}
0 & = \lim_{n \rightarrow  + \infty} \frac{1}{(\gamma_n^{-1})_*\mu(\gamma_n^{-1}\Oc_{R_j}(\gamma_n))}(\gamma_n^{-1})_*\mu\left( \left\{ x \in \gamma_n^{-1}\Oc_{R_j}(\gamma_n) : \dist_{Y}(F(\gamma \gamma_n x), F(x_0)) > \epsilon\right\} \right) \\
& \asymp  \lim_{n \rightarrow  + \infty} \frac{1}{\mu(\gamma_n^{-1}\Oc_{R_j}(\gamma_n))}\mu\left( \left\{ x \in \gamma_n^{-1}\Oc_{R_j}(\gamma_n) : \dist_{Y}(F(\gamma \gamma_n x), F(x_0)) > \epsilon\right\} \right)
\end{align*}
for all  $j \in \Nb$ and $\epsilon > 0$. Since
$$
\mu(\gamma_n^{-1}\Oc_{R_j}(\gamma_n)) \leq 1,
$$
then
   \begin{align*}
0 =  \lim_{n \rightarrow  + \infty} \mu\left( \left\{ x \in \gamma_n^{-1}\Oc_{R_j}(\gamma_n) : \dist_{Y}(F(\gamma \gamma_n x), F(x_0)) > \epsilon\right\} \right)
\end{align*}
for all  $j \in \Nb$ and $\epsilon > 0$. Then we can find $S_n=R_{j_n} \rightarrow  + \infty$ and $\epsilon_n \searrow 0$ such that $j_n \leq n$  and after passing to a subsequence of $\{ \ga_n\}$,
$$
\sum_{n \ge 1}  \mu\left( \left\{ x \in\gamma_n^{-1}\Oc_{S_n}(\gamma_n) : \dist_{Y}(F(\gamma \gamma_n x), F(x_0)) > \epsilon_n\right\} \right)<+\infty.
$$
Let
$$
E_n : = \left\{ x \in\gamma_n^{-1}\Oc_{S_n}(\gamma_n) : \dist_{Y}(F(\gamma \gamma_n x), F(x_0)) > \epsilon_n\right\} 
$$
and 
$$
E':= \bigcap_{N\ge 1} \bigcup_{n\ge N} E_n,
$$
which is $\mu$-null by the Borel--Cantelli Lemma.

Passing to a subsequence we can suppose that $\left[ M \smallsetminus \gamma_n^{-1}\Oc_{S_n}(\gamma_n) \right] \rightarrow Z$ with respect to the Hausdorff distance.

If $\mu(Z) = 0$, then we set
$$
E := E' \cup  Z
$$
which is $\mu$-null. Fix $x \in M \smallsetminus E$. Since $x \notin Z$ and $Z$ is closed, there exists some  $N_0 \in \Nb$ such that $x \in \bigcap_{n \ge N_0} \ga_n^{-1} \Oc_{S_n} (\ga_n)$. Since $x \notin E'$, there exists some $N_1 \in \Nb$ such that $x \notin \bigcup_{n \ge N_1} E_n$. Hence for $n \ge \max (N_0, N_1)$ we have 
$$
\dist_{Y}( F(\ga \ga_n x), F(x_0)) \le \epsilon_n.
$$
Since $\epsilon_n \to 0$, this implies that $F(\ga \ga_n x) \to F(x_0)$, and the ``moreover'' part follows.

In general, we set
$$
E := \Ga \cdot E'
$$
which is $\mu$-null by the quasi-invariance of $\mu$. By 
Property \ref{item:empty Z intersection}, we can pick $h_1, \dots, h_m \in \Ga$ such that 
$$
\bigcap_{i=1}^m h_i^{-1} Z = \emptyset. 
$$
Then fix $r > 0$ such that 
$$
\bigcap_{i=1}^m h_i^{-1} \Nc_r(Z) = \emptyset, 
$$
equivalently 
$$
\bigcup_{i=1}^m h_i^{-1} \left( M \smallsetminus \Nc_r(Z)\right)=M.
$$
Notice that 
$$
M \smallsetminus \Nc_r(Z) \subset \bigcap_{n \ge N_0}  \gamma_n^{-1}\Oc_{S_n}(\gamma_n) 
$$
for some $N_0 \in \Nb$. 

Now fix $x \in M \smallsetminus E$. Then we can fix $1 \leq i \leq m$ such that $h_i x \in M \smallsetminus \Nc_r(Z)$. Since $E= \Gamma \cdot E'$, we have $h_i x \notin E'$. In particular, there exists some $N_1 \in \Nb$ such that 
$$
h_i x \notin \bigcup_{n \geq N_1} E_n.
$$
Then for $n \geq \max (N_0, N_1)$ we have 
$$
\dist_Y( F(\gamma \gamma_n h_i x), F(x_0)) \leq \epsilon_n. 
$$
So $F(\gamma \gamma_n h_i x) \rightarrow F(x_0)$. 
\qed


\section{Boundary maps and almost everywhere contraction}\label{sec: existence of boundary maps}


 In this section, we prove the existence of boundary maps, which is the main technical result of this paper.  We first define two notions which we need to state our result.

Suppose $\Ga$ is a locally compact group acting on a measure space $(M,\mu)$ where $\mu$ is a $\Ga$-quasi-invariant measure. Then the action of $\Ga$ on $(M,\mu)$ is \emph{amenable} if there exists a sequence $\lambda_n : M \rightarrow \Prob(\Gamma)$ so that for $\mu$-a.e.\ $x \in M$ and every $g \in \Ga$,
$$
\lim_{n \rightarrow  + \infty} \norm{\lambda_n(gx) - g_*\lambda_n(x)} =0
$$
where $\Prob( \cdot)$  is the space of  Borel probability measures and $\norm{ \cdot }$ denotes the total variation.

A subgroup $H < \Gsf$ is \emph{$\Psf_{\theta}$-contracting} if there exists a sequence $\{h_n\} \subset H$ such that
$$
\alpha(\kappa(h_n)) \to + \infty \quad \text{for all } \alpha \in \theta.
$$
If $H$ is Zariski dense, then $H$ is $\Psf_{\theta}$-contracting, see \cite[Theorem 3.6]{GolMar1989} or \cite{Prasad1994}.

\begin{theorem}\label{thm:boundary map}  Suppose $(M, \Gamma, \sigma, \mu)$ is a well-behaved Patterson--Sullivan system with respect to a hierarchy $\mathscr{H} = \{ \mathscr{H} (R) \subset \Ga : R \ge 0\}$,
$$
\mu(\Lambda^{\rm con}(\mathscr{H})) = 1,
$$
 and the $\Gamma$-action on $(M,\mu)$ is amenable. 
  
If  $\rho : \Gamma \rightarrow \Gsf$ is a representation such that $\rho(\Ga)$ is $\Psf_{\theta}$-contracting and strongly $(\Phi_{\alpha})_{\alpha \in \theta}$-irreducible, then there exists a $\rho$-equivariant $\mu$-a.e.\ defined measurable map $f : M \rightarrow \La_{\theta}(\Ga)$. Moreover: 
 \begin{enumerate}
 \item For $\mu$-a.e.\ $x \in M$, if $\gamma \in \Gamma$, $\{\gamma_n\} \subset \Gamma$ is escaping, $\gamma_n \in \mathscr{H}(R_n)$ for some $R_n \rightarrow +\infty$, and 
$$
x \in \bigcap_{n \geq 1} \gamma \Oc_R(\gamma_n),
$$
for some $R > 0$, then 
$$
\alpha(\kappa(\rho(\gamma \gamma_n))) \rightarrow + \infty \quad \text{for all} \quad \alpha \in \theta
$$
and 
$$
U_{\theta}(\rho(\ga \ga_n)) \rightarrow f(x).
$$
\item If $\mathsf{m} : M \rightarrow \Prob(\Fc_\theta)$  is some $\rho$-equivariant $\mu$-a.e.\ defined measurable map, then $\mathsf{m}(x) = \Dc_{f(x)}$ $\mu$-a.e.\ In particular, if $f' : M \to \Fc_{\theta}$ is a $\rho$-equivariant $\mu$-a.e.\ defined measurable map, then $f = f'$ $\mu$-a.e.

\item  If the hierarchy $\mathscr{H}$ is trivial (i.e., $\mathscr{H}(R) \equiv \Ga$), then for $\mu$-a.e.\ $x \in M$, there exist $R > 0$, $\ga \in \Ga$, and an escaping sequence $\{ \ga_n \} \subset \Ga$ such that
$$
\lim_{n \to + \infty} \min_{\alpha \in \theta} \alpha(\kappa(\rho(\ga_n))) = + \infty, \quad x \in \bigcap_{n \ge 1} \ga \Oc_R(\ga_n), \quad \text{and} \quad f(x) \in \bigcap_{n \ge 1} \rho(\ga) \Oc_R^{\theta}(\rho(\ga_n)).
$$
In particular,
$$
f_* \mu(\La_{\theta}^{\rm concon}(\Ga)) = 1.
$$

\end{enumerate} 

\end{theorem} 

\begin{remark} Recall that the conditions in (1) imply that $\rho(\ga \gamma_n ) o \rightarrow f(x)$ in $\overline{X}^\theta$ by Proposition~\ref{prop:embedding of partial flags}.
\end{remark}

In the rest of this section, we assume the hypotheses of Theorem \ref{thm:boundary map}.

\begin{lemma} There exists a $\rho$-equivariant $\mu$-a.e.\ defined measurable map $\mathsf{m}_0 : M \rightarrow \Prob(\Fc_\theta)$. Moreover, $\mathsf{m}_0$ maps into a single $\Gsf$-orbit.
\end{lemma} 

\begin{proof} 

   The existence of the map $\mathsf{m}_0$ follows from  \cite[Proposition 4.3.9]{Zimmer_book}. By Theorem~\ref{thm:ergodicity KZ}, the $\Gamma$-action on $(M,\mu)$ is ergodic. Then the fact that $\mathsf{m}_0$ maps into a single $\Gsf$-orbit follows from \cite[Corollary 3.2.17, Proposition 2.1.11]{Zimmer_book}, together with the ergodicity of the $\Ga$-action on $(M, \mu)$.
\end{proof}

Since $\mathsf{m}_0$ maps into a single $\Gsf$-orbit, we can fix a measure $\nu_0 \in \Prob(\Fc_\theta)$ and a measurable map $x \in M \mapsto g_x \in \Gsf$ such that  
$$
\mathsf{m}_0(x) = (g_x)_* \nu_0
$$
for $\mu$-a.e.\ $x \in M$.

Let $d_0 \in \Zb_{\geq 0}$ be the smallest non-negative integer where there exists an (real) algebraic variety of dimension $d_0$ in $\Fc_\theta$ with positive $\nu_0$ measure. Then let $\Ac$ denote the set of all irreducible $d_0$-dimensional algebraic varieties  with positive $\nu_0$ measure. Notice that if $Z_1, Z_2 \in \Ac$ are distinct, then $\nu_0(Z_1 \cap Z_2) = 0$ by the minimality of $d_0$. Hence 
$$
\nu_0 = \sum_{Z \in \Ac} \nu_0(\cdot \cap Z) +\tilde  \nu_0
$$
where $\tilde \nu_0$ is a non-negative measure. 

Now fix $\epsilon > 0$ sufficiently small such that the set 
$$
 \Ac_\epsilon : = \{ Z \in \Ac : \nu_0(Z) > \epsilon\}
$$
is non-empty (notice that it must be finite). Then define 
$$
\nu_1 := \frac{1}{ \sum_{Z \in \Ac_{\epsilon}} \nu_0(Z)}  \sum_{Z \in \Ac_\epsilon} \nu_0(\cdot \cap Z)
$$
and define $\mathsf{m}_1 : M \rightarrow \Prob(\Fc_\theta)$ by 
$$
\mathsf{m}_1(x) = (g_x)_* \nu_1.
$$
Then by construction the map $\mathsf{m}_1$ is $\rho$-equivariant and for $\mu$-a.e.\ $x \in M$, $\mathsf{m}_1(x)$ is supported on finitely many irreducible $d_0$-dimensional algebraic varieties.

For $\alpha \in \theta$, let $\rho_\alpha : = \Phi_\alpha \circ \rho : \Gamma \rightarrow \SL(V_\alpha)$. In what follows, we view the boundary map $\zeta_\alpha$ in Property \ref{item:boundary maps of Tits repn} as having domain $\Fc_\theta$. 

\begin{proposition}\label{prop:key result in boundary map} There exists a $\Gamma$-invariant  full $\mu$-measure set $M' \subset M$ such that:  if $x_0 \in M'$, $\{\gamma_n\} \subset \Gamma$ is escaping,  $\gamma_n \in \mathscr{H}(R_n)$ for some $R_n \rightarrow +\infty$, 
$$
x_0 \in \bigcap_{n \geq 1}\gamma  \Oc_R(\gamma_n)
$$
for some $R > 0$, and 
$$
\frac{\rho_\alpha(\gamma \gamma_n)}{\norm{\rho_\alpha(\gamma \gamma_n)}} \rightarrow S_\alpha \in \End (V_\alpha)
$$
 for all $\alpha \in \theta$, then 
 \begin{enumerate}
 \item $\rank S_\alpha = 1$ for all $\alpha \in \theta$, 
 \item $\mathsf{m}_1(x_0) = \Dc_{\xi}$ where $\xi \in \Fc_\theta$ is the unique point with $\zeta_\alpha (\xi) = {\rm im} \, S_\alpha$ for all $\alpha \in \theta$. 
\end{enumerate}
\end{proposition}

Delaying the proof of Proposition~\ref{prop:key result in boundary map} for a moment, we complete the proof of Theorem \ref{thm:boundary map}. The proposition implies that $\mathsf{m}_1(x)$ is a Dirac mass when $x \in M' \cap \Lambda^{\rm con}(\mathscr{H})$ and so  $f(x) = \supp \mathsf{m}_1(x)$ provides a $\rho$-equivariant $\mu$-a.e.\ defined measurable map $f : M \rightarrow \Fc_\theta$. 

\begin{lemma} $f$ satisfies  part (1) of the ``moreover'' part of Theorem \ref{thm:boundary map}. \end{lemma}

\begin{proof} Suppose $x \in M'$ and $x$, $\gamma$, $\{\gamma_n\}$ satisfy the hypothesis of part (1). Proposition~\ref{prop:key result in boundary map} implies that any limit point of 
$$
\frac{\rho_\alpha(\gamma \gamma_n)}{\norm{\rho_\alpha(\gamma \gamma_n)}} 
$$
has rank one and image $\zeta_\alpha  f(x)$. So by Property~\ref{item:SVs of Tits repn}  of the representations $\Phi_\alpha$, we have 
$$
\alpha(\kappa(\rho(\gamma \gamma_n))) \rightarrow + \infty \quad \text{for all} \quad \alpha \in \theta.
$$
Then Lemma~\ref{lem:rank one limits in Valpha} implies that 
$
U_{\theta}(\rho(\ga \ga_n)) \rightarrow f(x).
$
\end{proof} 

We now prove the part (2) of the theorem.
 With the notation introduced at the start of the proof, $d_0=0$ and $\Ac_\epsilon$ is a point. Since $\epsilon > 0$ was an arbitrary small positive number, this implies that $\Ac$ is a point. Hence there exists $\lambda > 0$ such that 
$$ 
\mathsf{m}_0(x) = \lambda \Dc_{f(x)} + \tilde{\mathsf{m}}_0(x) 
$$
$\mu$-a.e., where $\tilde{\mathsf{m}}_0(x)$ is a non-atomic non-negative measure.

\begin{lemma} If $\mathsf{m}: M \rightarrow \Prob(\Fc_\theta)$ is an $\rho$-equivariant $\mu$-a.e.\ defined map, then $\mathsf{m}(x) = \Dc_{f(x)}$ $\mu$-a.e.  \end{lemma} 

\begin{proof} At the start of the argument,  $\mathsf{m}_0 : M \rightarrow \Prob(\Fc_\theta)$ was an arbitrary $\rho$-equivariant $\mu$-a.e.\ defined map. 
Hence repeating the proof above, there exist $\lambda' > 0$ and a $\mu$-a.e.\ defined $\rho$-equivariant map $f' : M \rightarrow \Fc_\theta$ such that $\mathsf{m}(x) = \lambda' \Dc_{f'(x)} + \tilde{\mathsf{m}}(x)$, where $\tilde{\mathsf{m}}(x)$ is a non-atomic non-negative measure. Further, $f'$ satisfies part (1) of the moreover part of the theorem. Since $f$ also satisfies part (1) of the moreover part of the theorem, we must have $f'=f$ $\mu$-a.e. 

Then it suffices to show that $\tilde{\mathsf{m}}(x) \equiv 0$ for $\mu$-a.e.\ $x \in M$. Suppose not. Then since $x \mapsto \tilde{\mathsf{m}}(x)(\Fc_\theta)$ is $\Gamma$-invariant and $\Gamma$ acts ergodically on $(M,\mu)$, we have $\tilde{\mathsf{m}}(x)$ is a non-zero measure for $\mu$-a.e.\ $x$. Then $\frac{\tilde{\mathsf{m}}}{\norm{\tilde{\mathsf{m}}}} : M \rightarrow \Prob(\Fc_\theta)$ is an  $\rho$-equivariant $\mu$-a.e.\ defined map. So repeating the argument again, we see that $\tilde{\mathsf{m}}(x)$ has an atom for $\mu$-a.e.\ $x$, which is a contradiction. 
\end{proof}

Now it remains to show the part (3).  Suppose that $\mathscr{H}(R) \equiv \Ga$. Since $\Fc_{\theta} \subset \partial_{\theta} X$ by Proposition \ref{prop:embedding of partial flags}, we can consider $f$ as the map into $\partial_{\theta} X$. We then consider the \emph{mixed shadow }
$$
\Oc_R^{f}(\ga) := \Oc_R(\ga) \cap f^{-1}(\Oc_R^{\theta}(\rho(\ga)))
$$
for $\ga \in \Ga$ and $R > 0$. By the part (1)  of Theorem \ref{thm:boundary map}, it suffices to show that for $\mu$-a.e.\ $x \in M$, there exist $R > 0$, $\ga \in \Ga$, and an escaping sequence $\{ \ga_n \} \subset \Ga$ such that $x \in \ga \Oc_R^{f}(\ga_n)$ for all $n \in \Nb$. Then the ``in particular'' part follows from Lemma \ref{lem:shadow translate shadow}(1).

By Property \ref{item:properness}, we can fix an enumeration $\Ga = \{ \ga_n \}$ such that
$$
\norm{\ga_1}_{\sigma} \le \norm{\ga_2}_{\sigma} \le \cdots.
$$
By the Mixed Shadow Lemma in our earlier work \cite[Theorem 6.2]{KimZimmer1} and the Shadow Lemma (Proposition \ref{prop.shadowlemma}), we can fix $R > 0$ so that for some $C > 1$, we have
$$
\frac{1}{C} \mu \left( \Oc_R(\ga) \right) \le \mu \left( \Oc_R^{f}(\ga) \right)  \le C \mu \left( \Oc_R(\ga) \right)  \quad \text{for all} \quad \ga \in \Ga. 
$$

By \cite[Proof of Theorem 4.1]{KimZimmer1}, the sequence of sets $A_n := \Oc_R(\ga_n)$, $n \in \Nb$, satisfies the following Kochen--Stone Lemma, with $\Omega = M$ and $\nu = \mu$.

\begin{lemma}[Kochen--Stone Lemma, \cite{KochenStone}]\label{lem: KS BC lemma}
Let $(\Omega ,\nu)$ be a finite measure space. If $\{ A_n\} \subset \Omega$ is a sequence of measurable sets where 
$$
\sum_{n \ge 1} \nu(A_n) = +\infty \quad \text{and} \quad \liminf_{N \rightarrow  +  \infty} \frac{ \sum_{n, m = 1}^N \nu(A_n \cap A_m)}{\left(\sum_{n=1}^N \nu(A_n)\right)^2} < +\infty,
$$
then  
$$
\nu\left( \{ x \in \Omega : x \text{ is contained in infinitely many of } A_1, A_2, \dots \}\right) > 0. 
$$
\end{lemma}

The uniform estimates on shadows implies that the new sequence of sets $B_n := \Oc_R^{f}(\ga_n)$, $n \in \Nb$, also satisfies this Kochen--Stone Lemma. Therefore, setting
$$
E := \left\{ x \in M : \exists \ R > 0, \ \ga \in \Ga, \text{ escaping } \{ \ga_n \} \subset \Ga \text{ s.t. } x \in \bigcap_{n \ge 1} \ga \Oc_R^{f}(\ga_n ) \right\},
$$
we have $\mu(E) > 0$ and $E$ is $\Ga$-invariant. Since the $\Ga$-action on $(M, \mu)$ is ergodic (Theorem \ref{thm:ergodicity KZ}), we have 
$
\mu(E) = 1.
$
This completes the proof of (3).

\medskip

Therefore, Theorem \ref{thm:boundary map} follows, and it remains to prove Proposition \ref{prop:key result in boundary map}.

\subsection{Proof of Proposition~\ref{prop:key result in boundary map}}

Let $M' \subset M$ be a full $\mu$-measure set satisfying the conclusion of Theorem~\ref{thm:weak continuity property of boundary maps}. Suppose $x_0 \in M'$, $\gamma \in \Gamma$, $\{\gamma_n\} \subset \Gamma$ is escaping,  $\gamma_n \in \mathscr{H}(R_n)$ for some $R_n \rightarrow +\infty$, 
$$
x_0 \in \bigcap_{n \geq 1}\gamma  \Oc_R(\gamma_n)
$$
for some $R > 0$, and 
$$
\frac{\rho_\alpha(\gamma \gamma_n)}{\norm{\rho_\alpha(\gamma \gamma_n)}} \rightarrow S_\alpha \in \End (V_\alpha)
$$
 for all $\alpha \in \theta$. 

Using Theorem~\ref{thm:weak continuity property of boundary maps} and replacing $\{\gamma_n\}$ by a subsequence, there exists a $\mu$-null set $E \subset M$ and $h_1,\dots, h_N \in \Gamma$ such that: if $y \in M \smallsetminus E$, then 
$$
\mathsf{m}_1(x_0) = \lim_{n \rightarrow  + \infty} \mathsf{m}_1( \gamma \gamma_n h_i y)
$$
for some $1 \leq i \leq N$. 

Given $T \in \End(V_\alpha)$ and a measure $\nu$ on $\Pb(V_\alpha)$, we define the pushforward measure 
$$
T_* \nu := T_*\left( \nu|_{\Pb(V_\alpha) \smallsetminus \Pb(\ker T)} \right). 
$$
The following well known observation will allow us to study the limit points of $\frac{\rho_\alpha(\gamma \gamma_n)}{\norm{\rho_\alpha(\gamma \gamma_n)}}$ in $\End(V_\alpha)$. 

\begin{observation}\label{obs:limits of endos acting on measures} Suppose $\nu$ is a Borel probability measure on $\Pb(\Rb^d)$ and $T_n \rightarrow T$ in $\End(\Rb^d)$. If $\nu( \Pb(\ker T)) = 0$, then $(T_n)_* \nu \rightarrow T_* \nu$.
\end{observation}

Since $\rho(\Ga)$ is $\Psf_{\theta}$-contracting,   there exists a sequence $\{g_m\} \subset \Gamma$ such that 
$$
\min_{\alpha \in \theta} \alpha(\kappa(\rho(g_m))) \rightarrow  + \infty. 
$$
Passing to a subsequence we can suppose that 
$$
\frac{\rho_\alpha(g_m)}{\norm{\rho_\alpha(g_m)}} \rightarrow T_\alpha \in \End (V_{\alpha})
$$
 for all $\alpha \in \theta$. By Property~\ref{item:SVs of Tits repn}, 
 $$
 \frac{ \sigma_1( \rho_\alpha(g_m))}{ \sigma_2( \rho_\alpha(g_m))} \rightarrow +\infty
 $$
 and hence $T_\alpha$ has rank one for each $\alpha \in \theta$.

Fix 
$$
y_0 \in M \smallsetminus \Gamma \cdot E
$$
(this set has full $\mu$-measure and hence is non-empty).

\begin{lemma}\label{lem:finding u1,u2}

   There exist $u_1,u_2 \in \Gamma$ such that 
$$
S_\alpha\rho_\alpha(h_i u_1) T_\alpha \neq 0 \quad \text{for all} \quad\alpha \in \theta, \ 1 \leq i \leq N
$$
and
$$
\rho(u_2) Y \not\subset \zeta_\alpha^{-1}\Big( \Pb(\ker S_\alpha\rho_\alpha(h_i u_1) T_\alpha)\Big) \quad \text{for all} \quad\alpha \in \theta, \  Y \in g_{y_0}\Ac_{\epsilon}.
$$
\end{lemma} 

\begin{proof} For each $\alpha \in \theta$ and $1 \leq i \leq N$, fix $v_\alpha \in {\rm im} \, T_\alpha$ non-zero. By Lemma \ref{lem:rho irr implies flag irr}, there exists $u_1 \in \Gamma$ such that 
$$
\rho_\alpha(u_1) v_\alpha \notin \rho_\alpha(h_i)^{-1} \ker S_\alpha
$$
for all $\alpha \in \theta$ and $1 \leq i \leq N$. Then $u_1$ satisfies the first part of the lemma. Next for each $Y \in g_{y_0}\Ac_{\epsilon}$, fix a non-zero vector $w_Y \in V_\alpha$ with $[w_Y] \in \zeta_\alpha (Y)$. Then by Lemma \ref{lem:rho irr implies flag irr}, there exists $u_2 \in \Gamma$ such that 
$$
\rho_\alpha(u_2) w_Y \notin \ker S_\alpha\rho_\alpha(h_i u_1) T_\alpha
$$
for all $\alpha \in \theta$ and $Y \in g_{y_0}\Ac_{\epsilon}$ (recall that $\Ac_\epsilon$ is finite). Then $u_2$ satisfies the second part of the lemma. 
\end{proof}

\medskip

For each $m \in \Nb$, since $u_1 g_m u_2 y_0 \notin E$, there exists some $1 \leq i_m \leq N$ such that 
$$
\lim_{n \rightarrow  + \infty} \rho(\gamma \gamma_n h_{i_m} u_1 g_m u_2)_*\mathsf{m}_1(y_0) =  \lim_{n \rightarrow  + \infty} \mathsf{m}_1(\gamma \gamma_n h_{i_m} u_1 g_m u_2y_0) = \mathsf{m}_1(x_0).
$$
Replacing $\{g_m\}$ with a subsequence and relabelling the $\{h_i\}$ we can assume that $h_{i_m} = h_1$ for all $m$. Next fix $n_m \rightarrow  + \infty$ such that 
$$
\lim_{m \rightarrow  + \infty} \rho(\gamma \gamma_{n_m} h_1 u_1 g_m u_2)_*\mathsf{m}_1(y_0)  = \mathsf{m}_1(x_0).
$$ 
Notice that 
$$
\lim_{m \rightarrow  + \infty} \frac{\rho_\alpha(\gamma\gamma_{n_m} h_1 u_1 g_m u_2)}{\norm{\rho_\alpha(\ga \gamma_{n_m} h_1 u_1 g_m u_2)}} = \frac{S_\alpha \rho_\alpha(h_1 u_1) T_\alpha \rho_\alpha(u_2)}{\norm{S_\alpha \rho_\alpha(h_1 u_1) T_\alpha \rho_\alpha(u_2)}}=:T'_\alpha \in \End(V_\alpha),  
$$
since $S_\alpha \rho_\alpha(h_1u_1) T_\alpha \neq 0$. 

By the choice of $u_2$, if $Y \in g_{y_0} \Ac_{\epsilon}$  and $\alpha \in \theta$, then 
$$
Y \not\subset \zeta_\alpha^{-1}\left( \Pb(\ker T_\alpha')\right),
$$
and so 
$$
\dim \left( Y \cap \zeta_\alpha^{-1}\left(\Pb(\ker T_\alpha')\right)\right) < d_0.
$$
 Hence by the definition of $\mathsf{m}_1$,
$$
(\zeta_\alpha)_*\mathsf{m}_1(y_0)\left( \Pb(\ker T_\alpha') \right) =\mathsf{m}_1(y_0)\left( \zeta_\alpha^{-1}\left( \Pb(\ker T_\alpha')\right) \right) =0.
$$
So Observation~\ref{obs:limits of endos acting on measures}  implies that
\begin{align*}
(\zeta_\alpha)_* \mathsf{m}_1(x_0) & = \lim_{m \rightarrow  + \infty} (\zeta_\alpha )_*  \rho(\gamma \gamma_{n_m} h_1 u_1 g_{m} u_2)_*\mathsf{m}_1(y_0) \\
& = \lim_{m \rightarrow  + \infty} \rho_\alpha (\gamma \gamma_{n_m}  h_1 u_1 g_{m} h_2)_*(\zeta_\alpha )_* \mathsf{m}_1(y_0) \\
& = ( T_\alpha')_* (\zeta_\alpha)_* \mathsf{m}_1(y_0).
\end{align*}
Since $\rank T_\alpha' = \rank T_\alpha=1$ we then see that 
$$
(\zeta_\alpha)_* \mathsf{m}_1(x_0) = \Dc_{\Pb({\rm im} T_\alpha')}. 
$$

Thus $\mathsf{m}_1(x_0)=\Dc_\xi$ where $\xi \in \Fc_\theta$ is the unique point with $\zeta_\alpha (\xi) = \Pb({\rm im} \, T_\alpha')$ for all $\alpha \in \theta$. Since ${\rm im} \, T_\alpha' \subset {\rm im} \, S_\alpha$, to finish the proof it suffices to show that $\rank S_\alpha =1$. 

Fix $\alpha \in \theta$. Since $\mathsf{m}_1(x_0)$ is a Dirac mass, the definition of $\mathsf{m}_1$ implies that $\mathsf{m}_1(x)$ is a Dirac mass for $\mu$-a.e.\ $x$. Then $f_\alpha(x) := \supp\, (\zeta_\alpha)_* \mathsf{m}_1(x)$ defines a  $\rho$-equivariant $\mu$-a.e.\ defined measurable map $f_\alpha : M \rightarrow \Pb(V_\alpha)$.

Suppose for a contradiction that $\rank S_\alpha > 1$. Fix $z_0 \in M \smallsetminus \Ga \cdot E$ such that $f_{\alpha}(z_0)$ is defined.

\begin{lemma}    
   There exists $u_3 \in \Gamma$ such that 
$$
\rho_\alpha(h_i u_3) f_\alpha(z_0) \notin \Pb(\ker S_\alpha )
$$
and 
$$
S_\alpha \rho_\alpha(h_i u_3)f_\alpha(z_0) \neq  f_\alpha(x_0) 
$$
for all $1 \leq i \leq N$. 
 \end{lemma}  
 
 \begin{proof}Let $f_\alpha(z_0)=[v]$ for some non-zero $v \in V_\alpha$. Notice that since $S_\alpha$ has rank at least two, $S_\alpha^{-1}(f_\alpha(x_0) ) \subset V_\alpha$ is a proper linear subspace. So by Lemma \ref{lem:rho irr implies flag irr}, there exists $u_3 \in \Gamma$ such that 
$$
\rho_\alpha(u_3) v \notin \rho_\alpha(h_i)^{-1} \ker S_\alpha 
$$
and 
$$
 \rho_\alpha(u_3)v \notin   \rho_\alpha(h_i)^{-1} S_\alpha^{-1}( f_\alpha(x_0) ) 
$$
for all $1 \leq i \leq N$. Then $u_3$ satisfies the desired conclusion. 
 \end{proof} 
 
 Since $u_3 z_0 \notin E$, there exists $1 \leq i \leq N$ such that 
 $$
\mathsf{m}_1(x_0)=  \lim_{n \rightarrow +  \infty} \mathsf{m}_1(\gamma \gamma_n h_i  u_3z_0)=\lim_{n \rightarrow  + \infty} \rho(\gamma \gamma_n h_i  u_3)_*\mathsf{m}_1(z_0).
$$
Applying $(\zeta_\alpha )_*$, Observation~\ref{obs:limits of endos acting on measures} implies that
\begin{align*}
\Dc_{f_\alpha(x_0)} & = (\zeta_\alpha )_*\mathsf{m}_1(x_0)=\lim_{n \rightarrow  + \infty} \rho_\alpha(\gamma \gamma_n h_i u_3)_*\Dc_{f_\alpha(z_0)} \\
& = \Dc_{S_\alpha  \rho_\alpha(h_i u_3)f_\alpha(z_0)}.
\end{align*} 
Since $S_\alpha \rho_\alpha(h_iu_3)f_\alpha(z_0) \neq  f_\alpha(x_0)$,  we have a contradiction. Hence $\rank S_\alpha =1$, completing the proof of Proposition~\ref{prop:key result in boundary map}.
\qed


\section{Amenable actions of transverse groups}\label{sec:amenable actions of transverse groups}


In this section we prove that transverse groups act amenably on their limit sets, and hence  the amenable assumption in Theorem \ref{thm:boundary map} is satisfied.

\begin{theorem}\label{thm:amenable} Suppose $\Gamma < \Gsf$ is a non-elementary $\Psf_\theta$-transverse group. If $\mu$ is a $\Gamma$-quasi-invariant Borel probability measure on $\Lambda_\theta(\Gamma)$, then the $\Gamma$-action on $(\Lambda_\theta(\Gamma), \mu)$ is amenable. 
\end{theorem} 

To prove Theorem~\ref{thm:amenable} we first use a result from~\cite{CZZ2024} that says a transverse group can be identified with a group acting nicely on a properly convex domain in some projective space. The Hilbert metric on this properly convex domain then has enough hyperbolic behavior to adapt an argument of Kaimanovich~\cite{Kaimanovich2004}.

\subsection{Background on convex real projective geometry}

We recall some terminology from real projective geometry. 

Suppose $\Omega \subset \Pb(\Rb^d)$ is \emph{properly convex}, that is, it is a bounded convex subset of some affine chart of $\Pb(\Rb^d)$. The \emph{automorphism group} of $\Omega$ is 
$$
\Aut(\Omega) = \{ g \in \PGL(d,\Rb) : g \Omega = \Omega\}.
$$
The \emph{limit set} of a subgroup $H < \Aut(\Omega)$, denoted by $\Lambda_\Omega(H)$, is the set of points $x \in \partial \Omega$ where there exist $o \in \Omega$ and $\{h_n\} \subset H$ with $h_n o \rightarrow x$. 

Given $x,y \in \overline\Omega$, we let $[x,y]$ denote the projective line segment contained in $\overline\Omega$ joining $x$ to $y$. We further define $(x,y) : =[x,y] \smallsetminus \{x,y\}$, $[x,y) : =[x,y] \smallsetminus \{y\}$, and $(x,y] : =[x,y] \smallsetminus \{x\}$.

The \emph{Hilbert metric} $\dist_\Omega$ on $\Omega$ is defined as follows: for $p,q \in \Omega$ distinct let $a,b \in \partial \Omega$ be the unique points with $p,q \in (a,b)$ and with order $a,p,q,b$, then define 
$$
\dist_\Omega(x,y) := \frac{1}{2} \log \frac{\norm{a-q}\norm{b-p}}{\norm{a-p}\norm{b-q}}
$$
where $\norm{\cdot}$ is any norm on any affine chart containing $\overline{\Omega}$. The Hilbert metric is a proper geodesic metric where $\Aut(\Omega)$ acts by isometries. Further, for $p,q \in \Omega$ the line segment $[p,q]$ can be parametrized to be a geodesic.

An element $W \in \Gr_{d-1}(\Rb^d)$ 
is a \emph{supporting hyperplane} at $x \in \partial \Omega$ if $x \in \Pb(W)$ and $\Pb(W) \cap \Omega = \emptyset$. Convexity of $\Omega$ implies that every boundary point is contained in at least one supporting hyperplane and we say that $\partial \Omega$ is \emph{$\Cc^1$-smooth at $x \in \partial\Omega$} if $x$ is contained in exactly one supporting hyperplane.  

We end this section by recording some useful geometric properties. 

\begin{observation}[{see e.g.\ the proof of ~\cite[Lemma 3.4]{ConvexDivisibleI}}]\label{obs:decay of geodesics} Suppose $x \in \partial \Omega$ is a $\Cc^1$-smooth point. For $p \in \Omega$, let $\ell_p: [0, + \infty) \rightarrow \Omega$ denote the unit speed parametrization of the geodesic ray $[p,x)$. If $p,q \in \Omega$, then there exists $T_{p,q} \in \Rb$ such that 
$$
\lim_{t \rightarrow  + \infty} \dist_\Omega(\ell_p(t),\ell_q(t+T_{p,q})) = 0.
$$
Moreover $T_{p,q}$ depends continuously on $p,q$ and the convergence is uniform on compact subsets of $\Omega$. 
\end{observation} 

Given a discrete group $\Gamma < \Aut(\Omega)$, the \emph{Hilbert metric critical exponent} is 
$$
\delta_\Omega(\Gamma) : = \limsup_{R \rightarrow  + \infty} \frac{1}{R} \log \#\{ \gamma \in \Gamma : \dist_\Omega(\gamma o, o) \leq R\}
$$
where $o \in \Omega$ is any fixed point. A result of Tholozan implies that this critical exponent has a uniform upper bound, which only depends on the dimension $d$. 

\begin{theorem}[{\cite{Tholozan}}] If $\Gamma < \Aut(\Omega)$ is discrete, then $\delta_\Omega(\Gamma) \leq d-2$. \end{theorem}

\subsection{Projectively visible groups} \label{sec:projectively visible} 

In ~\cite{CZZ2024,CZZ2026}, Canary, Zhang, and the second author studied transverse groups by constructing actions on certain types of properly convex domains. 

\begin{definition} Suppose $\Omega \subset \Pb(\Rb^d)$ is properly convex. A discrete subgroup $\Gamma < \Aut(\Omega)$ is \emph{projectively visible} if 
\begin{itemize} 
\item $(x,y) \subset \Omega$ for every $x,y \in \Lambda_\Omega(\Gamma)$, 
\item every $x \in  \Lambda_\Omega(\Gamma)$ is a $\Cc^1$-smooth point of $\partial \Omega$. 
\end{itemize} 
\end{definition} 

Under some mild conditions on $\Gsf$ and $\Psf_\theta$, every transverse group can be identified with a projectively visible subgroup. 

\begin{theorem}[{\cite[Theorem 6.2]{CZZ2024}}] \label{thm:transverse have proj models} Suppose $\Gsf$ has trivial center, $\theta \subset \Delta$,  
   and $\Psf_\theta$ contains no simple factors of $\Gsf$. If $\Gamma < \Gsf$ is $\Psf_\theta$-transverse, then there   exist $d \in \Nb$, a properly convex domain $\Omega \subset \Pb(\Rb^d)$, a projectively visible subgroup $\Gamma_0 < \Aut(\Omega)$, an isomorphism $\rho : \Gamma_0 \rightarrow \Gamma$, and a $\rho$-equivariant homeomorphism $\xi : \Lambda_\Omega(\Gamma_0) \rightarrow \Lambda_\theta(\Gamma)$. 
\end{theorem}

\subsection{Proof of Theorem~\ref{thm:amenable}} 
By replacing $\Gsf$ with a quotient, it suffices to consider the case where $\Gsf$ has trivial center and $\Psf_\theta$ does not contain any simple factors of $\Gsf$, see~\cite[Section 2.4]{CZZ2024}. 

Then by Theorem~\ref{thm:transverse have proj models} there exist $d \in \Nb$, a properly convex domain $\Omega \subset \Pb(\Rb^d)$, a projectively visible subgroup $\Gamma_0 < \Aut(\Omega)$, an isomorphism $\rho : \Gamma_0 \rightarrow \Gamma$, and a $\rho$-equivariant homeomorphism $\xi : \Lambda_\Omega(\Gamma_0) \rightarrow \Lambda_\theta(\Gamma)$. 

Then it suffices to fix a $\Gamma_0$-quasi-invariant Borel probability measure $\mu$ on $\Lambda_\Omega(\Gamma_0)$ and show that $\Gamma_0$ acts amenably on $(\Lambda_\Omega(\Gamma_0),\mu)$. Using Observation~\ref{obs:decay of geodesics}, we can argue exactly as in ~\cite[Theorems 1.33 and 3.15]{Kaimanovich2004}. 

\begin{lemma}[compare to {\cite[Theorem 1.33, Theorem 1.38]{Kaimanovich2004}}] There exists a sequence of $\Gamma_0$-equivariant maps $\lambda_n : \Omega \times \Lambda_\Omega(\Gamma_0) \rightarrow \Prob(\Gamma_0)$ such that 
$$
\lim_{n \rightarrow  + \infty} \norm{\lambda_n(p, x) - \lambda_n(q,x)} = 0
$$
for any $p,q \in \Omega$ and $x \in \Lambda_\Omega(\Gamma_0)$. Moreover, with $x$ fixed the convergence is uniform on any compact subset of $\Omega$. 
\end{lemma}  

\begin{proof} Fix $\delta > \delta_\Omega(\Gamma_0)$ and $o \in \Omega$. Then for $p \in \Omega$ consider the measures 
$$
\nu_p := \frac{1}{\sum_{\gamma \in \Gamma_0} e^{-\delta \dist_\Omega(p, \gamma o)}} \sum_{\gamma \in \Gamma_0} e^{-\delta \dist_\Omega(p, \gamma o)} \Dc_{\gamma} \in {\rm Prob}(\Gamma_0)
$$
where $\Dc_{\gamma }$ is the Dirac mass  at $\gamma $. 

\medskip

\noindent \textbf{Claim:} If $\dist_\Omega(p,q) \leq 1$, then 
\begin{equation}\label{eqn:difference between nu measures}
\norm{\nu_p - \nu_q } \leq 2\delta e^{2\delta} \dist_\Omega(p,q).
\end{equation} 

\medskip 

\noindent \emph{Proof of Claim:} For $p \in \Omega$, let $M_p : = \sum_{\gamma \in \Gamma_0} e^{-\delta \dist_\Omega(p, \gamma o)} $. Now fix $p,q \in \Omega$ with $\dist_\Omega(p,q) \leq 1$. Then  for every $\gamma \in \Gamma$, 
$$
e^{-\delta \dist_\Omega(p, \gamma o)} = e^{\epsilon_\gamma}e^{-\delta \dist_\Omega(q, \gamma o)} \quad \text{where} \quad\abs{\epsilon_\gamma} \leq \delta \dist_\Omega(p,q).
$$
So $M_p = e^{\epsilon} M_q$ where $\abs{\epsilon} \leq \delta \dist_\Omega(p,q)$. Then 
\begin{align*}
\norm{ \nu_p-\nu_q} & = \sum_{\gamma \in \Gamma} \abs{ \frac{1}{M_p}e^{-\delta \dist_\Omega(p, \gamma o)} - \frac{1}{M_q}e^{-\delta \dist_\Omega(q, \gamma o)} } \\
& = \frac{1}{M_q}  \sum_{\gamma \in \Gamma} \abs{ e^{-\epsilon} e^{-\delta \dist_\Omega(p, \gamma o)} - e^{-\delta \dist_\Omega(q, \gamma o)} }\\
&  = \frac{1}{M_q}  \sum_{\gamma \in \Gamma} e^{-\delta \dist_\Omega(q, \gamma o)} \abs{ e^{-\epsilon+\epsilon_\gamma}  - 1  } \leq \sup_{\gamma \in \Gamma}\abs{ e^{-\epsilon+\epsilon_\gamma}  - 1  } .
\end{align*} 
Now $\abs{-\epsilon+\epsilon_\gamma} \leq 2 \delta \dist_\Omega(p,q)$, so by the mean value theorem 
$$
\norm{ \nu_p-\nu_q}  \leq e^{2\delta \dist_\Omega(p,q)} 2 \delta \dist_\Omega(p,q) \leq 2 \delta e^{2\delta} \dist_\Omega(p,q). 
$$
\hfill $\blacktriangleleft$

\medskip

Next for $p \in \Omega$ and $x \in \Lambda_\Omega(\Gamma_0)$, let $\ell_{px} : [0, + \infty) \rightarrow \Omega$ be the unit speed parametrization of the geodesic ray $[p,x)$. Then define 
$$
\lambda_n(p,x) : = \frac{1}{n} \int_0^n \nu_{\ell_{px}(t)} d t. 
$$
Observation~\ref{obs:decay of geodesics} and Equation~\eqref{eqn:difference between nu measures} imply that these measures have the desired properties. 
\end{proof} 

Next define $\lambda_n : \Lambda_\Omega(\Gamma_0) \rightarrow \Prob(\Gamma_0)$ by 
$$
\lambda_n(x) := \lambda_n(o,x).
$$
Then
$$
\lim_{n \rightarrow  + \infty} \norm{\lambda_n(\gamma x) - \gamma_*\lambda_n(x)} = \lim_{n \rightarrow  + \infty} \norm{\lambda_n(o,\gamma x) - \lambda_n(\gamma o, \gamma x)}=0 
$$
for every $x \in \Lambda_{\Omega}(\Gamma_0)$ and every $\ga \in \Ga_0$.
\qed

\part{Applications} 


\section{Lifting Patterson--Sullivan measures} \label{section:lifting map}


For the rest of the section, suppose $\Gamma < \Gsf$ is a $\Psf_\theta$-transverse group, $\phi \in \mfa_\theta^*$, $\delta:=\delta^\phi(\Gamma) < +\infty$, and 
$$
\sum_{\gamma \in \Gamma} e^{-\delta \phi(\kappa(\gamma))} = +\infty.
$$
 Let $\mu$ denote the unique $(\Ga, \phi, \delta)$-Patterson--Sullivan measure on $\La_{\theta}(\Ga) \subset \Fc_{\theta} \subset \partial_{\theta} X$, see Theorem~\ref{thm:CZZ PS transverse}. This section is devoted to the proof of the following.

\begin{theorem} \label{thm:lifting properties}
With the notations above, suppose $\Theta \supset \theta$ and $\Gamma $ is $\Psf_{\Theta}$-contracting and strongly $(\Phi_{\alpha})_{\alpha \in \Theta}$-irreducible. Then the following holds:
\begin{enumerate}
\item There exists a unique $\mu$-a.e defined injective $\Gamma$-equivariant measurable map $$f : \Lambda_\theta(\Gamma) \rightarrow \Fc_{\Theta}.$$
\item The pushforward $f_*\mu$ is the unique $(\Ga, \phi, \delta)$-Patterson--Sullivan measure on $\partial_\Theta X$. 
\item $(f_*\mu)( \Lambda^{\rm concon}_\Theta(\Gamma)) = 1$.  Moreover,
for $\mu$-a.e.\ $x \in \La_{\theta}(\Ga)$, there exist $R > 0$ and an escaping sequence $\{ \ga_n \} \subset \Ga$ such that
$$
\lim_{n \to + \infty} \min_{\alpha \in \Theta} \alpha(\kappa(\ga_n)) = + \infty, \quad x \in \bigcap_{n \ge 1} \Oc_R^{\theta}(\ga_n), \quad \text{and} \quad f(x) \in \bigcap_{n \ge 1} \Oc_R^{\Theta}(\ga_n).
$$

\item If $\mu'$ is a $(\Gamma, \phi, \beta)$-Patterson--Sullivan measure on $\partial_\Theta X$, then $\beta \geq \delta$. 
\end{enumerate}

\end{theorem}

By Theorem \ref{thm:transverse well-behaved}, $(\partial_\theta X, \Ga, \phi \circ B_{\theta}, \mu)$ is a well-behaved PS-system with respect to the trivial hierarchy $\mathscr{H}(R) \equiv \Ga$. Hence, together with Theorem \ref{thm:KZ Rigidity} and \cite{Benoist_properties}, we obtain singularity between Patterson--Sullivan measures.

\begin{corollary} \label{cor:singularity between PS}
   Suppose further that $\Ga$ is Zariski dense in $\Gsf$. If $\psi \in \fa_{\Theta}^*$, $\beta \ge 0$, and $\mu_{\psi}$ is a $(\Ga, \psi, \beta)$-Patterson--Sullivan measure on $\partial_{\Theta} X$, then
   $$
\text{$f_* \mu$ and $\mu_\psi$ are non-singular} \quad \Longleftrightarrow \quad f_* \mu = \mu_\psi \quad \Longleftrightarrow \quad \delta \cdot \phi = \beta \cdot \psi.
   $$
\end{corollary}

\begin{proof}
If $\delta \cdot \phi = \beta \cdot \psi$, then it follows from the uniqueness in Theorem \ref{thm:lifting properties}(2) that $f_* \mu = \mu_{\psi}$. Clearly, if $f_* \mu = \mu_{\psi}$, then they are non-singular. So it suffices to show that non-singularity implies that the functionals coincide up to an appropriate scaling. 

Suppose $f_* \mu$ and $\mu_\psi$ are non-singular. Then by Theorem \ref{thm:KZ Rigidity} we have 
$$
\sup_{\ga \in \Ga} \abs{ \delta \cdot \phi(\kappa(\ga)) - \beta \cdot \psi (\kappa(\ga))} < + \infty.
$$
By definition, the Jordan projections then satisfy 
$$
\delta \cdot \phi(\lambda(\ga)) =  \beta \cdot \psi (\lambda(\ga))
$$
for all $\gamma \in \Gamma$. Since $\Ga$ is Zariski dense, this implies $\delta \cdot \phi = \beta \cdot \psi$ by \cite{Benoist_properties}. 
\end{proof}

The rest of this section is devoted to the proof of Theorem \ref{thm:lifting properties}.

\subsection{Part (1)} Theorem \ref{thm:amenable} implies that the $\Ga$-action on $(\La_{\theta}(\Gamma), \mu)$ is amenable and Theorem~\ref{thm:CZZ PS transverse} (and Proposition \ref{prop:concon limit set for transverse groups}) implies that $\mu$ is supported on the conical limit set.
So we can apply Theorem \ref{thm:boundary map} to the inclusion $\Ga \hookrightarrow \Gsf$. As a result, we have a unique measurable $\Ga$-equivariant $\mu$-a.e.\ defined map
$$
f : \La_{\theta}(\Ga) \to \Fc_\Theta.
$$
 To finish the proof of part (1) we only need to show that $f$ is injective on a set of full $\mu$-measure. Let $\pi : \Fc_\Theta \rightarrow \Fc_\theta$ be the natural projection.  Then $\pi f$ is a $\Ga$-equivariant $\mu$-a.e.\ defined map. 
However, $\Gamma$ is $\Psf_\theta$-contracting and strongly $(\Phi_{\alpha})_{\alpha \in \theta}$-irreducible, so by Theorem \ref{thm:boundary map} the identity is the unique measurable $\Ga$-equivariant $\mu$-a.e.\ defined map
$
\La_{\theta}(\Ga) \to \Fc_\theta.
$
Thus we must have 
\begin{equation}\label{eqn:f is a section}
\pi f = \id_{\Fc_\theta} \quad \text{$\mu$-a.e. }
\end{equation} 
Hence $f$ is injective on a set of full $\mu$-measure.

\subsection{Parts (2) and (4)} Since $\phi \in \mfa_\theta^*={\rm span}\{\omega_\alpha : \alpha \in \theta\}$, Equation~\eqref{eqn:defn of pi_theta to mfa_theta}  implies that 
$$
\phi B_\Theta^{IW}(g,x)  = \phi B_\theta^{IW}(g,\pi(x))
$$
for all $x \in \Fc_\Theta$ and $g \in \Gsf$. Hence Equation~\eqref{eqn:f is a section}  implies that $f_*\mu$ is a $(\Gamma, \phi, \delta)$-Patterson--Sullivan measure on $\Fc_\Theta \subset \partial_\Theta X$. 

The uniqueness in part (2) and the estimate in part (4) will be a consequence of the following general result (which we also use in the proof of Theorem \ref{thm:convexity body} below).

\begin{proposition} \label{prop:uniquePS}
   Suppose that  $\mu'$ is a Radon measure on $\partial_{\Theta} X$ such that for each $R > 0$ sufficiently large there exists $C_0 = C_0(R) > 0$ so that
   $$
   (f_* \mu)(\Oc_R^{\Theta}(\ga)) \le C_0 \mu' (\Oc_R^{\Theta}(\ga)) \quad \text{for all} \quad \ga \in \Ga.
   $$
   Then
   $$
   f_*\mu \ll \mu'.
   $$
\end{proposition}

\begin{proof} 

Since $f_* \mu$ is supported on $\Fc_{\Theta}$, it suffices to fix a Borel subset $E \subset \Fc_{\Theta}$ with $\mu'(E)=0$ and then show that $f_*\mu(E)=0$.  

For each $R > 0$, let $E_R \subset f^{-1}(E)$ be the subset satisfying:
\begin{itemize}
   \item for each $x \in E_R$, there exists an escaping sequence $\{\ga_n\} \subset \Ga$ such that
   $$
   x \in \bigcap_{n \ge 1} \Oc_R^{\theta}(\ga_n)
   $$
   and

   \item Theorem \ref{thm:boundary map}(1) holds for all $x \in E_R$.
\end{itemize}
Then Proposition~\ref{prop:concon limit set for transverse groups} implies that  $(f_* \mu)(E) = \mu\left( \bigcup_{R > 0} E_R \right)$. Since $\{E_R\}_{R > 0}$ is an increasing family of sets, it suffices to fix $R > 0$ sufficiently large and show that $\mu(E_R) =0$.

Fix $\epsilon > 0$ and let $\Uc \subset \partial_{\Theta} X$ be an open neighborhood of $E$ such that
$$
\mu'(\Uc) \le  \epsilon.
$$

\noindent \textbf{Claim:} If $x \in E_R$, then there exists $\gamma_x \in \Gamma$ so that 
$
x \in \Oc_R^\theta(\ga_x)$ and $\Oc_R^{\Theta}(\ga_x) \subset \Uc.
$

\medskip 

\noindent \emph{Proof of Claim:} By the definition of $E_R$ and Theorem \ref{thm:boundary map}(1), there exists an escaping sequence $\{ \ga_n \} \subset \Ga$ such that
$$
x \in \bigcap_{n \ge 1} \Oc_R^{\theta}(\ga_n), \quad \lim_{n \to + \infty} \min_{\alpha \in \Theta} \alpha(\kappa(\ga_n)) = + \infty, \quad \text{and} \quad f(x) =  \lim_{n \to + \infty} U_{\Theta}(\ga_n) \in \Fc_{\Theta}.
$$
By Lemma \ref{lem:endpoint in shadow}(2), $U_{\Theta}(\ga_n) \in \Oc_R^{\Theta}(\ga_n)$ for all large $n \ge 1$. Moreover, $\diam \Oc_R^{\Theta}(\ga_n) \to 0$ as $n \to + \infty$ by Lemma \ref{lem:diam decay}(3). Since $\Uc \subset \partial_{\Theta} X$ is an open neighborhood of $f(x)$, it follows that there exists $\ga_x \in \Ga$ such that 
$$x \in \Oc_R^{\theta}(\ga_x) \quad \text{and} \quad \Oc_R^{\Theta}(\ga_x) \subset \Uc.$$
\hfill $\blacktriangleleft$

Hence we have a countable subset $I := \{ \ga_x : x \in E_R \} \subset \Ga$ so that 
$$
E_R \subset \bigcup_{\ga \in I} \Oc_R^{\theta}(\ga) \quad \text{and} \quad \bigcup_{\ga \in I} \Oc_R^{\Theta}(\ga) \subset \Uc.
$$
By Lemma \ref{lem:V covering}, there exist $J \subset I$ and $R' \ge R$ such that the shadows $\left\{ \Oc_R^{\theta}(\ga) : \ga \in J \right\}$ are pairwise disjoint and
$$
\bigcup_{\ga \in I} \Oc_R^{\theta}(\ga) \subset \bigcup_{\ga \in J} \Oc_{R'}^{\theta}(\ga).
$$

We then have 
$$
\mu(E_R) \le \mu \left( \bigcup_{\ga \in I} \Oc_R^{\theta}(\ga) \right) \le \sum_{\ga \in J} \mu( \Oc_{R'}^{\theta}(\ga)).
$$
By Theorem \ref{thm:Zdense PS}, $(\partial_{\theta} X,  \Ga, \phi \circ B_{\theta}, \mu)$ and $(\partial_{\Theta} X,  \Ga, \phi \circ B_{\Theta}, f_*\mu)$ are PS-systems. So, after possibly increasing $R$, the Shadow Lemma (Proposition \ref{prop.shadowlemma}) implies that there exists $C_1 = C_1(R) > 0 $ so that
$$
\mu(\Oc_{R'}^{\theta}(\ga)) \le C_1 \mu(\Oc_R^{\theta}(\ga)) 
$$
and 
$$
\mu(\Oc_{R}^{\theta}(\ga)) \le C_1 (f_*\mu)(\Oc_R^{\Theta}(\ga)) 
$$
for all  $\ga \in \Ga$. 
This implies that
$$
\mu(E_R) \le C_1 \sum_{\ga \in J} \mu( \Oc_R^{\theta}(\ga)) \le C_1^2 \sum_{\ga \in J} (f_*\mu)( \Oc_R^{\Theta}(\ga)) \le C_0C_1^2 \sum_{\ga \in J} \mu'( \Oc_R^{\Theta}(\ga)).
$$

Let $\pi_{\theta} : \partial_{\Theta} X \to \partial_{\theta} X$ denote the map obtained by postcomposing with the canonical  projection $\pi_{\theta} : \fa_{\Delta} \to \fa_{\theta}$ restricted to $\fa_\Theta$. Equation~\eqref{eqn:defn of pi_theta to mfa_theta}  implies that $\pi_{\theta} \Oc_R^{\Theta}(\ga) \subset \Oc_R^{\theta}(\ga)$. Then, since the shadows $\left\{ \Oc_R^{\theta}(\ga) : \ga \in J \right\}$ are pairwise disjoint,  the shadows $\left\{ \Oc_R^{\Theta}(\ga) : \ga \in J \right\}$ are pairwise disjoint as well. Therefore,
$$\begin{aligned}
\mu(E_R) & \le C_0C_1^2 \mu' \left( \bigcup_{\ga \in J} \Oc_R^{\Theta}(\ga) \right)  \le C_0C_1^2  \mu'(\Uc)  \le  C_0C_1^2 \epsilon.
\end{aligned}
$$
Since  $\epsilon > 0$ is arbitrary and $C_0,C_1$ are independent of $\epsilon$, we have 
$$
\mu(E_R) = 0.
$$
This finishes the proof.
\end{proof}

Now we deduce the uniqueness in part (2) and the estimate in part (4) using Proposition \ref{prop:uniquePS}. 

\begin{lemma} If $\beta \leq \delta$ and $\mu'$ is a $(\Ga, \phi, \beta)$-Patterson--Sullivan measure on $\partial_{\Theta} X$, then $\beta = \delta$ and $\mu' = \mu$. 
\end{lemma} 

\begin{proof} 
 Since both $(\partial_{\Theta}X, \Ga, \phi \circ B_{\Theta}, f_*\mu)$ and $(\partial_{\Theta}X, \Ga, \phi \circ B_{\Theta}, \mu')$  are PS-systems by Theorem \ref{thm:Zdense PS}, it follows from the Shadow Lemma (Proposition \ref{prop.shadowlemma}) that for all large enough $R > 0$, there exists $C_0 > 0$ such that 
$$
   (f_* \mu)(\Oc_R^{\Theta}(\ga)) \le C_0 \mu' (\Oc_R^{\Theta}(\ga)) \quad \text{for all} \quad \ga \in \Ga.
   $$
   Hence, by Proposition \ref{prop:uniquePS}, we have 
$
f_* \mu \ll \mu'.
$

Fix a full $\mu$-measure $\Gamma$-invariant set $Y \subset \Lambda_\theta(\Gamma)$ where $f$ is defined. We claim that $\mu'(f(Y)) = 1$. If not, then 
$$
\mu'' : = \frac{1}{\mu'(\partial_{\Theta} X \smallsetminus f(Y)) } \mu'|_{\partial_{\Theta} X \smallsetminus f(Y)}
$$
is a $(\Ga, \phi, \beta)$-Patterson--Sullivan measure. Then arguing as above, we have $f_* \mu \ll \mu''$, which is impossible. Hence $\mu'(f(Y)) = 1$.

Recall that $\pi f = \id_{\Fc_\theta}$, see Equation~\eqref{eqn:f is a section}. So  $\pi_* \mu'$ is a $(\Ga, \phi, \beta)$-Patterson--Sullivan measure on $\Lambda_\theta(\Gamma)$. Then $\beta =\delta$ and $\mu = f^{-1}_* \mu'$ by  \cite{CZZ2024} (see Theorem \ref{thm:CZZ PS transverse existence}).
\end{proof} 

\subsection{Part (3)}  \label{subsection:contracting limit set full measure}

This follows from the part (3) of Theorem \ref{thm:boundary map} and Lemma \ref{lem:shadow translate shadow}(1), as the associated hierarchy $\mathscr{H}$ is trivial (Theorem \ref{thm:transverse well-behaved}). 
\qed


\section{Ergodic dichotomy for BMS-measures on homogeneous~spaces} \label{section:BMS}


In this section, we prove the ergodic dichotomy for the diagonal action on homogeneous spaces.
Let $\Ga < \Gsf$ be a $\Psf_{\theta}$-transverse group which is $\Psf_\Delta$-contracting and strongly $(\Phi_{\alpha})_{\alpha \in \Delta}$-irreducible. In this section, we apply the machinery developed in earlier sections to study ergodic theory on the homogeneous space $\Ga \ba \Gsf$ or $\Ga \ba \Gsf / \Msf$.

First recall the Hopf parametrization $\Gsf / \Msf = \Fc_{\Delta}^{(2)} \times \fa$ given by
$$
g \Msf = \left(g \Psf_{\Delta}, g w_0 \Psf_{\Delta}, B_{\Delta}^{IW}(g, \Psf_{\Delta}) \right)
$$ 
where $\Fc_{\Delta}^{(2)} \subset \Fc_{\Delta} \times \Fc_{\Delta}$ is the subset of transverse pairs.
Then right multiplication $\Asf$ on $\Gsf / \Msf$ corresponds to translation on the $\fa$-component of $\Fc_{\Delta}^{(2)} \times \fa$.

Fix $\phi \in \fa_{\theta}^*$ and $\delta \ge 0$, and suppose that there exist $(\Ga, \phi, \delta)$ and $(\Ga, \opp^*\phi , \delta)$-PS measures  $\mu_{\phi}$ and $\mu_{\opp^*\phi }$ on $\La_{\Delta}(\Ga) \subset \Fc_{\Delta}$ respectively. While they may not exist or they may not be unique in general, we assume the existence and make a choice.

We first define a $\Ga$-invariant Radon measure $\nu_{\phi}$ on $\Fc_{\Delta}^{(2)}$ by 
$$
d \nu_{\phi}(\xi, \eta) := e^{ \delta \phi \Gc_{\Delta}(\xi, \eta)} d \mu_{\phi}(\xi) d \mu_{\opp^*\phi }(\eta)
$$
where $ \Gc_{\Delta}$ is the \emph{$\fa$-valued Gromov product}, i.e.
\begin{equation} \label{eqn:Gromov product}
   \Gc_{\Delta}(\xi, \eta) :=  - \left( B_{\Delta}^{IW}(g^{-1}, \xi) + \opp B (g^{-1}, \eta) \right)
\end{equation} for $g \in \Gsf$ with $(g  \Psf_{\Delta}, g w_0 \Psf_{\Delta} ) = (\xi, \eta)$. By the irreducibility of $\Ga$ and the quasi-invariance of $\mu_{\phi}$ under the $\Ga$-action, it follows from Fubini's theorem that
\begin{equation} \label{eqn:positive product meaasure}
\nu_{\phi}(\Fc_{\Delta}^{(2)}) > 0.
\end{equation}
See \cite[Proposition 10.2]{KOW_PD} for instance.

Then the measure $\nu_{\phi} \otimes {\rm Leb}_{\fa}$ on $\Gsf / \Msf = \Fc_{\Delta}^{(2)} \times \fa$ is $\Ga$-invariant, and hence it induces the $\Asf$-invariant Radon measure
$$
\mathsf{m}_{\phi} \quad \text{on} \quad \Ga \ba \Gsf / \Msf
$$
which we call the \emph{Bowen--Margulis--Sullivan measure} associated to $(\mu_{\phi}, \mu_{\phi \circ \opp})$.

Combining the ergodicity results for the diagonal action of $\Gamma$ on $\Fc_{\theta} \times \Fc_{\opp^* \theta}$ established in~\cite{CZZ2024} with our lifting theorem, we obtain the ergodic dichotomy for $\nu_{\phi}$ and $\mathsf{m}_{\phi}$, which generalizes the classical Hopf--Tsuji--Sullivan dichotomy.

\begin{theorem} \label{thm:ergodic dichotomy}
The following dichotomy holds.
\begin{enumerate}
   \item If $\sum_{\ga \in \Ga} e^{-\delta \phi(\kappa(\ga))} = + \infty$, then $\delta = \delta^{\phi}(\Ga)$ and $(\mu_{\phi}, \mu_{\opp^*\phi })$ is a unique pair of PS-measures, and we have 
   \begin{itemize}
      \item $\mu_{\phi}(\La_{\Delta}^{\rm concon}(\Ga)) = \mu_{\opp^*\phi }(\La_{\Delta}^{\rm concon}(\Ga)) = 1$,
      \item the diagonal $\Ga$-action on $(\Fc_{\Delta}^{(2)}, \nu_{\phi})$ is ergodic, and
      \item the $\Asf$-action on $(\Ga \ba \Gsf / \Msf, \mathsf{m}_{\phi})$ is ergodic.
   \end{itemize}
   \item If $\sum_{\ga \in \Ga} e^{-\delta \phi(\kappa(\ga))} < + \infty$, then we have 
   \begin{itemize}
      \item $\mu_{\phi}(\La_{\Delta}^{\rm con}(\Ga)) = \mu_{\opp^*\phi }(\La_{\Delta}^{\rm con}(\Ga)) = 0$,
      \item the diagonal $\Ga$-action on $(\Fc_{\Delta}^{(2)}, \nu_{\phi})$ is not ergodic, and
      \item the $\Asf$-action on $(\Ga \ba \Gsf / \Msf, \mathsf{m}_{\phi})$ is not ergodic.
   \end{itemize}
\end{enumerate}
\end{theorem}

\begin{proof}
Suppose first that $\sum_{\ga \in \Ga} e^{-\delta \phi(\kappa(\ga))} = + \infty$. The existence of $(\mu_{\phi}, \mu_{\opp^*\phi})$ implies $\delta \ge \delta^{\phi}(\Ga)$ by Theorem \ref{thm:lifting properties}(4), and hence $\delta = \delta^{\phi}(\Ga)$ due to the divergence of the series.

Without loss of generality, we may assume $\theta = \opp^* \theta$ (Observation \ref{obs:symmetry theta}).
Now by Theorem \ref{thm:lifting properties}, we have $\mu_\phi = f_* \mu$ and $\mu_{\opp^*\phi} = \bar f_* \bar \mu$ where 
\begin{itemize}
   \item $\mu$ is the $(\Ga, \phi, \delta)$-Patterson--Sullivan measure on $\La_{\theta}(\Ga)$ and $f :  \La_{\theta}(\Ga) \to \Fc_{\Delta}$ is the $\mu$-a.e.\ defined injective $\Ga$-equivariant map, 
      \item $\bar\mu$ is the $(\Ga, \opp^*\phi, \delta)$-Patterson--Sullivan measure $\La_{\theta}(\Ga)$ and $\bar f :  \La_{\theta}(\Ga) \to \Fc_{\Delta}$ is the $\bar\mu$-a.e.\ defined injective $\Ga$-equivariant map.
\end{itemize}
Theorem \ref{thm:lifting properties} also implies that 
$$
\mu_{\phi}(\La_{\Delta}^{\rm concon}(\Ga)) = \mu_{\opp^*\phi}(\La_{\Delta}^{\rm concon}(\Ga)) = 1.
$$

As proved in \cite[Corollary 12.1]{CZZ2024}, the diagonal $\Ga$-action on $(\La_{\theta}(\Ga)^2, \mu \otimes \bar\mu)$ is ergodic.  Hence, the $\Ga$-action on $(\Fc_{\Delta}^2, \mu_{\phi} \otimes \mu_{\opp^*\phi })$ is ergodic as well. Since $\Fc_{\Delta}^{(2)} \subset \Fc_{\Delta}^{2}$ is $\Ga$-invariant, the $\Ga$-action on $(\Fc_{\Delta}^{(2)}, \nu_{\phi})$ is ergodic as well. It then follows from the definition of $\mathsf{m}_{\phi}$ that the $\Asf$-action on $(\Ga \ba \Gsf / \Msf, \mathsf{m}_{\phi})$ is also ergodic.

\medskip

Next suppose that $\sum_{\ga \in \Ga} e^{-\delta \phi(\kappa(\ga))} < + \infty$. The Shadow Lemma (Proposition \ref{prop.shadowlemma}) implies that $\mu_{\phi}(\La_{\Delta}^{\rm con}(\Ga)) = \mu_{\opp^*\phi }(\La_{\Delta}^{\rm con}(\Ga)) = 0$. It now suffices to show that the $\Ga$-action on $(\Fc_{\Delta}^{(2)}, \nu_{\phi})$ is not ergodic, which also implies the non-ergodicity of the $\Asf$-action on $(\Ga \ba \Gsf / \Msf, \mathsf{m}_{\phi})$.

Suppose to the contrary that the $\Ga$-action on $(\Fc_{\Delta}^{(2)}, \nu_{\phi})$ is ergodic. Then passing through the projection $\pi : \Fc_{\Delta}^{(2)} \to \Fc_{\theta}^{(2)}$, the $\Ga$-action on $(\Fc_{\theta}^{(2)}, \pi_*\nu_{\phi})$ is ergodic, where $\Fc_{\theta}^{(2)}$ is the set of transverse pairs in $\Fc_{\theta}^2$. On the other hand, since $\pi_* \nu_{\phi}$ is supported on $\La_\theta(\Ga) \times \La_{\theta}(\Ga)$, such an ergodicity implies $\sum_{\ga \in \Ga} e^{-\delta \phi(\kappa(\ga))} = + \infty$ by \cite[Corollary 12.1]{CZZ2024}, contradiction.
\end{proof}


\section{Strict convexity of critical exponent} \label{section:strict convexity}


This section is devoted to the proof of our main theorem on the strict convexity of the entropy. We combine the strategy of \cite[Proposition 14.5]{BCZZ_coarse} with the machinery developed in this paper.

\begin{theorem} \label{thm:convexity body}

Suppose $\Ga < \Gsf$ is a $\Psf_{\theta}$-transverse group,  $\Theta \supset \theta$, and $\Phi_\alpha(\Gamma) < \SL(V_\alpha)$ has semisimple Zariski closure for each $\alpha \in \Theta$. 
Assume:
\begin{itemize}
   \item $\phi \in \fa_{\theta}^*$ satisfies $\delta^{\phi}(\Ga) < + \infty$ and $\sum_{\ga \in \Ga} e^{-\delta^{\phi}(\Ga) \phi(\kappa(\ga))} = + \infty$.
   \item $\phi_1, \phi_2 \in \fa_{\Theta}^*$ satisfy $\delta^{\phi_1}(\Ga) = \delta^{\phi_2}(\Ga) = 1$.
   \item  There exist $t \in (0, 1)$  and $C \ge 0$ such that 
   $$
   \phi(\kappa(\gamma)) \ge t \phi_1(\kappa(\gamma)) + (1 - t) \phi_2(\kappa(\gamma)) -C \quad \text{for all } \ga \in \Ga.
   $$
\end{itemize}
Then $\delta^{\phi}(\Ga) \le 1$ and equality holds if and only if 
   $$\sup_{\ga \in \Ga} \abs{ \phi(\kappa(\ga))- \phi_i(\kappa(\ga))} < + \infty \quad \text{for } i = 1, 2.$$
   In particular, if $\delta^{\phi}(\Ga) = 1$, then
   $$
   \phi(\lambda(\ga)) = \phi_1(\lambda(\ga)) = \phi_2(\lambda(\ga)) \quad \text{for all } \ga \in \Ga.
   $$
\end{theorem}

Delaying the proof of Theorem \ref{thm:convexity body}, we first deduce an analogous statement for general linear combinations.

\begin{corollary} \label{cor:convexity general lin comb}
  Suppose $\Ga < \Gsf$ is a $\Psf_{\theta}$-transverse group,  $\Theta \supset \theta$, and $\Phi_\alpha(\Gamma) < \SL(V_\alpha)$ has semisimple Zariski closure for each $\alpha \in \Theta$.  
Assume:
\begin{itemize}
   \item $\phi \in \fa_{\theta}^*$ satisfies $\delta^{\phi}(\Ga) < + \infty$ and $\sum_{\ga \in \Ga} e^{-\delta^{\phi}(\Ga) \phi(\kappa(\ga))} = + \infty$.
   \item $\phi_1, \phi_2 \in \fa_{\Theta}^*$ satisfy  $\delta^{\phi_1}(\Ga), \delta^{\phi_2}(\Ga) < + \infty$ and 
   $$
   \phi \ge c_1 \phi_1 + c_2 \phi_2 \quad \text{on} \quad \mfa^+
   $$
 for some $c_1, c_2 > 0$.

\end{itemize}
Then 
$$
\delta^{\phi}(\Ga) \le \frac{1}{\frac{c_1}{\delta^{\phi_1}(\Ga)} + \frac{c_2}{\delta^{\phi_2}(\Ga)} }
$$
and  equality holds if and only if
$$
\sup_{\ga \in \Ga} \abs{ \delta^{\phi}(\Ga) \phi(\kappa(\ga)) - \delta^{\phi_i}(\Ga) \phi_i(\kappa(\ga))} < + \infty \quad \text{for} \quad i = 1, 2.
$$
\end{corollary}

\begin{proof}
Let 
$$\psi := \frac{1}{\frac{c_1}{\delta^{\phi_1}(\Ga)} + \frac{c_2}{\delta^{\phi_2}(\Ga)}} \cdot \phi \in \fa_{\theta}^*.$$
Then we have
$$
\delta^{\psi}(\Ga) = \left( \frac{c_1}{\delta^{\phi_1}(\Ga)} + \frac{c_2}{\delta^{\phi_2}(\Ga)} \right) \delta^{\phi}(\Ga) \quad \text{and} \quad \sum_{\ga \in \Ga} e^{-\delta^{\psi}(\Ga) \psi(\kappa(\ga))} = + \infty.
$$
In addition, setting $\psi_i := \delta^{\phi_i}(\Ga) \cdot \phi_i$, we have $\delta^{\psi_i}(\Ga) = 1$ for $i = 1, 2$, and moreover
$$
\psi \ge t \psi_1 + (1 - t) \psi_2
$$
where $t := \frac{ \frac{c_1}{\delta^{\phi_1}(\Ga)}}{ \frac{c_1}{\delta^{\phi_1}(\Ga)} + \frac{c_2}{\delta^{\phi_2}(\Ga)} } \in (0, 1)$. Then applying Theorem \ref{thm:convexity body} to $\psi$, $\psi_1$, and $\psi_2$ finishes the proof.
\end{proof}

\subsection{Proof of Theorem \ref{thm:convexity body} in a special case} \label{subsection:proof of convexity} We first prove the theorem with an extra assumption. 

\medskip
\noindent \textbf{Extra Assumption:}  $\Ga$ is $\Psf_{\Theta}$-contracting and strongly $(\Phi_{\alpha})_{\alpha \in \Theta}$-irreducible. 
\medskip

Notice that, if $s > 0$, then H\"older inequality implies that 
$$\begin{aligned}
\sum_{\ga \in \Ga} e^{-s \phi(\kappa(\ga))} &  \le e^{Cs} \sum_{\ga \in \Ga} e^{-s t\phi_1(\kappa(\ga))} e^{-s (1-t)\phi_2(\kappa(\ga))} \\
& \le  e^{Cs} \left( \sum_{\ga \in \Ga} e^{-s \phi_1(\kappa(\ga))} \right)^t \left( \sum_{\ga \in \Ga} e^{-s \phi_2(\kappa(\ga))} \right)^{1-t}.
\end{aligned}
$$
So Equation~\eqref{eqn:series defn of critical exp} implies that 
$$
\delta^{\phi}(\Ga) \le 1.
$$

In addition, it is straightforward that $ \sup_{\ga \in \Ga} \abs{ \phi(\kappa(\ga))- \phi_i(\kappa(\ga))} < + \infty$ for $i = 1, 2$ implies  $\delta^{\phi}(\Ga) = 1$, and therefore the theorem reduces to the following. 

\medskip

\noindent \textbf{Claim:} If $\delta^{\phi}(\Ga) = 1$, then $  \sup_{\ga \in \Ga} \abs{ \phi(\kappa(\ga))- \phi_i(\kappa(\ga))} < + \infty$ for $i = 1, 2$.

\medskip

Assume $\delta^\phi(\Gamma)=1$. Let $\mu$ be the unique $(\Ga, \phi, 1)$-Patterson--Sullivan measure on $\La_{\theta}(\Ga)$ (see Theorem~\ref{thm:CZZ PS transverse}) and let $f : \Lambda_\theta(\Gamma) \rightarrow \Fc_{\Theta}$ denote the map in Theorem~\ref{thm:lifting properties}. Then $f_* \mu$ is a  $(\Ga, \phi, 1)$-Patterson--Sullivan measure on $\Fc_\Theta \subset \partial_\Theta X$. Next for $i=1,2$ let  $\mu_i$ be a $(\Gamma,\phi_i,1)$-Patterson--Sullivan measure on $\partial_{\Theta} X$ (see Proposition \ref{prop:PS meas exists}). 

By Theorems \ref{thm:Zdense PS} and \ref{thm:lifting properties}, the measures $f_*\mu$, $\mu_1$, and $\mu_2$ are each part of a Patterson--Sullivan system on $\partial_\Theta X$. In particular, these measures satisfy the Shadow Lemma (Proposition \ref{prop.shadowlemma}). Then we have that for any large $R > 0$, there exists $C_0 = C_0(R) > 1$ such that for all $\ga \in \Ga$ and $i = 1,2$,
$$
(f_* \mu)(\Oc_R^{\Theta}(\ga)) \le C_0 e^{ - \phi (\kappa(\ga))} \quad \text{and} \quad \mu_i ( \Oc_R^{\Theta}(\ga)) \ge C_0^{-1} e^{- \phi_i(\kappa(\ga))}.
$$
Hence, it follows from $\phi(\kappa(\ga)) \ge t \phi_1(\kappa(\ga)) + (1-t) \phi_2(\kappa(\ga)) - C$ that for all $\ga \in \Ga$,
$$\begin{aligned}
(f_* \mu)(\Oc_R^{\Theta}(\ga)) & \le C_0^2 e^{C} \mu_1(\Oc_R^{\Theta}(\ga))^{t} \mu_2(\Oc_R^{\Theta}(\ga))^{1-t} \\
& \le C_0^2 e^{C} \left( t \mu_1(\Oc_R^{\Theta}(\ga)) + (1-t) \mu_2(\Oc_R^{\Theta}(\ga)) \right) \\
& \le C_0^2 e^{C} (\mu_1 +  \mu_2 )(\Oc_R^{\Theta}(\ga))
\end{aligned}
$$
where the weighted arithmetic-mean geometric-mean inequality is used in the second inequality.

By Proposition \ref{prop:uniquePS}, we then have
$$
f_* \mu \ll \mu_1 + \mu_2.
$$
In particular, at least one of $\mu_1$ or $\mu_2$ is non-singular to $f_* \mu$. After relabeling, we can suppose that $\mu_1$ is non-singular to $f_* \mu$. Then by Theorem \ref{thm:KZ Rigidity}, we have
$$
\sup_{\ga \in \Ga} \abs{ \phi(\kappa(\ga)) - \phi_1(\kappa(\ga))} < + \infty.
$$
This implies that  there there exists $C' > 0$ such that 
   $$
   \phi(\kappa(\gamma)) \ge \phi_2(\kappa(\gamma)) -C'
   $$
    for all $\gamma \in \Gamma$. Then using the Shadow Lemma and Proposition \ref{prop:uniquePS} as above, we have 
$$
f_* \mu \ll \mu_2.
$$
Therefore by Theorem \ref{thm:KZ Rigidity},
$$
\sup_{\ga \in \Ga} \abs{ \phi(\kappa(\ga)) - \phi_2(\kappa(\ga))} < + \infty.
$$
This completes the proof.
\qed

\subsection{Proof of Theorem \ref{thm:convexity body} in general} We now prove the theorem in full generality. The idea is to replace $\Gsf$ with a different group where $\Gamma$ is contracting and strongly irreducible.

We will freely use the notation in Section~\ref{sec:sl notation}. For each $\alpha \in \Theta$, let $\Hsf_\alpha < \SL(V_\alpha)$ denote the Zariski closure of $\Phi_\alpha(\Gamma) < \SL(V_{\alpha})$. Replacing $\Gamma$ with a finite index subgroup, we can assume that each $\Hsf_\alpha$ is connected. 

\begin{lemma} \label{lem:irr prox reps}
    For each $\alpha \in \Theta$ there exists an irreducible representation $\rho_\alpha : \Hsf_\alpha \rightarrow \SL(W_\alpha)$  such that 
\begin{enumerate}
\item There exist $C_\alpha > 0$ and $r_\alpha \in \Nb$ where
$$
\abs{\omega_1( \kappa(\rho_\alpha( h))) -r_\alpha \omega_1(\kappa(h))} \leq C_{\alpha}
$$
for all $h \in \Hsf_\alpha$.
\item $\rho_{\alpha}(\Hsf_{\alpha})$ is $\Psf_{\alpha_1}$-contracting. 
\end{enumerate}
When $\alpha \in \theta$, we can assume that  $r_{\alpha} =1$ and $W_{\alpha} < V_{\alpha}$ is an irreducible factor of the $\Hsf_\alpha$-action on $V_\alpha$.
\end{lemma} 

\begin{remark} In the lemma, we assume that $W_\alpha$ is endowed with some inner product (the properties do not depend on the choice). \end{remark}

\begin{proof} Fix $\alpha \in \Theta$, let $d = \dim V_\alpha$, and identify $\SL(V_\alpha)$ with $\SL(d, \Rb)$ using an orthogonal basis of $V_\alpha$. 
   Conjugating $\Hsf_\alpha$ (this is what introduces the additive error), we can assume that the Cartan subgroup $\Asf_\alpha$ of $\Hsf_\alpha$ consists of positive diagonal matrices and the maximal compact subgroup of $\Hsf_\alpha$ is a subgroup of $\mathsf{SO}(d)$. 

Fix $r \in \Nb$ minimally so that $\alpha_{r}( \kappa( a)) > 0$ for some $a \in \Asf_\alpha$. Consider the standard representation $\tau : \SL(d, \Rb) \rightarrow \SL(\wedge^r \Rb^d)$, and identify $\SL(\wedge^r \Rb^d)$ with $\SL( {d \choose r}, \Rb)$. Notice that if $a\in \Asf_\alpha$, then the $r$ largest entries along the diagonal in $a$ are equal to  $e^{\omega_1(\kappa(a))}$. Hence $\omega_1( \kappa( \tau(a))) = r \omega_1(\kappa(a))$. Thus $\omega_1( \kappa( \tau(h))) = r \omega_1(\kappa(h))$ for all $h \in \Hsf_\alpha$. Further, by the choice of $r$, there exists $a \in \Asf_\alpha$ so that 
$$
\alpha_{1}( \kappa( \tau(a))) = \alpha_r (\kappa(a)) > 0
$$
and hence 
$$
\alpha_{1}( \kappa( \tau(a^n))) \rightarrow +\infty.
$$

 Finally, since $\Hsf_\alpha$ is semisimple, we can decompose $\wedge^r \Rb^d$ into $\tau|_{\Hsf_\alpha}$-irreducible factors and then pick the factor $W_\alpha \subset \wedge^r \Rb^d$ where the associated representation $\rho : \Hsf_\alpha \rightarrow \SL(W_\alpha)$ satisfies 
$$
\omega_1( \kappa( \rho(h)))  = \omega_1( \kappa( \tau(h))) 
$$
for all $h \in \Hsf_\alpha$. Then $\rho_\alpha : = \rho$ and $r_\alpha : = r$ have the desired properties. 

Notice that when $\alpha \in \theta$, then $r=1$ by the divergence property of transverse groups, from which the last claim follows.
\end{proof} 

Let $\bar \Gsf : = \prod_{\alpha \in \Theta} \SL(W_\alpha)$ and $\rho : =  ( \rho_\alpha \circ \Phi_\alpha|_{\Gamma})_{\alpha \in \Theta} : \Gamma \rightarrow \bar \Gsf$ using the representations in Lemma \ref{lem:irr prox reps}, and let
$$
\bar \Gamma : = \rho(\Gamma). 
$$
We can assume that the simple roots of $\bar \Gsf$ are $\cup_{\alpha \in \Theta} \{ \beta_j^{\alpha}\}_{j=1}^{d_{\alpha}-1}$ where $\beta_1^{\alpha}, \dots, \beta_{d_\alpha-1}^{\alpha}$ are the standard simple roots of $\SL(W_{\alpha}) = \SL(d_{\alpha}, \Rb)$ described in Section~\ref{sec:sl notation}.  Notice that we can choose the irreducible representations $\Phi_{\beta_1^\alpha}$ for $\bar \Gsf$ to just be projection onto the associated factor. Then let 
$$
\bar \Theta : = \{ \beta^\alpha_1\}_{\alpha \in \Theta} \quad \text{and} \quad \bar \theta : = \{ \beta^\alpha_1\}_{\alpha \in \theta}.
$$

\begin{lemma}  The group $\bar\Ga$ is $\Psf_{\bar \Theta}$-contracting, strongly $(\Phi_{\beta})_{\beta \in \bar\Theta}$-irreducible, and $\bar \Gamma$ is $\Psf_{\bar \theta}$-transverse.
\end{lemma} 

\begin{proof} 
   
   First notice that the strong $(\Phi_{\beta})_{\beta \in \bar\Theta}$-irreducibility follows from the irreducibility of $\rho_{\alpha}$ and that $\Hsf_{\alpha}$ is the Zariski closure of $\Phi_{\alpha}(\Ga)$. Moreover, since $\Hsf_{\alpha}$ is the identity component of the Zariski closure of $\Phi_{\alpha}(\Ga)$, it follows from Lemma \ref{lem:irr prox reps}(2) and \cite[Lemma 6.23]{BQ_book} that for each $\beta \in \bar\Theta$, there exists a sequence $\{\ga_n\} \subset \Ga$ satisfying
   $$
   \beta(\kappa(\rho(\ga_n))) \to + \infty.
   $$
   Now by \cite[Lemma 6.25]{BQ_book}, the sequence $\{\ga_n\} \subset \Ga$ can be chosen independent of the choice of $\beta$, i.e., there exists a sequence $\{\ga_n\} \subset \Ga$ such that for each $\beta \in \bar\Theta$,
   $$
   \beta(\kappa(\rho(\ga_n))) \to + \infty.
   $$
   Hence, $\bar\Ga$ is $\Psf_{\bar\Theta}$-contracting.

   It remains to prove that $\bar \Ga$ is $\Psf_{\bar \theta}$-transverse. For $\alpha \in \theta$, note that $W_{\alpha}$ can be chosen as an irreducible factor of the $\Hsf_{\alpha}$-action on $V_{\alpha}$   and $r_{\alpha} = 1$ in Lemma \ref{lem:irr prox reps}.  We then have 
   $$
   \beta(\kappa(\rho(\ga_n))) \to + \infty
   $$
   for all $\beta \in \bar\theta$ and an escaping sequence $\{\ga_n\} \subset \Ga$, since $\Ga$ is $\Psf_{\theta}$-transverse.
   In addition, we also have that $\bar \Gsf / \Psf_{\bar \theta} = \prod_{\alpha \in \theta} \Pb(W_{\alpha})$, which is a subspace of $\prod_{\alpha \in \theta} \Pb(V_{\alpha})$.  Since the limit set of $(\Phi_{\alpha})_{\alpha \in \theta}(\Ga) < \prod_{\alpha \in \theta} \SL(V_{\alpha})$ is transverse by Property \ref{item:boundary maps of Tits repn}, this implies that the limit set of $\bar \Gamma$ in $\bar \Gsf / \Psf_{\bar \theta}$ is transverse as well.
\end{proof} 

 Since $\phi \in \mfa_\theta^* = {\rm span}\{ \omega_\alpha\}_{\alpha \in \theta}$, we can write $\phi = \sum_{\alpha \in \theta} c_\alpha \omega_\alpha $. Let 
$$
\bar \phi : = \sum_{\alpha \in \theta} \frac{c_\alpha}{r_\alpha} \omega_{\beta^\alpha_1} \in \mfa_{\bar \theta}^*.
$$
Then there exists $C > 0$ such that 
$$
\abs{ \phi(\kappa(\gamma)) - \bar{\phi}(\kappa(\rho(\gamma)))} \leq C \quad \text{for all } \gamma \in \Gamma. 
$$
Likewise, we can define $\bar \phi_1, \bar \phi_2 \in \mfa_{\bar \Theta}^*$ corresponding to $\phi_1, \phi_2$, respectively. Then after increasing $C > 0$ we can assume that 
$$
\abs{ \phi_i(\kappa(\gamma)) - \bar{\phi}_i(\kappa(\rho(\gamma)))} \leq C \quad \text{for all } \gamma \in \Gamma. 
$$
Finally, applying the special case of the theorem to $\bar \Gamma$ and $\bar \phi, \bar \phi_1, \bar \phi_2$ finishes the proof. 
\qed


\section{Lipschitz limit sets} \label{section:Lipschitz}
 

In this section, we prove the following entropy rigidity result for Anosov groups with Lipschitz limit sets.

\begin{theorem}\label{thm:Lipschitz limit set rigidity}

   Suppose $\Gamma < \SL(d,\Rb)$ is a $\Psf_1$-Anosov group acting strongly irreducibly on $\Rb^d$ and on $\wedge^{p+1} \Rb^d$ whose limit set $\Lambda_1(\Gamma)$ is a Lipschitz $p$-manifold, for some $p \le d-2$. Then
$$
\delta^{\phi_{\rm H}}(\Gamma) = p,
$$
if and only if $\Gamma$ is conjugate to a uniform lattice in $\mathsf{SO}(d-1, 1)$ and $p=d-2$. 
\end{theorem} 

We will freely use the notation in Section~\ref{sec:sl notation}. In addition, notice that the Jordan projection $\lambda : \SL(d,\Rb)  \rightarrow \mfa^+$ is given by 
$$
\lambda(g) = {\rm diag}(\log \lambda_1 (g), \dots, \log \lambda_d (g))
$$
where 
$$
\lambda_1(g) \geq \cdots \geq \lambda_d(g)
$$
are the absolute values of the eigenvalues of $g$. 
Also, we can choose the representation $\Phi_{\alpha_j}$ to be the standard irreducible representation $\SL(d,\Rb) \rightarrow \SL(\wedge^j \Rb^d)$. 

In terms of the fundamental weights, the Hilbert functional $\phi_{\rm H} \in \mfa^*$ satisfies 
$$
\phi_{\rm H}= \frac{1}{2}( \omega_1+\omega_{d-1}).
$$

Let $\phi_p : = (p+1) \omega_1 - \omega_{p+1}$ and $\bar \phi_p : = (p+1) \omega_{d-1} - \omega_{d-p-1}$. Equation~\eqref{eqn:opp inv and cartan proj} implies that  
$$
 \bar\phi_p(\kappa(g)) = \phi_p(\kappa(g^{-1}))
 $$
 for all $g \in \SL(d,\Rb)$, so we have 
$$
\delta^{\bar \phi_p}(\Gamma) = \delta^{ \phi_p}(\Gamma) 
$$
for every subgroup $\Gamma < \SL(d,\Rb)$. 

We use the following result of Pozzetti--Sambarino--Wienhard. 

\begin{theorem}[{\cite[Theorem A]{PSW2023}}]\label{thm:PSW}

   Suppose $\Gamma < \SL(d,\Rb)$ is a $\Psf_1$-Anosov group acting strongly irreducibly on $\Rb^d$ and  on $\wedge^{p+1} \Rb^d$ for some $p \le d-2$. If $\Lambda_1(\Gamma)$ is a Lipschitz $p$-manifold, then 
$$
\delta^{\phi_p}(\Gamma) = \delta^{\bar \phi_p}(\Gamma) =1.
$$
\end{theorem}

\begin{remark} Theorem A in~\cite{PSW2023} does not include the assumption that $\Gamma$ acts strongly irreducibly on $\wedge^{p+1} \Rb^d$, however the proof of~\cite[Lemma 6.8]{PSW2023} uses ~\cite[Proposition 10.3]{Lab2006} which appears to be false. Assuming this irreduciblity allows one to avoid this citation in the proof. 

\end{remark}  

The functions $\phi_{\rm H}$, $\phi_p$, and $\bar \phi_p$ have the following relation.

\begin{lemma}\label{lem:comparing functionals} If $X = {\rm diag}(t_1,\dots, t_d) \in \mfa^+$, then 
$$
p \phi_{\rm H}(X) \geq \frac{1}{2} (\phi_p + \bar \phi_p)(X)
$$
with equality if and only if $t_2 = \cdots = t_{d-1}$. 
\end{lemma} 

\begin{proof} Note 
\begin{align*}
\frac{1}{2} & (\phi_p + \bar \phi_p)(X) = \frac{p+1}{2}( t_1 - t_{d}) -\frac{1}{2} ( t_1 + \cdots + t_{p+1} - t_{d-p} - \cdots - t_d) \\
& =  \frac{p}{2}( t_1 - t_{d}) - \frac{1}{2}(t_2 - t_{d-1}) - \frac{1}{2} (t_3-t_{d-2}) - \cdots - \frac{1}{2} (t_r - t_{d-r+1})
\end{align*} 
where $r := \min\left\{ \lfloor d/2 \rfloor, p+1\right\}$. So 
$$
p \phi_{\rm H}(X) \geq \frac{1}{2} (\phi_p + \bar \phi_p)(X)
$$
with equality if and only if $t_2=t_{d-1}$, $t_3=t_{d-2}$ ,..., and $t_r = t_{d-r+1}$. Since $t_2 \geq \cdots \geq t_{d-1}$ the result follows.  
\end{proof}

\subsection{Proof of Theorem~\ref{thm:Lipschitz limit set rigidity}} Suppose $\Gamma < \SL(d,\Rb)$ is a $\Psf_1$-Anosov group such that $\Lambda_1(\Gamma)$ is a Lipschitz $p$-manifold and $\Gamma$ acts strongly irreducibly on $\Rb^d$ and  on $\wedge^{p+1} \Rb^d$. By  \cite[Corollary 2.20]{BCLS2015}, $\Ga$ has semisimple Zariski closure, and hence each $\Phi_{\alpha}(\Ga)$ has semisimple Zariski closure. This implies that  $\Ga$ satisfies the first line of Theorem \ref{thm:convexity body} with 
$$
\theta = \{\alpha_1, \alpha_{d-1}\} \quad \text{and} \quad \Theta = \{ \alpha_1, \alpha_{p+1}, \alpha_{d-p-1}, \alpha_{d-1}\}.
$$
Further, $\phi_{\rm H} \in \mfa_\theta^*$ and $\phi_p, \bar \phi_p \in \mfa_\Theta^*$.

Theorem~\ref{thm:PSW} implies that 
$$
\delta^{\phi_p}(\Gamma) = \delta^{\bar\phi_p}(\Gamma) =1.
$$
Then Theorem \ref{thm:convexity body} and Lemma~\ref{lem:comparing functionals} imply that 
$$
 \delta^{\phi_{{\rm H}} }(\Gamma)= p \delta^{p \phi_{{\rm H}} }(\Gamma) \leq p \cdot 1=p
$$
with equality if and only if 
\begin{equation}\label{eqn:equality of length}
\phi_p(\lambda(\gamma)) =  \bar\phi_p(\lambda(\gamma))=p \phi_{{\rm H}}(\lambda(\gamma))
\end{equation} 
for all $\gamma \in \Gamma$.

The backward direction of the theorem is clear, so suppose for the rest of the section that $ \delta^{\phi_{{\rm H}} }(\Gamma) = p$.  Then Lemma~\ref{lem:comparing functionals} and Equation~\eqref{eqn:equality of length} imply that 
\begin{equation}\label{eqn:equality of middle eigenvalues}
\lambda_2(\gamma) = \cdots = \lambda_{d-1}(\gamma)
\end{equation} 
for all $\gamma \in \Gamma$. 

\begin{lemma} $\lambda_2(\gamma) = \cdots = \lambda_{d-1}(\gamma)=1$ for all $\gamma \in \Gamma$. \end{lemma} 

\begin{proof} Fix $\gamma \in \Gamma$. Then Equations~\eqref{eqn:equality of length} and~\eqref{eqn:equality of middle eigenvalues} imply that 
$$
p \log \lambda_1(\gamma) -p\log \lambda_2(\gamma) = \frac{p}{2} ( \log \lambda_1(\gamma) - \log \lambda_d(\gamma)).
$$
So 
$$
 \log \lambda_2(\gamma)  = \frac{1}{2} ( \log \lambda_1(\gamma) + \log \lambda_d(\gamma)).
$$
On the other hand 
$$
\log \lambda_1(\gamma) + (d-2) \log \lambda_2(\gamma) + \log \lambda_d(\gamma) = 0. 
$$
So 
$$
\frac{d}{2} \left( \log \lambda_1(\gamma) + \log \lambda_d(\gamma) \right)= 0, 
$$
which implies that $\lambda_d(\gamma) = \lambda_1(\gamma)^{-1}$ and hence 
\begin{equation*}
\lambda_2(\gamma) = \cdots = \lambda_{d-1}(\gamma)=1. \qedhere
\end{equation*}
\end{proof} 

To finish the proof we use the following observation, which is a consequence of a theorem of Benoist and the basic theory of irreducible representations of semisimple Lie groups. 
 
\begin{observation}\label{obs:eigenvalue implies conjugacy} If $d \geq 3$, $G < \SL(d,\Rb)$ is a strongly irreducible $\Psf_1$-contracting  subgroup, and 
$$
\lambda_2(g) = \cdots = \lambda_{d-1}(g) = 1
$$
for all $g \in G$, then $G$ is conjugate to a Zariski dense subgroup of  $\mathsf{SO}_0(d-1,1)$ or  $\mathsf{SO}(d-1,1)$.

\end{observation} 

Delaying the proof of the observation, we finish the proof of the theorem. Observation~\ref{obs:eigenvalue implies conjugacy} implies that $\Gamma$ is conjugate into $\mathsf{SO}(d-1,1)$. Next consider the Klein-Beltrami model $\Hb^{d-1} \subset \Pb(\Rb^d)$ of real hyperbolic $(d-1)$-space. Then $\mathsf{SO}(d-1,1) \rightarrow \Isom(\Hb^d)$ is a finite cover and the limit set $\Lambda_1(\Gamma) \subset \Pb(\Rb^d)$ coincides with the hyperbolic limit set of the image of $\Gamma$ in $\Isom(\Hb^d)$. Then, since $\Lambda_1(\Gamma)$ is a Lipschitz $p$-manifold, ~\cite[Theorem 1.3]{Kapovich2009} implies that $\Gamma$ preserves and acts cocompactly on  a totally geodesic copy of $\Hb^{p+1}$ inside $\Hb^{d-1}$. Since $\Gamma$ is strongly irreducible, we must have $\Hb^{p+1}=\Hb^{d-1}$. Hence $p=d-2$ and $\Gamma$ is a uniform lattice in $\mathsf{SO}(d-1,1)$. 
\qed

\subsection{Proof of Observation~\ref{obs:eigenvalue implies conjugacy}} Let $G'$ be the image of $G$ in $\PSL(d,\Rb)$, let $\mathsf{H}$ denote the Zariski closure of $G'$ in $\PSL(d,\Rb)$, and let $\mathsf{H}^0 < \mathsf{H}$ denote the connected component of the identity. By~\cite[Lemma 2.18]{BCLS2015}, $\mathsf{H}^0$ is a connected semisimple Lie group with trivial center and no compact factors. By a theorem of Benoist~\cite{Benoist_properties}, 
$$
\lambda_2(h) = \cdots = \lambda_{d-1}(h) = 1
$$
for all $h \in \mathsf{H}^0$. Thus $\mathsf{H}^0$ is a rank one non-compact simple group. Let $X$ be the symmetric space associated to $\mathsf{H}^0$ and let $\rho : \mathsf{H}^0 \rightarrow \Isom(X)$ be the induced map. Since $\mathsf{H}^0$ has trivial center, $\rho$ induces an isomorphism between $\mathsf{H}^0$ and $\Isom_0(X)$, the connected component of the identity in $\Isom(X)$. Further, $X$ is a negatively curved symmetric space, the geodesic boundary has an $\Isom(X)$-invariant smooth structure, and there exists a $\rho^{-1}$-equivariant smooth embedding $\xi : \partial X \hookrightarrow \Pb(\Rb^d)$ of the boundary of $X$ (for details about the construction of $\xi$, see for instance~\cite[Section 4]{ZhuZimmerII}). 

\begin{lemma} $X=\Hb^m$ is the real hyperbolic $m$-space, $m = \dim X$. \end{lemma} 

\begin{proof} Suppose $\gamma \in \Isom(X)$ is loxodromic, i.e. $\gamma$ has no fixed points in $X$ and has two fixed points $x^\pm$ in $\partial X$. Then the eigenvalue condition implies that all eigenvalues of the derivative $d(\gamma)_{x^\pm} : T_{x^\pm} \partial X \rightarrow T_{x^\pm} \partial X$ have the same modulus. From the description of the negatively curved symmetric spaces in~\cite[Chapther 19]{Mostow_book}, this is only possible if $X$ is a real hyperbolic space. 
\end{proof} 

Now we can identify $\Isom(X)$ with $\PO(m,1)$ and view $\rho^{-1}$ as an  irreducible representation of $\PO_0(m,1)$, the connected component of the identity in $\PO(m,1)$. It then follows from the eigenvalue condition and the theory of highest weights (see for instance~\cite[Lemma 10.4]{ZhangZimmer}) that $m=d-1$ and $\mathsf{H}^0 = \rho^{-1}(\PO_0(d-1,1))$ is conjugate to $\PO_0(d-1,1)$. So, after conjugating, we can assume that $\mathsf{H}^0  = \PO_0(d-1, 1)$.

Next let $\widehat{\Hsf}$ be the normalizer of $\PO_0(d-1,1)$ in $\PGL(d,\Rb)$ and let $\tau : \widehat{\Hsf} \rightarrow \Aut( \PO_0(d-1,1))$ be the map induced by conjugation.  By Schur's lemma, $\tau$ is injective. Further,  $\tau|_{\PO(d-1,1)}$ is onto. Hence $\mathsf{H} < \widehat{\Hsf} = \PO(d-1,1)$. This implies that  $G$ is  a Zariski dense subgroup of  $\mathsf{SO}_0(d-1,1)$ or  $\mathsf{SO}(d-1,1)$.
\qed


\section{Representations of mapping class groups and boundary maps} \label{section:MCG}


In this section, we apply our theory to representations of mapping class groups.
Let $S$ be a connected orientable surface of finite type with negative Euler characteristic. Its mapping class group $\Mod(S)$ is the group of isotopy classes of orientation-preserving homoeomorphisms on $S$.

Let $\PML$ denote the projective space of measured laminations on $S$, on which the mapping class group $\Mod(S)$ naturally acts. The space $\PML$ is homeomorphic to a sphere and admits a Lebesgue measure class. We show that representations of $\Mod(S)$ transfer this measure structure on $\PML$ to partial flag manifolds.

\begin{theorem} \label{thm:PML boundary map body}
   Suppose $\rho : \Mod(S) \to \Gsf$ is a representation such that $\rho(\Mod(S))$ is $\Psf_{\theta}$-contracting and strongly $(\Phi_{\alpha})_{\alpha \in \theta}$-irreducible. Then there exists a unique ${\rm Leb}$-a.e.\ defined $\rho$-equivariant map
   $$
   f : \PML \to \Fc_{\theta}.
   $$
   Moreover,  $f(x) \in \La_{\theta}^{\rm concon}(\rho(\Mod(S)))$ for ${\rm Leb}$-a.e.\ $x \in \PML$. 
\end{theorem}
 
\begin{proof}
We now deduce Theorem \ref{thm:PML boundary map body} from  Theorem \ref{thm:boundary map} and the universal amenability of the $\Mod(S)$-action on the boundary of the curve graph $\Cc(S)$ of $S$, shown by Hamenst\"adt \cite{Hamenstadt_amenable}.  The curve graph $\Cc(S)$ is the simplicial graph whose vertices are isotopy classes of essential simple closed curves on $S$, and two of them are connected by an edge if they have disjoint representatives. There is a natural action of $\Mod(S)$ on $\Cc(S)$ by isometries, and Masur and Minsky showed that $\Cc(S)$ is Gromov hyperblic \cite{MM_CC1}. In particular, the $\Mod(S)$-action on $\Cc(S)$ continuously extends to the Gromov boundary $\partial \Cc(S)$, and Hamenst\"adt proved that the $\Mod(S)$-action on $\partial \Cc(S)$ is amenable with resepect to any quasi-invariant Borel measure.

We now relate the above discussion to $\PML$. Klarreich \cite{Klarreich_boundary} and Hamenst\"adt \cite{Hamenstadt_train} characterized $\partial \Cc(S)$ as the space of (unmeasured) filling geodesic laminations on $S$. Hence, denoting by $\mathcal{FML} \subset \PML$ the projective space of filling measured laminations on $S$, there exists a measure-forgetting map
$$
\pi : \mathcal{FML} \to \partial \Cc(S)
$$
which is continuous \cite[Lemma 3.12]{hamenstadt2009invariant}. In particular, on the space $\UE \subset \PML$ of uniquely ergodic measured laminations, whose elements have unique transverse measures up to scaling, the map $\pi$ is injective.

By Masur \cite{masur1982interval} and Veech \cite{veech1982gauss}, $\UE$ is ${\rm Leb}$-conull. Therefore, the universal amenability of the $\Mod(S)$-action on $\partial \Cc(S)$ implies that the amenabiilty of the $\Mod(S)$-action on $(\UE, {\rm Leb})$.

In addition, $\UE \subset \PML$ is also embedded in the Gardiner--Masur boundary $\partial_{GM} \Tc(S)$ \cite{gardiner1991extremal,miyachi2013teichmuller} of the Teichm\"uller space $\Tc(S)$ of $S$. In our earlier work \cite[Theorem 10.1]{KimZimmer1}, we showed that the $\Mod(S)$-action on $\partial_{GM} \Tc(S)$ equipped with the pushforward of ${\rm Leb}$ under the embedding $\UE \hookrightarrow \partial_{GM} \Tc(S)$ is a well-behaved Patterson--Sullivan system whose conical limit set is conull. Therefore, we can apply Theorem \ref{thm:boundary map} and then Theorem~\ref{thm:PML boundary map body} follows, except for the ``moreover'' part.

To show the ``moreover'' part, we first note that the hierarchy involved in the associated well-behaved PS-system is not trivial (see \cite[Section 10]{KimZimmer1} for explicit description of shadows and hierarchy), and hence the part (3) of Theorem \ref{thm:boundary map} does not apply directly. Nevertheless, it was shown in \cite[Section 5.3]{Coulon_PS} that its shadows induced from the hierarchy satisfy the Kochen--Stone Lemma (Lemma \ref{lem: KS BC lemma}). Hence, in the proof of the part (3) of Theorem \ref{thm:boundary map},  one can replace shadows with mixed shadows and deduce the desired property.
\end{proof}


\section{Gromov hyperbolic spaces with exponentially bounded geometry} \label{section:GH exp bdd}


In this section, we apply our framework to a Gromov hyperbolic space  when it has exponentially bounded geometry.

\begin{definition} \label{defn:exp bdd}
   A proper geodesic Gromov hyperbolic space $(X, \dist_X)$ is said to have \emph{exponentially bounded geometry} if there exist $a, r > 0$ such that for any $x \in X$ and $R > 0$, the radius $R$-ball $B_X(x, R) \subset X$ centered at $x$ can contain at most $a^R$ number of pairwise disjoint balls of radius $r$.
\end{definition}

In the rest of this section, we fix a proper geodesic Gromov hyperbolic space $(X, \dist_X)$ with exponentially bounded geometry, and a discrete subgroup $\Ga < \Isom(X)$ of isometries, i.e., the $\Ga$-action on $X$ is proper. In this case, the critical exponent $\delta_{X}(\Ga)$ of the Poincar\'e series
$$
s \mapsto \sum_{\ga \in \Ga} e^{-s \dist_X(o, \ga o)}
$$
is finite \cite[Proposition 1.29]{Kaimanovich2004}, for any fixed basepoint $o \in X$. 

We denote by $\partial X$ the Gromov boundary of $X$ and by $\La(\Ga) \subset \partial X$ the limit set of $\Ga$, the set of accumulation points of $\Ga o \subset X$. We further assume that $\# \La(\Ga) \ge 3$, that is, $\Ga$ is non-elementary. We then have
$$
0 < \delta_{X}(\Ga) < + \infty,
$$
and that the $\Ga$-action on $\La(\Ga) \subset \partial X$ is a minimal convergence action. 

The \emph{Busemann coarse-cocycle} $\beta: \Ga \times \La(\Ga) \to \Rb$ is defined by 
$$
\beta(\ga, x) := \limsup_{p \to x} \dist_X(\ga^{-1} o, p) - \dist_X(o, p).
$$
In \cite{Coornaert_PS}, Coornaert constructed a coarse $(\Ga, \beta, \delta_{X}(\Ga))$-Patterson--Sullivan measure in the sense of Equation \eqref{eqn.coarse PS meaasure intro}. 

As an application of our theory, we also show that representations of $\Ga$ admit measurable boundary maps.

\begin{theorem} \label{thm:GH exp bdd}
Suppose $(X, \dist_X)$ and $\Ga$ are as above, $\sum_{\ga \in \Ga} e^{-\delta_{X}(\Ga) \dist_X(o, \ga o)} = + \infty$, and $\mu$ is a coarse $(\Ga, \beta, \delta_{X}(\Ga))$-PS measure.
If $\rho : \Ga \to \Gsf$ is a representation such that $\rho(\Ga)$ is $\Psf_{\theta}$-contracting and strongly $(\Phi_{\alpha})_{\alpha \in \theta}$-irreducible, then there exists a unique $\mu$-a.e.\ defined $\rho$-equivariant map
$$
f:\La(\Ga) \to \Fc_{\theta}.
$$
Moreover, $f(x) \in \La_{\theta}^{\rm concon}(\rho(\Ga))$ for $\mu$-a.e.\ $x \in \La(\Ga)$.
\end{theorem}

\begin{proof} It is easy to verify that this coarse PS-measure gives rise to a well-behaved Patterson--Sullivan system with respect to the trivial hierarchy $\mathscr{H}(R) \equiv \Ga$, see~\cite[Section 8]{ KimZimmer1}. 

It is also a straightforward application of the Borel--Cantelli Lemma to verify that the assumption $\sum_{\ga \in \Ga} e^{-\delta_{X}(\Ga) \dist_X(o, \ga o)} = + \infty$ implies that 
$$
\mu(\La^{\rm con}(\mathscr{H})) = 1,
$$
see for instance \cite[Proposition 7.1]{BCZZ_coarse}. Moreover, Kaimanovich showed that the $\Ga$-action on $\La(\Ga)$ is amenable with respect to any $\Ga$-quasi-invariant Borel measure \cite[Theorem 3.15]{Kaimanovich2004}. Therefore, Theorem \ref{thm:boundary map} applies.

\end{proof}

\section{Proving everything claimed in the introduction} \label{section:deductions}

In this last section, we explain why all of the the statements in the introduction are true. We first note that Zariski dense subgroups of $\Gsf$ are strongly $(\Phi_{\alpha})_{\alpha \in \Delta}$-irreducible (Remark \ref{remark:Zdense irreducible}) and $\Psf_{\Delta}$-contracting (see Section \ref{sec: existence of boundary maps}).

If $\Ga < \Gsf$ is a non-elementary  $\Psf_{\theta}$-Anosov group, then $\Ga$ is $\Psf_{\theta}$-transverse. Moreover in this case, $\La_{\theta}(\Ga) = \La_{\theta}^{\rm con} =  \La_{\theta}^{\rm concon}(\Ga)$. Hence, for any $\phi \in \fa_{\theta}^*$ with $\delta^{\phi}(\Ga) < + \infty$, the existence of a $(\Ga, \phi, \delta^{\phi}(\Ga))$-Patterson--Sullivan measure on $\La_{\theta}(\Ga)$ and the Shadow Lemma imply  $\sum_{\ga \in \Ga} e^{-\delta^{\phi}(\Ga) \phi(\kappa(\ga))} = + \infty$ \cite{Quint_PS,Sambarino_report}. 
\begin{itemize}
   \item Theorem \ref{thm:boundary map in intro} is a special case of Theorem \ref{thm:boundary map}.
   \item Theorem \ref{thm:Anosov existence uniqueness} follows from Theorem \ref{thm:lifting properties} and Proposition \ref{prop:embedding of partial flags}.
   \item 
Theorem \ref{thm:singular introduction} is a special case of Corollary \ref{cor:singularity between PS}.

\item 
Theorem \ref{thm:BMSergodic} is a special case of Theorem \ref{thm:ergodic dichotomy}, together with Equation \eqref{eqn:positive product meaasure}.

\item Theorem \ref{thm:extreme points} follows from Theorem \ref{thm:convexity body}.

\item 
Theorem \ref{thm:strict convexity intro} is a special case of Corollary \ref{cor:convexity general lin comb}.

\item 
Theorem \ref{thm:Lipschitz limit set rigidity intro} is proved as Theorem \ref{thm:Lipschitz limit set rigidity}.

\end{itemize}

\bibliographystyle{alpha}
\bibliography{geom}

\end{document}